\documentclass[11pt,reqno,sumlimits]{amsart}
\usepackage{amsfonts, amsmath, amscd, amssymb, euscript, amsthm, array, booktabs, dcolumn, shortvrb, tabularx, units, url, mathrsfs, tikz-cd, mathtools, kpfonts}
\usepackage[pdfstartview=FitH]{hyperref}
\usepackage[all]{xy}
\usetikzlibrary{arrows, decorations.markings}
\makeatletter
\newcommand{\addresseshere}{\enddoc@text\let\enddoc@text\relax}
\makeatother

\DeclareMathOperator{\bun}{Bun}

\DeclareMathOperator{\op}{op}
\DeclareMathOperator{\im}{Im}

\DeclareMathOperator{\CCS}{CCS}
\DeclareMathOperator{\Aut}{Aut}
\DeclareMathOperator{\todd}{Todd}

\DeclareMathOperator{\cok}{cok}
\DeclareMathOperator{\ind}{ind}
\DeclareMathOperator{\vol}{vol}
\DeclareMathOperator{\CS}{CS}
\DeclareMathOperator{\odd}{odd}

\DeclareMathOperator{\id}{id}
\DeclareMathOperator{\ch}{ch}
\DeclareMathOperator{\GL}{GL}
\DeclareMathOperator{\End}{End}
\DeclareMathOperator{\ho}{Hom}
\DeclareMathOperator{\rk}{rank}
\DeclareMathOperator{\str}{str}
\DeclareMathOperator{\tr}{tr}
\begin{document}
\theoremstyle{definition}
\newtheorem{coro}{Corollary}
\newtheorem{thm}{Theorem}
\newtheorem{defi}{Definition}
\newtheorem{lemma}{Lemma}
\newtheorem*{claim}{Claim}
\newtheorem{exam}{Example}
\newtheorem{prop}{Proposition}
\newtheorem{remark}{Remark}
\newtheorem{thmx}{Theorem}
\newtheorem{propx}{Proposition}
\renewcommand{\thethmx}{\Alph{thmx}}
\renewcommand{\thepropx}{\Alph{propx}}
\newcommand{\hto}{\hookrightarrow}
\newcommand{\wt}[1]{{\widetilde{#1}}}
\newcommand{\ov}[1]{{\overline{#1}}}
\newcommand{\un}[1]{{\underline{#1}}}
\newcommand{\wh}[1]{{\widehat{#1}}}
\newcommand{\br}[1]{{\breve{#1}}}
\newcommand{\wo}{\mathbin{\wh{\otimes}}}
\newcommand{\deff}[1]{{\bf\emph{#1}}}
\newcommand{\boo}[1]{\boldsymbol{#1}}
\newcommand{\abs}[1]{\lvert#1\rvert}
\newcommand{\norm}[1]{\lVert#1\rVert}
\newcommand{\inner}[1]{\langle#1\rangle}
\newcommand{\poisson}[1]{\{#1\}}
\newcommand{\biginner}[1]{\Big\langle#1\Big\rangle}
\newcommand{\set}[1]{\{#1\}}
\newcommand{\Bigset}[1]{\Big\{#1\Big\}}
\newcommand{\BBigset}[1]{\bigg\{#1\bigg\}}
\newcommand{\dis}[1]{$\displaystyle#1$}
\newcommand{\KER}{\mathbf{Ker}}
\newcommand{\EE}{\mathbf{E}}
\newcommand{\FF}{\mathbf{F}}
\newcommand{\HH}{\mathbf{H}}
\newcommand{\KK}{\mathbf{K}}
\newcommand{\LL}{\mathbf{L}}
\newcommand{\VV}{\mathbf{V}}
\newcommand{\WW}{\mathbf{W}}
\newcommand{\CC}{\mathbf{C}}
\newcommand{\R}{\mathbb{R}}
\newcommand{\N}{\mathbb{N}}
\newcommand{\Z}{\mathbb{Z}}
\newcommand{\Q}{\mathbb{Q}}
\newcommand{\C}{\mathbb{C}}
\newcommand{\aaa}{\mathbb{A}}
\newcommand{\bbb}{\mathbb{B}}
\newcommand{\ttt}{\mathbb{T}}
\newcommand{\sss}{\mathbb{S}}
\newcommand{\E}{\mathcal{E}}
\newcommand{\T}{\mathcal{T}}
\newcommand{\G}{\mathcal{G}}
\newcommand{\F}{\mathcal{F}}
\newcommand{\I}{\mathcal{I}}
\newcommand{\V}{\mathcal{V}}
\newcommand{\W}{\mathcal{W}}
\newcommand{\A}{\mathcal{A}}
\newcommand{\HHH}{\mathcal{H}}
\newcommand{\PP}{\mathcal{P}}
\newcommand{\K}{\mathcal{K}}
\newcommand{\z}{\mathcal{Z}}
\newcommand{\RRR}{\mathscr{R}}
\newcommand{\DDD}{\mathscr{D}}
\newcommand{\KKK}{\mathscr{K}}
\newcommand{\LLL}{\mathscr{L}}
\newcommand{\JJJ}{\mathscr{J}}
\newcommand{\WWW}{\mathscr{W}}
\newcommand{\e}{\mathscr{E}}
\newcommand{\hh}{\mathscr{H}}
\newcommand{\kk}{\mathscr{K}}
\newcommand{\jj}{\mathscr{J}}
\newcommand{\w}{\mathscr{W}}
\newcommand{\f}{\mathscr{F}}
\newcommand{\g}{\mathscr{G}}
\newcommand{\so}{\mathfrak{so}}
\newcommand{\gl}{\mathfrak{gl}}
\newcommand{\DD}{\mathsf{D}}
\newcommand{\ff}{\mathsf{F}}
\newcommand{\bpi}{\boldsymbol{\pi}}
\newcommand{\bee}{\boldsymbol{\e}}
\newcommand{\bhh}{\boldsymbol{\hh}}
\newcommand{\bkk}{\boldsymbol{\kk}}
\newcommand{\bjj}{\boldsymbol{\jj}}
\newcommand{\bLLL}{\boldsymbol{\LLL}}
\newcommand{\bRRR}{\boldsymbol{\RRR}}
\newcommand{\bWWW}{\boldsymbol{\WWW}}
\newcommand{\bSSS}{\boldsymbol{S}}
\numberwithin{equation}{section}
\numberwithin{thm}{section}
\numberwithin{lemma}{section}
\numberwithin{coro}{section}
\numberwithin{prop}{section}
\numberwithin{remark}{section}
\numberwithin{defi}{section}
\normalsize
\title[]{Local index theory and $\Z/k\Z$ $K$-theory}
\author{Man-Ho Ho}
\email{homanho@bu.edu}
\address{Hong Kong}
\subjclass[2020]{Primary 19K56, 19L10; Secondary 19L50, 58J20}
\keywords{Riemann--Roch--Grothendieck theorem, analytic index, Bismut--Cheeger eta form, $\Z/k\Z$ $K$-theory}

\nocite{*}
\begin{abstract}
For any given submersion $\pi:X\to B$ with closed, oriented and spin$^c$ fibers of even dimension, equipped with a Riemannian and differential spin$^c$ structure, we apply the Atiyah--Singer--Gorokhovsky--Lott approach to the local family index theorem without the kernel bundle assumption to construct an analytic index $\ind^a_k$ in odd $\Z/k\Z$ $K$-theory at the cocycle level. This is achieved by associating to every cocycle $(\EE, \FF, \alpha)$ of the odd $\Z/k\Z$ $K$-theory group of $X$ a cocycle $\ind^a_k(\EE, \FF, \alpha)$ of the odd $\Z/k\Z$ $K$-theory group of $B$. We also prove a Riemann--Roch--Grothendieck-type formula in odd $\Z/k\Z$ $K$-theory, which expresses the Cheeger--Chern--Simons form of $\ind^a_k(\EE, \FF, \alpha)$ in terms of that of $(\EE, \FF, \alpha)$. Furthermore, we show that the analytic index $\ind^a_k$ and the Riemann--Roch--Grothendieck-type formula in odd $\Z/k\Z$ $K$-theory refine the underlying geometric bundle of the analytic index and the Riemann--Roch--Grothendieck theorem in $\R/\Z$ $K$-theory, respectively.
\end{abstract}
\maketitle
\section{Introduction}\label{s 1}

\subsection{Main results}\label{s 1.1}

Let $k$ be a positive integer. Atiyah--Patodi--Singer define the $\Z/k\Z$ $K$-theory group \cite[\S5]{APS75} as
$$K(X; \Z/k\Z):=K(X\times M_k; X\times\set{\ast}),$$
where $M_k$ is the homotopy cofiber of the map $f_k:\sss^1\to\sss^1$ given by $f_k(z)=z^k$. Elements of the odd $\Z/k\Z$ $K$-theory group $K^{-1}(X; \Z/k\Z)$ can be represented by triples of the form
$$(E, F, \alpha),$$
where $E\to X$ and $F\to X$ are complex vector bundles and $\alpha:kE\to kF$ is a bundle isomorphism \cite[Proposition 5.5]{APS75}. Alternatively, the odd $\Z/k\Z$ $K$-theory group can also be defined as the $K$-theory group of a functor due to Karoubi \cite[6.18 of Chapter II]{K08}. Denote by $\bun(X)$ the category of complex vector bundles over $X$ with morphisms bundle homomorphisms. Define a functor $\varphi_k:\bun(X)\to\bun(X)$ by
$$\varphi_k(E)=kE,\qquad\varphi_k(T)=kT,$$
where $E\to X$ is a complex vector bundle and $T:E\to F$ is a bundle homomorphism. The odd $\Z/k\Z$ $K$-theory group $K^{-1}(X; \Z/k\Z)$ is defined to be $K(\varphi_k)$ \cite[2.13 of Chapter II]{K08}.

In this paper, for any given submersion $\pi:X\to B$ with closed, oriented and spin$^c$ fibers of even dimension, equipped with a Riemannian and differential spin$^c$ structure, we construct an analytic index in odd $\Z/k\Z$ $K$-theory at the cocycle level. That is, for any triple $(E, F, \alpha)$ representing an element of $K^{-1}(X; \Z/k\Z)$, we define an analytic index $\ind^a_k(E, F, \alpha)$ as an element of $K^{-1}(B; \Z/k\Z)$.

To define an analytic index in odd $\Z/k\Z$ $K$-theory using techniques in local index theory, we need to consider an odd $\Z/k\Z$ $K$-theory group defined in terms of a category of complex vector bundles which takes metrics and connections into account (see also Remark \ref{remark 6.1}). For this purpose, we consider a category $\bun_\nabla(X)$ of geometric bundles $\EE=(E, g^E, \nabla^E)$ over $X$, where $E\to X$ is a complex vector bundle equipped with a Hermitian metric $g^E$ and a unitary connection $\nabla^E$. Given any two geometric bundles $\EE$ and $\FF$, a morphism $T:\EE\to\FF$ in $\bun_\nabla(X)$ is an isometric isomorphism $T:(E, g^E)\to(F, g^F)$ of the underlying Hermitian bundles. Define a functor $\varphi_{k, \nabla}:\bun_\nabla(X)\to\bun_\nabla(X)$ by
$$\varphi_{k, \nabla}(\EE)=k\EE,\qquad\varphi_{k, \nabla}(T)=kT,$$
where $\EE$ is a geometric bundle and $T:\EE\to\FF$ is a morphism. Throughout this paper, we work with the odd $\Z/k\Z$ $K$-theory group defined by $K(\varphi_{k, \nabla})$, and still denote it as $K^{-1}(X; \Z/k\Z)$. Elements of $K^{-1}(X; \Z/k\Z)$ are represented by triples $(\EE, \FF, \alpha)$, which are called $k$-cocycles over $X$, where $\alpha:k\EE\to k\FF$ is a morphism. If $\EE$ and $\FF$ are $\Z_2$-graded and $\alpha$ is even, we say the $k$-cocycle $(\EE, \FF, \alpha)$ is $\Z_2$-graded.

Let $\pi:X\to B$ be a submersion with closed, oriented and spin$^c$ fibers of even dimension, equipped with a Riemannian and differential spin$^c$ structure $\bpi$. In this paper, by employing the Atiyah--Singer--Gorokhovsky--Lott (ASGL) approach to the local family index theorem (FIT) without the kernel bundle assumption \cite[\S4, 5]{GL18} (the kernel bundle assumption refers to the family of kernels $\set{\ker(\DD^{S\otimes E}_b)}_{b\in B}$ of the twisted spin$^c$ Dirac operator $\DD^{S\otimes E}$ is assumed to form a ($\Z_2$-graded) complex vector bundle $\ker(\DD^{S\otimes E})\to B$), we associate to each $k$-cocycle $(\EE, \FF, \alpha)$ over $X$ a $\Z_2$-graded $k$-cocycle over $B$:
\begin{equation}\label{eq 1.1}
\ind^a_k(\EE, \FF, \alpha):=\big(\KER\big(\DD^{S\otimes E; \wt{\Delta}_E}\big), \KER\big(\DD^{S\otimes F; \wt{\Delta}_F}\big), \wt{h}\big),
\end{equation}
where
\begin{itemize}
  \item $\DD^{S\otimes E}$ and $\DD^{S\otimes F}$ are the twisted spin$^c$ Dirac operators associated to $(\EE, \bpi)$ and $(\FF, \bpi)$, respectively, which are not assumed to satisfy the kernel bundle assumption,
  \item $\KER\big(\DD^{S\otimes E; \wt{\Delta}_E}\big)$ and $\KER\big(\DD^{S\otimes F; \wt{\Delta}_F}\big)$ are $\Z_2$-graded geometric bundles over $B$, called geometric kernel bundles in this paper, whose underlying bundles are bundle representatives of the analytic indexes of $[E]$ and $[F]$ in $K$-theory, respectively, and
  \item $\wt{h}:k\KER\big(\DD^{S\otimes E; \wt{\Delta}_E}\big)\to k\KER\big(\DD^{S\otimes F; \wt{\Delta}_F}\big)$ is a $\Z_2$-graded morphism.
\end{itemize}
We call $\ind^a_k(\EE, \FF, \alpha)$ the analytic index of $(\EE, \FF, \alpha)$. The details are given in \S\ref{s 6.2}.

We would like to think of $\ind^a_k$ as the analytic index in odd $\Z/k\Z$ $K$-theory at the cocycle level. However, we do not claim the analytic index $\ind^a_k$ (\ref{eq 1.1}) descends to a well defined group homomorphism
$$\ind^a_{\Z/k\Z}:K^{-1}(X; \Z/k\Z)\to K^{-1}(B; \Z/k\Z),$$
which was the original purpose of this paper. See Remark \ref{remark 6.2} for a detailed discussion.

We associate to every $k$-cocycle $(\EE, \FF, \alpha)$ a differential form $\CCS_k(\EE, \FF, \alpha)$ of odd degree, called the Cheeger--Chern--Simons form of $(\EE, \FF, \alpha)$. The first main result is a Riemann--Roch--Grothendieck (RRG) type formula in odd $\Z/k\Z$ $K$-theory (Theorem \ref{thm 6.1}), which expresses the Cheeger--Chern--Simons form of the analytic index $\ind^a_k(\EE, \FF, \alpha)$ in terms of that of $(\EE, \FF, \alpha)$.
\begin{thmx}\label{thm A}
Let $\pi:X\to B$ be a submersion with closed, oriented and spin$^c$ fibers of even dimension, equipped with a Riemannian and differential spin$^c$ structure. For every $k$-cocycle $(\EE, \FF, \alpha)$ over $X$,
\begin{displaymath}
\begin{split}
\CCS_k\big(\ind^a_k(\EE, \FF, \alpha)\big)&=\int_{X/B}\todd(\nabla^{T^VX})\wedge\CCS_k(\EE, \FF, \alpha)\\
&\qquad+\wt{\eta}(\EE, \bpi, \HH_E, s_E)-\wt{\eta}(\FF, \bpi, \HH_F, s_F)
\end{split}
\end{displaymath}
in $\frac{\Omega^{\odd}(B)}{\im(d)}$, where $\wt{\eta}(\FF, \bpi, \HH_F, s_F)$ and $\wt{\eta}(\EE, \bpi, \HH_E, s_E)$ are the corresponding Bismut--Cheeger eta forms.
\end{thmx}

One motivation of defining the analytic index $\ind^a_k$ (\ref{eq 1.1}) and proving the RRG-type formula in odd $\Z/k\Z$ $K$-theory (Theorem \ref{thm A}) comes from the analytic index and the RRG theorem in $\R/\Z$ $K$-theory and the fact that $\R/\Z$ $K$-theory is defined in terms of $\Z/k\Z$ $K$-theory \cite[p.88]{APS76}. To put these results into context, we recall the analytic index and the RRG theorem in $\R/\Z$ $K$-theory. The geometric model of $\R/\Z$ $K$-theory $K^{-1}(X; \R/\Z)$ given by Lott \cite[Definition 7]{L94} can be defined in terms of $\Z_2$-graded generators of the form $\E=(\EE, \omega)$, where $\EE$ is a $\Z_2$-graded geometric bundle and $\omega\in\frac{\Omega^{\odd}(X)}{\im(d)}$ satisfies
\begin{equation}\label{eq 1.2}
\ch(\nabla^{E, +})-\ch(\nabla^{E, -})=-d\omega.
\end{equation}
Moreover, for any given submersion $\pi:X\to B$ with closed, oriented and spin$^c$ fibers of even dimension, equipped with a Riemannian and differential spin$^c$ structure $\bpi$, Lott defines an analytic index
$$\ind^a_{\R/\Z}:K^{-1}(X; \R/\Z)\to K^{-1}(B; \R/\Z)$$
in $\R/\Z$ $K$-theory \cite[Definition 14]{L94} by
$$\ind^a_{\R/\Z}(\E)=\bigg(\KER(\DD^{S\otimes E}), \int_{X/B}\todd(\nabla^{T^VX})\wedge\omega+\wh{\eta}(\EE, \bpi)\bigg),$$
where, for the purpose of illustration, the twisted spin$^c$ Dirac operator $\DD^{S\otimes E}$ associated to $(\EE, \bpi)$ is assumed to satisfy the kernel bundle assumption and $\wh{\eta}(\EE, \bpi)$ is the corresponding Bismut--Cheeger eta form. He then proves the FIT in $\R/\Z$ $K$-theory \cite[Corollary 3]{L94}, namely,
$$\ind^a_{\R/\Z}=\ind^t_{\R/\Z},$$
where $\ind^t_{\R/\Z}:K^{-1}(X; \R/\Z)\to K^{-1}(B; \R/\Z)$ is the topological index in $\R/\Z$ $K$-theory. By applying the Chern character $\ch_{\R/\Q}:K^{-1}(B; \R/\Z)\to H^{\odd}(B; \R/\Q)$ in $\R/\Z$ $K$-theory \cite[Definition 9]{L94} to the above FIT, he obtains the RRG theorem in $\R/\Z$ $K$-theory \cite[Corollary 4]{L94}, which states that
\begin{equation}\label{eq 1.3}
\ch_{\R/\Q}\big(\ind^a_{\R/\Z}(\E)\big)=\int_{X/B}\todd(T^VX)\cup\ch_{\R/\Q}(\E)
\end{equation}
in $H^{\odd}(B; \R/\Q)$.

Another motivation comes from our proof of the RRG theorem in $\R/\Z$ $K$-theory \cite[Theorem 1.5]{H23}. We wonder whether the analytic index and the RRG-type formula in odd $\Z/k\Z$ $K$-theory refine their analogues in $\R/\Z$ $K$-theory in a certain sense. The second main result is an affirmative answer to this question. 

We associate to each ($\Z_2$-graded) $k$-cocycle $(\EE, \FF, \alpha)$ in $K^{-1}(X; \Z/k\Z)$ a $\Z_2$-graded generator $T(\EE, \FF, \alpha)$ in $K^{-1}(X; \R/\Z)$ (see (\ref{eq 6.2}) and (\ref{eq 6.3}) for the details). We prove that $T$ descends to a group homomorphism (Proposition \ref{prop 6.1})
$$T:K^{-1}(X; \Z/k\Z)\to K^{-1}(X; \R/\Z).$$ 
The second main result (Propositions \ref{prop 6.2} and \ref{prop 6.3}) can now be stated more precisely as follows.
\begin{propx}\label{prop A}
Let $\pi:X\to B$ be a submersion with closed, oriented and spin$^c$ fibers of even dimension, equipped with a Riemannian and differential spin$^c$ structure.
\begin{enumerate}
\item For every $k$-cocycle $(\EE, \FF, \alpha)$ over $X$,
    $$T\big(\ind^a_k(\EE, \FF, \alpha)\big)=\ind^a_{\R/\Z}\big(T(\EE, \FF, \alpha)\big).$$
\item Let $\E=(\EE, \omega)$ be a $\Z_2$-graded generator of $K^{-1}(X; \R/\Z)$. Fix a positive integer $k$ satisfying $k\EE^+\cong k\EE^-$ and a morphism $\alpha:k\EE^+\to k\EE^-$ so that $(\EE^+, \EE^-, \alpha)$ is a $k$-cocycle. The RRG-type formula in odd $\Z/k\Z$ $K$-theory for $(\EE^+, \EE^-, \alpha)$ is a refinement of the RRG theorem in $\R/\Z$ $K$-theory for $\E$ at the differential form level.
\end{enumerate}
\end{propx}

Since $\Z/k\Z$ $K$-theory is a 2-periodic cohomology theory \cite[p.428]{APS75}, it follows from \cite[Chapter 1D]{D69} that there is a topological index 
$$\ind^t_{\Z/k\Z}:K^{-1}(X; \Z/k\Z)\to K^{-1}(B; \Z/k\Z).$$ 
It would be interesting to prove the FIT in odd $\Z/k\Z$ $K$-theory, i.e. 
$$\ind^a_k=\ind^t_{\Z/k\Z}.$$ 
In particular, this would imply the analytic index $\ind^a_k$ is a well defined group homomorphism.

\subsection{Method of proof}\label{s 1.2}

In this subsection, we make some remarks on the method we use to construct the analytic index $\ind^a_k$ in odd $\Z/k\Z$ $K$-theory, namely, the ASGL approach to the local FIT without the kernel bundle assumption.

In this paper, we employ the ASGL approach to construct the analytic index $\ind^a_k$ and prove the RRG-type formula in odd $\Z/k\Z$ $K$-theory, while we employ the Mi\v s\v cenko--Fomenko--Freed--Lott (MFFL) approach \cite[\S7]{FL10} to derive the main results in \cite{H20, H23}. One might ask why we adopt the ASGL approach instead of the MFFL approach in this paper. The answer turns out to be somewhat subtle.

First, as an intermediate result (Proposition \ref{prop 4.1}), we give a proof that the two approaches are equivalent in differential $K$-theory.
\begin{propx}\label{prop B}
Let $\pi:X\to B$ be a submersion with closed, oriented and spin$^c$ fibers of even dimension, equipped with a Riemannian and differential spin$^c$ structure. Denote by $\wt{\ind}^a_{\wh{K}}$ and $\ind^a_{\wh{K}}$ the analytic indexes in differential $K$-theory defined by the ASGL and MFFL approaches, respectively. Then
$$\wt{\ind}^a_{\wh{K}}=\ind^a_{\wh{K}}:\wh{K}(X)\to\wh{K}(B),$$
where $\wh{K}$ is the geometric model of differential $K$-theory given by Freed--Lott \cite[Definition 2.16]{FL10}.
\end{propx}

Recall that Gorokhovsky--Lott define a geometric model $\wh{K}_{\GL}$ of differential $K$-theory and an analytic index $\pi_*$ on $\wh{K}_{\GL}$ \cite[Definitions 6 and 7]{GL18}, and prove that
\begin{equation}\label{eq 1.4}
\pi_*=\ind^a_{\wh{K}}
\end{equation}
when $\pi_*$ is restricted to $\wh{K}$ \cite[Proposition 9]{GL18}. Note that Proposition \ref{prop B} is a special case of (\ref{eq 1.4}). They give a detailed proof of (\ref{eq 1.4}) under the kernel bundle assumption. For the general case of (\ref{eq 1.4}), they reduce it to the case where the kernel bundle assumption holds. We give a direct proof of Proposition \ref{prop B} without the kernel bundle assumption. As an immediate consequence of Proposition \ref{prop B}, we express the analytic index $\ind^a_{\R/\Z}$ in $\R/\Z$ $K$-theory via the ASGL approach (Corollary \ref{coro 6.1}), which enables us to prove Proposition \ref{prop A}.

Despite Proposition \ref{prop B}, there is a subtle difference between the ASGL and MFFL approaches from the perspective of the construction of the analytic index $\ind^a_k$. This is exactly the reason we adopt the former approach in this paper. To clarify the reason, we first briefly review the two approaches.

For any given twisted spin$^c$ Dirac operator $\DD^{S\otimes E}$, Atiyah--Singer prove that there always exist a finite rank complex vector bundle $H\to B$ and a linear map $s:H\to(\pi_*E)^-$ such that $\DD^{S\otimes E}_++s:(\pi_*E)^+\oplus H\to(\pi_*E)^-$ is surjective \cite[Proposition 2.2]{AS71} (see also \cite[Lemma 8.4 of Chapter III]{LM89} and \cite[Lemma 9.30]{BGV}). Moreover, $H\to B$ can always be taken as a trivial bundle. In this case, the analytic index of $[E]$ in $K$-theory is given by
$$\ind^a([E])=\big[\ker\big(\DD^{S\otimes E; \Delta}\big)\big]-[H],$$
where $\Delta$ is an odd self-adjoint map associated to $(H, s)$ and $\DD^{S\otimes E; \Delta}$ is a perturbation of $\DD^{S\otimes E}$. Based on the Atiyah--Singer approach, Gorokhovsky–Lott define a corresponding Bismut superconnection $\wt{\bbb}^{E; \Delta}$ in the sense of Quillen \cite{D85}. The local FIT in the ASGL approach states that
$$d\wt{\eta}(\EE, \bpi, \HH, s)=\int_{X/B}\todd(\nabla^{T^VX})\wedge\ch(\nabla^E)-\ch(\nabla^{\ker(\DD^{S\otimes E; \Delta})}),$$
where $\HH$ is a geometric bundle associated to $H\to B$ and $\wt{\eta}(\EE, \bpi, \HH, s)$ is the Bismut--Cheeger eta form associated to $\wt{\bbb}^{E; \Delta}$.

On the other hand, Mi\v s\v cenko--Fomenko prove that there always exist a finite rank $\Z_2$-graded subbundle $L\to B$ of $\pi_*E\to B$ and a $\Z_2$-graded closed complementary subbundle $K\to B$ of $L\to B$ such that
\begin{equation}\label{eq 1.5}
(\pi_*E)^+=K^+\oplus L^+\quad\textrm{ and }\quad(\pi_*E)^-=K^-\oplus L^-,
\end{equation}
$\DD^{S\otimes E}_+:(\pi_*E)^+\to(\pi_*E)^-$ is block diagonal with respect to (\ref{eq 1.5}) and $\DD^{S\otimes E}_+:K_+\to K_-$ is an isomorphism \cite[p.96-97]{MF79}. In this case, the analytic index of $[E]$ in $K$-theory is given by
$$\ind^a([E])=[L^+]-[L^-].$$
Based on the Mi\v s\v cenko--Fomenko approach, Freed–Lott define a corresponding Bismut superconnection $\wh{\bbb}^E$. The local FIT in the MFFL approach states that
$$d\wh{\eta}(\EE, \bpi, \LL)=\int_{X/B}\todd(\nabla^{T^VX})\wedge\ch(\nabla^E)-\ch(\nabla^L),$$
where $\wh{\eta}(\EE, \bpi, \LL)$ is the Bismut--Cheeger eta form associated to $\wh{\bbb}^E$.

The main difficulty in the construction of the analytic index $\ind^a_k$ is to define a $\Z_2$-graded morphism  between geometric kernel bundles. To do so, we first apply the extended variational formula for the Bismut--Cheeger eta form in the ASGL approach (Proposition \ref{prop 5.4}) to the $k$-cocycle $(\EE, \FF, \alpha)$ and the additivity of geometric kernel bundles and of the Bismut--Cheeger eta form (Proposition \ref{prop 5.5}) to obtain a $\Z_2$-graded morphism
\begin{equation}\label{eq 1.6}
h:k\KER\big(\DD^{S\otimes E; \Delta_E}\big)\oplus\wh{\VV}_0\to k\KER\big(\DD^{S\otimes F; \Delta_F}\big)\oplus\wh{\VV}_1,
\end{equation}
where $\wh{\VV}_0$ and $\wh{\VV}_1$ are $\Z_2$-graded geometric bundles. In fact, $h$ is defined to be the composition of the $\Z_2$-graded morphisms in the vertical columns and the bottom horizontal row of the following diagram.
\begin{center}
\begin{tikzcd}
k\KER\big(\DD^{S\otimes E; \Delta_E}\big)\oplus\wh{\VV}_0 \arrow{d} \arrow[rr, dashed]{}{h} & & k\KER\big(\DD^{S\otimes F; \Delta_F}\big)\oplus\wh{\VV}_1 \\ \KER\big(\DD^{S\otimes kE; k\Delta_E}\big)\oplus\wh{\VV}_0 \arrow{r} & \KER\big(\DD^{S\otimes kE; \Delta_\alpha}_\alpha\big)\oplus\wh{\VV}_1 \arrow{r} & \KER\big(\DD^{S\otimes kF; k\Delta_F}\big)\oplus\wh{\VV}_1 \arrow{u} & 
\end{tikzcd}
\end{center}
See (\ref{eq 6.10}) for the details. When defining the analytic indexes of $\DD^{S\otimes E}$ and $\DD^{S\otimes F}$ in $K$-theory via the Atiyah--Singer approach, the fact that we can take $H_E\to B$ and $H_F\to B$ to be trivial bundles and our proofs of Propositions \ref{prop 5.1} to \ref{prop 5.3} show that the $\wh{\VV}_0$ and $\wh{\VV}_1$ in (\ref{eq 1.6}) are trivial (see (2) of Remark \ref{remark 5.1} for the details). This fact enables us to take the direct sum of the domain and the image of $h$ (\ref{eq 1.6}) with $\Z_2$-graded trivial geometric bundles $\wh{\CC}^r$ for some appropriate rank $r$, whose purpose is to make the ranks of both the domain and the image of $h$ (\ref{eq 1.6}) to be some multiple of $k$. Then by using Lemma \ref{lemma 5.1}, the resulting domain and image of $h$ (\ref{eq 1.6}) can be written as the direct sum of $k$ copies of the geometric kernel bundles appearing on the right-hand side of (\ref{eq 1.1}). See \S\ref{s 6.2} for the details. Thus the triple defining $\ind^a_k(\EE, \FF, \alpha)$ (\ref{eq 1.1}) is a $\Z_2$-graded $k$-cocycle over $B$. In contrast, we are unable to obtain a $\Z_2$-graded cocycle of the form (\ref{eq 1.1}) by employing the Mi\v s\v cenko--Fomenko approach since, in general, the corresponding $\wh{\VV}_0$ and $\wh{\VV}_1$ cannot be made trivial to begin with.

If our goal was just to obtain the extended variational formula for the Bismut--Cheeger eta form in the ASGL approach, we could have applied Proposition \ref{prop B} and \cite[Theorem 3.4]{H23}. However, in such a case, the $\wh{\VV}_0$ and $\wh{\VV}_1$ in (\ref{eq 1.6}) are not trivial in general. This necessitates proving the extended variational formula for the Bismut--Cheeger eta form in the ASGL approach (Proposition \ref{prop 5.4}) along the lines similar to those in \cite[\S3]{H23}.

\subsection{Outline}\label{s 1.3}

The paper is organized as follows. In Section \ref{s 2}, we fix terminology and notation and review the background material, including the category $\bun_\nabla(X)$ and properties of the Chern--Simons form. In Section \ref{s 3}, we review the ASGL and MFFL approaches to the local FIT without the kernel bundle assumption.

We prove the equivalence of the analytic indexes in differential $K$-theory defined by the ASGL and MFFL approaches in Section \ref{s 4}. In Section \ref{s 5}, we prove an extended variational formula for the Bismut--Cheeger eta form in the ASGL approach. In Section \ref{s 6}, we first define the odd $\Z/k\Z$ $K$-theory group by applying Karoubi's construction to the category $\bun_\nabla(X)$ and the functor $\varphi_{k, \nabla}$. We then construct the analytic index $\ind^a_k$ and prove the RRG-type formula in odd $\Z/k\Z$ $K$-theory. Finally, we show that the analytic index $\ind^a_k$ and the RRG-type formula in odd $\Z/k\Z$ $K$-theory refine their analogues in $\R/\Z$ $K$-theory.

\section*{Acknowledgement}

We thank Steve Rosenberg for his invaluable advice, Thomas Schick for his inspiring comments and suggestions on $\Z/k\Z$ $K$-theory, and Steven Costenoble for generously answering many questions. We also thank the referee for the comments and suggestions to improve the paper significantly.

\section{Preliminaries}\label{s 2}

We review the background material in this section. We introduce the category $\bun_\nabla(X)$, recall the notion of split quadruple and a morphism between geometric bundles in \S\ref{s 2.1}. In \S\ref{s 2.2}, we recall the definition and some properties of Chern--Simons form, and we refer the readers to \cite[\S B.5]{MM07} for the details.

\subsection{Background material}\label{s 2.1}

In this paper, $X$ and $B$ are closed manifolds and $I$ is the closed interval $[0, 1]$. Given a manifold $X$, define $\wt{X}=X\times I$. Given $t\in I$, define a map $i_{X, t}:X\to\wt{X}$ by $i_{X, t}(x)=(x, t)$. Denote by $p_X:\wt{X}\to X$ the standard projection map. For any differential forms $\omega$ and $\eta$, we write $\omega\equiv\eta$ if $\omega-\eta\in\im(d)$.

Let $M$ be a smooth manifold and $\pi:M\to B$ a smooth fiber bundle with compact fibers of dimension $n$ which satisfies certain orientability assumptions. Then
$$\int_{M/B}\pi^*\alpha\wedge\beta=\alpha\wedge\bigg(\int_{M/B}\beta\bigg)$$
for any $\alpha\in\Omega^\bullet(B)$ and $\beta\in\Omega^\bullet(M)$. If $M$ has nonempty boundary, then Stokes' theorem for integration along the fibers \cite[Problem 4 of Chapter VII]{GHV} states that for any $\omega\in\Omega^k(M)$,
\begin{equation}\label{eq 2.1}
(-1)^{k-n+1}\int_{\partial M/B}i^*\omega=\int_{M/B}d^M\omega-d^B\int_{M/B}\omega,
\end{equation}
where $i:\partial M\hto M$ is the inclusion map.

Let $\bun_\nabla(X)$ be the category of geometric bundles $\EE=(E, g^E, \nabla^E)$, where $E\to X$ is a complex vector bundle equipped with a Hermitian metric $g^E$ and a unitary connection $\nabla^E$. Given any two geometric bundles $\EE$ and $\FF$, a morphism $\alpha:\EE\to\FF$ is an isometric isomorphism $\alpha:(E, g^E)\to(F, g^F)$ of the underlying Hermitian bundles. Throughout this paper, a morphism $T:\EE\to\FF$ between geometric bundles always means the one in $\bun_\nabla(X)$, but not the morphism between the underlying bundles in $\bun(X)$. Since every morphism in $\bun_\nabla(X)$ is an isomorphism, $\bun_\nabla(X)$ is a groupoid. But we will not use this fact. A geometric bundle $\EE=(E, g^E, \nabla^E)$ is said to be $\Z_2$-graded if $E\to X$, $g^E$ and $\nabla^E$ are $\Z_2$-graded. For a $\Z_2$-graded geometric bundle $\EE$, denote by $\EE^\pm$ the even and the odd part of $\EE$, respectively, and by $\EE^{\op}$ the $\Z_2$-graded geometric bundle whose $\Z_2$-grading is opposite to that of $\EE$. Given a geometric bundle $\EE$, denote by $\wh{\EE}$ the associated $\Z_2$-graded geometric bundle defined by $\wh{\EE}^\pm=\EE$. Such $\Z_2$-graded geometric bundles are said to be elementary. For any $\ell\in\N$, write $\CC^\ell=(\C^\ell, g^\ell, d^\ell)$, where $\C^\ell\to X$ is the trivial complex vector bundle of rank $\ell$, and $g^\ell$ and $d^\ell$ are the standard metric and the standard connection on $\C^\ell\to X$, respectively.

We now recall the definition of a split quadruple \cite[Definition 2.9]{B96}. Let $\WW$ be a geometric bundle over $B$. The map $s_{\wh{W}}:\wh{W}\to\wh{W}$ defined by
\begin{equation}\label{eq 2.2}
s_{\wh{W}}=\begin{pmatrix} 0 & \id \\ \id & 0 \end{pmatrix}
\end{equation}
is odd self-adjoint. Then $(\wh{\WW}, s_{\wh{W}})$ is an example of split quadruple. Fix $a\in(0, 1)$. Let $\alpha:[0, \infty)\to I$ be a smooth function that satisfies $\alpha(t)=0$ for all $t\leq a$ and $\alpha(t)=1$ for all $t\geq 1$. Define a rescaled superconnection on $\wh{W}\to B$ by
\begin{equation}\label{eq 2.3}
\aaa^{\wh{W}}_t=\sqrt{t}\alpha(t)s_{\wh{W}}+\nabla^{\wh{W}}.
\end{equation}
Define the Bismut--Cheeger eta form associated to $\aaa^{\wh{W}}_t$ by
$$\wh{\eta}^{\wh{W}}=\frac{1}{\sqrt{\pi}}\int^\infty_0\str\bigg(\frac{d\aaa^{\wh{W}}_t}{dt}e^{-\frac{1}{2\pi i}(\aaa^{\wh{W}}_t)^2}\bigg)dt.$$
Since 
$$\str\bigg(\frac{d\aaa^{\wh{W}}_t}{dt}e^{-\frac{1}{2\pi i}(\aaa^{\wh{W}}_t)^2}\bigg)=0$$ 
by \cite[(2.24)]{B96}, it follows that
\begin{equation}\label{eq 2.4}
\wh{\eta}^{\wh{W}}=0.
\end{equation}

We now recall a morphism between geometric bundles defined by parallel transport. Let $\bee=(\e, g^\e, \nabla^\e)$ be a geometric bundle over $\wt{X}$. Let $j\in\set{0, 1}$ and define a geometric bundle $\EE_j$ over $X$ by $\EE_j=i_{X, j}^*\bee$. Note that for each $p\in X$,
$$(E_j)_p=(i_{X, j}^*\e)_p=\e_{(p, j)}.$$

For each fixed $p\in X$, define a smooth curve $c_p:I\to\wt{X}$ by $c_p(t)=(p, t)$. The linear map $P_{c_p}:\e_{(p, 0)}\to\e_{(p, 1)}$ given by the parallel transport associated to $\nabla^\e$ along $c_p$ is independent of the choices of reparameterizations of $c_p$ and the representation of $c_p$ as a concatenation of piecewise smooth curves. Write $\pi:E_0\to X$ for the bundle projection. The map
\begin{equation}\label{eq 2.5}
P_{01}:\EE_0\to\EE_1
\end{equation}
defined by
$$P_{01}(e):=P_{c_{\pi(e)}}(e)$$
is a morphism (see, for example, \cite[Remark 5.2]{AB07}) and is uniquely determined by $\bee$ and $c_p$.

\subsection{Chern--Simons form}\label{s 2.2}

Let $\EE=(E, g^E, \nabla^E)$ be a geometric bundle over $X$. Denote by $R^E$ the curvature of $\nabla^E$. The Chern character form of $\nabla^E$ is defined to be
$$\ch(\nabla^E)=\tr(e^{-\frac{1}{2\pi i}R^E}),$$
and the first Chern form of $\nabla^E$ is defined to be
$$c_1(\nabla^E)=-\frac{1}{2\pi i}\tr(R^E).$$

Let $\EE_0=(E, g^E_0, \nabla^E_0)$ and $\EE_1=(E, g^E_1, \nabla^E_1)$ be geometric bundles over $X$. By \cite[Theorem 8.8 of Chapter I]{K08}, there exists a unique $f\in\Aut(E)$ such that $g^E_0=f^*g^E_1$. Thus $f^*\nabla^E_1$ is unitary with respect to $g^E_0$. For each $t\in I$, let $f_t=(1-t)\id_E+tf^{-1}$. Then
\begin{equation}\label{eq 2.6}
g^E_t=f_t^*g^E_0\quad\textrm{ and }\quad\nabla^E_t=f_t^*\big((1-t)\nabla^E_0+tf^*\nabla^E_1\big)
\end{equation}
are smooth paths of Hermitian metrics and unitary connections from $g^E_0$ to $g^E_1$ and from $\nabla^E_0$ to $\nabla^E_1$, respectively. Note that $\nabla^E_t$ is unitary with respect to $g^E_t$ for each $t\in I$. Define a geometric bundle $\bee=(\e, g^\e, \nabla^\e)$ over $\wt{X}$ by $\e=p_X^*E$,
\begin{equation}\label{eq 2.7}
g^\e=p_X^*g^E_t\quad\textrm{ and }\quad\nabla^\e=dt\wedge\bigg(\frac{\partial}{\partial t}+\frac{1}{2}(g^E_t)^{-1}\frac{\partial}{\partial t}g^E_t\bigg)+\nabla^E_t.
\end{equation}
Note that $i_{X, j}^*\bee=\EE_j$ for $j\in\set{0, 1}$. The Chern--Simons form $\CS(\nabla^E_0, \nabla^E_1)\in\frac{\Omega^{\odd}(X)}{\im(d)}$ is defined by
$$\CS(\nabla^E_0, \nabla^E_1):=-\int_{\wt{X}/X}\ch(\nabla^\e)\mod\im(d).$$
Note that $\CS(\nabla^E_0, \nabla^E_1)$ does not depend on the choice of $(g^\e, \nabla^\e)$ satisfying $i_{X, j}^*(g^\e, \nabla^\e)=(g^E_j, \nabla^E_j)$, and satisfies the following transgression formula:
$$d\CS(\nabla^E_0, \nabla^E_1)=\ch(\nabla^E_1)-\ch(\nabla^E_0).$$
Equivalently, the Chern--Simons form can be defined as
\begin{equation}\label{eq 2.8}
\CS(\nabla^E_0, \nabla^E_1)=-\frac{1}{2\pi i}\int^1_0\tr\bigg(\frac{d\nabla^E_t}{dt}e^{-\frac{1}{2\pi i}R^E_t}\bigg)dt.
\end{equation}
The choices of 0 and 1 are immaterial. If $t<T$ are two fixed positive real numbers, then one can replace 0 by $t$ and 1 by $T$ in (\ref{eq 2.8}).

The Chern--Simons form satisfies the following properties:
\begin{align}
\CS(\nabla^E_1, \nabla^E_0)&\equiv-\CS(\nabla^E_0, \nabla^E_1),\label{eq 2.9}\\
\CS(\nabla^E_1, \nabla^E_0)&\equiv\CS(\nabla^E_1, \nabla^E_2)+\CS(\nabla^E_2, \nabla^E_0),\label{eq 2.10}\\
\CS(\nabla^E_1\oplus\nabla^F_1, \nabla^E_0\oplus\nabla^F_0)&\equiv\CS(\nabla^E_1, \nabla^E_0)+\CS(\nabla^F_1, \nabla^F_0),\label{eq 2.11}
\end{align}
where $\FF_0=(F, g^F_0, \nabla^F_0)$ and $\FF_1=(F, g^F_1, \nabla^F_1)$ are geometric bundles over $X$. In particular, (\ref{eq 2.9}) implies 
\begin{equation}\label{eq 2.12}
\CS(\nabla^E, \nabla^E)\equiv 0.
\end{equation}
If $\EE_j$ is $\Z_2$-graded, then
\begin{equation}\label{eq 2.13}
\CS(\nabla_0^E, \nabla^E_1)\equiv\CS(\nabla_0^{E, +}, \nabla_1^{E, +})-\CS(\nabla_0^{E, -}, \nabla_1^{E, -}).
\end{equation}

Now suppose $g^F_0=g^F_1$. Let $\alpha:\EE\to\FF_j$ be a morphism. The naturality of Chern--Simons form implies
\begin{equation}\label{eq 2.14}
\CS(\alpha^*\nabla^F_0, \alpha^*\nabla^F_1)\equiv\CS(\nabla^F_0, \nabla^F_1).
\end{equation}

Let $(H, g^H, \nabla^H)$ be a Riemannian bundle with a metric connection. Denote by $R^H$ the curvature of $\nabla^H$. The $\wh{A}$-genus form of $\nabla^H$ is defined to be
$$\wh{A}(\nabla^H)=\sqrt{\det\bigg(\frac{-\frac{1}{4\pi i}R^H}{\sinh(-\frac{1}{4\pi i}R^H)}\bigg)}\in\Omega^{4\bullet}(X).$$

\section{Local index theory for twisted spin$^c$ Dirac operators}\label{s 3}

In this section, we review the setup and the statement of the local FIT for twisted spin$^c$ Dirac operators without the kernel bundle assumption. We review the geometric setup on a given submersion $\pi:X\to B$ with closed, oriented and spin$^c$ fibers of even dimension in \S\ref{s 3.1}. In \S\ref{s 3.2} and \S\ref{s 3.3}, we review the local FIT without the kernel bundle assumption via the ASGL and MFFL approaches, respectively.

\subsection{Riemannian and differential spin$^c$ structures}\label{s 3.1}

Let $\pi:X\to B$ be a submersion with closed, oriented and spin$^c$ fibers of even dimension $n$. Denote by $T^VX\to X$ the vertical tangent bundle. Recall from \cite[p.918]{FL10} that a Riemannian structure $(T^HX, g^{T^VX})$ on $\pi:X\to B$ consists of a horizontal distribution $T^HX\to X$, i.e. $TX=T^VX\oplus T^HX$, and a metric $g^{T^VX}$ on $T^VX\to X$. Put a Riemannian metric $g^{TB}$ on $TB\to B$. Define a metric $g^{TX}$ on $TX\to X$ by
$$g^{TX}=g^{T^VX}\oplus\pi^*g^{TB}.$$
Denote by $\nabla^{TX}$ and $\nabla^{TB}$ the Levi-Civita connections on $TX\to X$ and $TB\to X$ associated to $g^{TX}$ and $g^{TB}$, respectively. Let $P^{T^VX}:TX\to T^VX$ be the projection map. Note that $\nabla^{T^VX}:=P^{T^VX}\circ\nabla^{TX}\circ P^{T^VX}$ is a metric connection on the Riemannian bundle $(T^VX, g^{T^VX})$.

Define a connection $\wt{\nabla}^{TX}$ on $TX\to X$ by
$$\wt{\nabla}^{TX}=\nabla^{T^VX}\oplus\pi^*\nabla^{TB}.$$
Then $S:=\nabla^{TX}-\wt{\nabla}^{TX}\in\Omega^1(X, \End(TX))$. By \cite[Theorem 1.9]{B86}, the $(3, 0)$ tensor $g^{TX}(S(\cdot)\cdot, \cdot)$ depends only on $(T^HX, g^{T^VX})$.
Let $\set{e_1, \ldots, e_n}$ be a local orthonormal frame for $T^VX\to X$. For any $U\in\Gamma(B, TB)$, denote by $U^H\in\Gamma(X, T^HX)$ its horizontal lift. Define a horizontal one-form $k$ on $X$ by
\begin{equation}\label{eq 3.1}
k(U^H)=-\sum_{k=1}^ng^{TX}(S(e_k)e_k, U^H).
\end{equation}
For any $U, V\in\Gamma(B, TB)$,
\begin{equation}\label{eq 3.2}
T(U, V):=-P^{T^VX}[U^H, V^H]
\end{equation}
is a horizontal two-form with values in $T^VX$ and is called the curvature of $\pi:X\to B$.

Denote by $d\vol(Z)$ the Riemannian volume element of the fiber $Z$, which is a section of $\Lambda^n(T^VX)^*\to X$.

Fix a topological spin$^c$ structure on $T^VX\to X$. This fixes a complex line bundle $\lambda\to X$ satisfying $w_2(T^VX)=c_1(\lambda)\mod 2$ \cite[p.397]{LM89}. The spinor bundle $S(T^VX)\to X$ associated to the chosen topological spin$^c$ structure of $T^VX\to X$ is given by
$$S(T^VX)=S_0(T^VX)\otimes\lambda^{\frac{1}{2}},$$
where $S_0(T^VX)$ is the spinor bundle for the locally existing spin structure of $T^VX\to X$ and $\lambda^{\frac{1}{2}}$ is the locally existing square root of $\lambda\to X$. Since $n$ is even, $S(T^VX)\to X$ is $\Z_2$-graded. Recall from \cite[p.918]{FL10} that a differential spin$^c$ structure on $\pi:X\to B$ is a geometric bundle $(\lambda, g^\lambda, \nabla^\lambda)$ over $X$.

Denote by
$$\bpi=(T^HX, g^{T^VX}, \lambda, g^\lambda, \nabla^\lambda)$$
a Riemannian and differential spin$^c$ structure on $\pi:X\to B$. Note that $\bpi$ induces a $\Z_2$-graded geometric bundle $\bSSS(T^VX)$ over $X$. Define the Todd form of $\nabla^{T^VX}$ by
$$\todd(\nabla^{T^VX})=\wh{A}(\nabla^{T^VX})\wedge e^{\frac{1}{2}c_1(\nabla^\lambda)}.$$

Let $\EE=(E, g^E, \nabla^E)$ be a geometric bundle over $X$. Consider the tensor product $\bSSS(T^VX)\otimes\EE$. The twisted spin$^c$ Dirac operator $\DD^{S\otimes E}:\Gamma(X, S(T^VX)\otimes E)\to\Gamma(X, S(T^VX)\otimes E)$ associated to $(\EE, \bpi)$ is defined by
$$\DD^{S\otimes E}=\sum_{k=1}^nc(e_k)\nabla_{e_k}^{S(T^VX)\otimes E},$$
where $c$ is the Clifford multiplication. Note that $\DD^{S\otimes E}$ is odd self-adjoint.

Define an infinite rank $\Z_2$-graded complex vector bundle $\pi_*E\to B$ whose fiber over $b\in B$ is given by
$$(\pi_*E)_b=\Gamma(Z_b, (S(T^VX)\otimes E)|_{Z_b}).$$
The space of sections of $\pi_*E\to B$ is defined to be
$$\Gamma(B, \pi_*E):=\Gamma(X, S(T^VX)\otimes E).$$
Define an $L^2$-metric on $\pi_*E\to B$ by
$$g^{\pi_*E}(s_1, s_2)(b)=\int_{Z_b}g^{S(T^VX)\otimes E}(s_1, s_2)d\vol(Z).$$
Define a connection on $\pi_*E\to B$ by
$$\nabla^{\pi_*E}_Us:=\nabla^{S(T^VX)\otimes E}_{U^H}s,$$
where $s\in\Gamma(B, \pi_*E)$ and $U\in\Gamma(B, TB)$. Then the connection on $\pi_*E\to B$ defined by
$$\nabla^{\pi_*E, u}:=\nabla^{\pi_*E}+\frac{1}{2}k,$$
where $k$ is given by (\ref{eq 3.1}), is $\Z_2$-graded and unitary with respect to $g^{\pi_*E}$. Write $\pi_*\EE$ for the $\Z_2$-graded geometric bundle $(\pi_*E, g^{\pi_*E}, \nabla^{\pi_*E, u})$.

\subsection{The Atiyah--Singer--Gorokhovsky--Lott approach}\label{s 3.2}

For the details of the ASGL approach to the local FIT without the kernel bundle assumption, we refer the readers to \cite[\S4, 5]{GL18} for the details.

Atiyah--Singer prove that there exist a finite rank complex vector bundle $H\to B$ and a linear map $s:H\to(\pi_*E)^-$ such that
$$\DD^{S\otimes E}_++s:(\pi_*E)^+\oplus H\to(\pi_*E)^-$$
is surjective \cite[Proposition 2.2]{AS71}. Here, $\DD^{S\otimes E}_++s$ is understood to be $(\DD^{S\otimes E}_+\oplus 0)+(0\oplus s)$. Note that $H\to B$ can always be taken as a trivial bundle. Such a pair $(H, s)$ is said to satisfy the AS property for $\DD^{S\otimes E}$.

Define maps $\Delta_\pm:(\pi_*E\oplus\wh{H})^\pm\to(\pi_*E\oplus\wh{H})^\mp$ by
$$\Delta_+=\begin{pmatrix} 0 & s \\ 0 & 0 \end{pmatrix}\quad\textrm{ and }\quad\Delta_-=\begin{pmatrix} 0 & 0 \\ s^* & 0 \end{pmatrix}.$$
The map $\Delta:\pi_*E\oplus\wh{H}\to\pi_*E\oplus\wh{H}$ defined by
\begin{equation}\label{eq 3.3}
\Delta=\begin{pmatrix} 0 & \Delta_- \\ \Delta_+ & 0 \end{pmatrix}
\end{equation}
is odd self-adjoint, and the same is true for the map
$$\DD^{S\otimes E; \Delta}:=(\DD^{S\otimes E}\oplus 0)+\Delta:\pi_*E\oplus\wh{H}\to\pi_*E\oplus\wh{H}.$$

With respect to the $\Z_2$-grading of $\pi_*E\oplus\wh{H}\to B$, $\DD^{S\otimes E; \Delta}$ can be written as
$$\DD^{S\otimes E; \Delta}=\begin{pmatrix} 0 & 0 & \DD^{S\otimes E}_- & 0 \\ 0 & 0 & s^* & 0 \\ \DD^{S\otimes E}_+ & s & 0 & 0 \\ 0 & 0 & 0 & 0 \end{pmatrix}.$$
In particular,
$$\DD^{S\otimes E; \Delta}_+=\begin{pmatrix} \DD^{S\otimes E}_+ & s \\ 0 & 0 \end{pmatrix}\quad\textrm{ and }\quad\DD^{S\otimes E; \Delta}_-=\begin{pmatrix} \DD^{S\otimes E}_- & 0 \\ s^* & 0 \end{pmatrix}.$$

Since $\DD^{S\otimes E}_++s$ is surjective, the vector bundle $\ker\big(\DD^{S\otimes E}_++s\big)\to B$ exists and is a finite rank subbundle of $(\pi_*E\oplus\wh{H})^+\to B$. Thus the vector bundle $\ker\big(\DD^{S\otimes E; \Delta}_+\big)\to B$ exists and
\begin{equation}\label{eq 3.4}
\ker\big(\DD^{S\otimes E; \Delta}_+\big)=\ker\big(\DD^{S\otimes E}_++s\big).
\end{equation}
On the other hand, the vector bundle $\ker\big(\DD^{S\otimes E; \Delta}_-\big)\to B$ exists and
\begin{equation}\label{eq 3.5}
\ker\big(\DD^{S\otimes E; \Delta}_-\big)\cong\cok\big(\DD^{S\otimes E; \Delta}_+\big)=\frac{(\pi_*E\oplus\wh{H})^-}{\im\big(\DD^{S\otimes E; \Delta}_+\big)}\cong\wh{H}^-.
\end{equation}
By (\ref{eq 3.4}) and (\ref{eq 3.5}), the vector bundle
$$\ker\big(\DD^{S\otimes E; \Delta}\big)\to B$$
exists and is a finite rank $\Z_2$-graded subbundle of $\pi_*E\oplus\wh{H}\to B$. There is an orthogonal decomposition
\begin{equation}\label{eq 3.6}
\pi_*E\oplus\wh{H}=\im\big(\DD^{S\otimes E; \Delta}\big)\oplus\ker\big(\DD^{S\otimes E; \Delta}\big).
\end{equation}
The analytic index $\ind^a([E])$ in $K$-theory is given by
$$\ind^a([E])=\big[\ker\big(\DD^{S\otimes E; \Delta}\big)^+\big]-\big[\ker\big(\DD^{S\otimes E; \Delta}\big)^-\big].$$

By putting a Hermitian metric $g^H$ and a unitary connection $\nabla^H$ on $H\to B$, we write $\HH$ for the geometric bundle $(H, g^H, \nabla^H)$. Denote by $g^{\ker(\DD^{S\otimes E; \Delta})}$ the $\Z_2$-graded Hermitian metric on $\ker\big(\DD^{S\otimes E; \Delta}\big)\to B$ inherited from $g^{\pi_*E}\oplus g^{\wh{H}}$.

Define a rescaled Bismut superconnection on $\pi_*E\oplus\wh{H}\to B$ by
\begin{equation}\label{eq 3.7}
\wt{\bbb}^{E; \Delta}_t=\sqrt{t}\DD^{S\otimes E; \Delta}+\big(\nabla^{\pi_*E, u}\oplus\nabla^{\wh{H}}\big)-\frac{c(T)}{4\sqrt{t}},
\end{equation}
where $T$ is given by (\ref{eq 3.2}) and the last term on the right-hand side of (\ref{eq 3.7}) is understood to be $\frac{c(T)}{4\sqrt{t}}\oplus 0$.\footnote{This convention applies to other Bismut--Cheeger eta forms.}

Let $P:\pi_*E\oplus\wh{H}\to\ker\big(\DD^{S\otimes E; \Delta}\big)$ be the even orthogonal projection with respect to (\ref{eq 3.6}). The connection on $\ker\big(\DD^{S\otimes E; \Delta}\big)\to B$ defined by
$$\nabla^{\ker(\DD^{S\otimes E; \Delta})}=P\circ\big(\nabla^{\pi_*E, u}\oplus\nabla^{\wh{H}}\big)\circ P$$
is $\Z_2$-graded and unitary with respect to $g^{\ker(\DD^{S\otimes E; \Delta})}$. Write $\KER\big(\DD^{S\otimes E; \Delta}\big)$ for the $\Z_2$-graded geometric bundle
$$\big(\ker\big(\DD^{S\otimes E; \Delta}\big), g^{\ker(\DD^{S\otimes E; \Delta})}, \nabla^{\ker(\DD^{S\otimes E; \Delta})}\big).$$
Henceforth, we call $\HH$ the geometric bundle defining $\KER\big(\DD^{S\otimes E; \Delta}\big)$.

By \cite[Theorem 10.23 and Corollary 9.22]{BGV},
\begin{align}
\lim_{t\to 0}\ch(\wt{\bbb}^{E; \Delta}_t)&=\int_{X/B}\todd(\nabla^{T^VX})\wedge\ch(\nabla^E),\label{eq 3.8}\\
\lim_{t\to\infty}\ch(\wt{\bbb}^{E; \Delta}_t)&=\ch(\nabla^{\ker(\DD^{S\otimes E; \Delta})}).\label{eq 3.9}
\end{align}
Define the Bismut--Cheeger eta form associated to $\wt{\bbb}^{E; \Delta}_t$ by
\begin{equation}\label{eq 3.10}
\wt{\eta}(\EE, \bpi, \HH, s)=\frac{1}{\sqrt{\pi}}\int^\infty_0\str\bigg(\frac{d\wt{\bbb}^{E; \Delta}_t}{dt}e^{-\frac{1}{2\pi i}(\wt{\bbb}^{E; \Delta}_t)^2}\bigg)dt.
\end{equation}
By (\ref{eq 3.8}) and (\ref{eq 3.9}), the local FIT for $\DD^{S\otimes E}$ in the ASGL approach states that
$$d\wt{\eta}(\EE, \bpi, \HH, s)=\int_{X/B}\todd(\nabla^{T^VX})\wedge\ch(\nabla^E)-\ch(\nabla^{\ker(\DD^{S\otimes E; \Delta})}).$$

\subsection{The Mi\v s\v cenko--Fomenko--Freed--Lott approach}\label{s 3.3}

For the details of the MFFL approach to the local FIT without the kernel bundle assumption, we refer the readers to \cite[\S7.12]{FL10} for the details.

Mi\v s\v cenko--Fomenko prove that there exist finite rank subbundles $L^\pm\to B$ and closed complementary subbundles $K^\pm\to B$ of $(\pi_*E)^\pm\to B$ such that
\begin{equation}\label{eq 3.11}
(\pi_*E)^+=K^+\oplus L^+,\qquad(\pi_*E)^-=K^-\oplus L^-,
\end{equation}
$\DD^{S\otimes E}_+:(\pi_*E)^+\to(\pi_*E)^-$ is block diagonal as a map with respect to (\ref{eq 3.11}) and $\DD^{S\otimes E}_+|_{K^+}:K^+\to K^-$ is an isomorphism \cite[p.96-97]{MF79} (see also \cite[Lemma 7.13]{FL10}).

Given complex vector bundles $L^\pm\to B$ satisfying the above conditions, we say the $\Z_2$-graded complex vector bundle $L\to B$, defined by $L=L^+\oplus L^-$, satisfies the MF property for $\DD^{S\otimes E}$. If $L\to B$ satisfies the MF property for $\DD^{S\otimes E}$, then the analytic index $\ind^a([E])$ in $K$-theory is defined to be
$$\ind^a([E])=[L^+]-[L^-].$$
It is proved that $\ind^a([E])$ does not depend on the choice of $L\to B$ satisfying the MF property for $\DD^{S\otimes E}$ \cite[p.96-97]{MF79}.

Let $g^L$ be the $\Z_2$-graded Hermitian metric on $L\to B$ inherited from $g^{\pi_*E}$. Denote by $P:\pi_*E\to L$ the $\Z_2$-graded projection map with respect to (\ref{eq 3.11}). The connection on $L\to B$ defined by
$$\nabla^L:=P\circ\nabla^{\pi_*E, u}\circ P.$$
is $\Z_2$-graded and unitary with respect to $g^L$. Write $\LL$ for the $\Z_2$-graded geometric bundle
$$(L, g^L, \nabla^L).$$

Given a $\Z_2$-graded complex vector bundle $L\to B$ satisfying the MF property for $\DD^{S\otimes E}$, consider the infinite rank $\Z_2$-graded complex vector bundle $\pi_*E\oplus L^{\op}\to B$. Let $i_-:L^-\to(\pi_*E)^-$ be the inclusion map and $z\in\C$. Define a map $\DDD^{S\otimes E}_+(z):(\pi_*E\oplus L^{\op})^+\to(\pi_*E\oplus L^{\op})^-$ by
\begin{equation}\label{eq 3.12}
\DDD^{S\otimes E}_+(z)=\begin{pmatrix} \DD^{S\otimes E}_+ & z i_- \\ z P_+ & 0 \end{pmatrix}.
\end{equation}
Note that $\DDD^{S\otimes E}_+(z)$ is invertible for all $z\neq 0$ \cite[Lemma 7.20]{FL10}. Define a map $\DDD^{S\otimes E}(z):\pi_*E\oplus L^{\op}\to\pi_*E\oplus L^{\op}$ by
\begin{equation}\label{eq 3.13}
\DDD^{S\otimes E}(z):=\begin{pmatrix} 0 & \DDD^{S\otimes E}_+(z)^* \\ \DDD^{S\otimes E}_+(z) & 0 \end{pmatrix}.
\end{equation}

Define a rescaled Bismut superconnection on $\pi_*E\oplus L^{\op}\to B$ by
\begin{equation}\label{eq 3.14}
\wh{\bbb}^E_t=\sqrt{t}\DDD^{S\otimes E}(\alpha(t))+\big(\nabla^{\pi_*E, u}\oplus\nabla^{L, \op}\big)-\frac{c(T)}{4\sqrt{t}},
\end{equation}
where $\alpha:[0, \infty)\to I$ is the smooth function given in \S\ref{s 2.1}. Since $\DDD^{S\otimes E}(\alpha(t))$ is invertible for $t\geq 1$,
\begin{equation}\label{eq 3.15}
\lim_{t\to\infty}\ch(\wh{\bbb}^E_t)=0.
\end{equation}
On the other hand, for $t\leq a$, $\wh{\bbb}^E_t$ decouples, i.e.
\begin{equation}\label{eq 3.16}
\wh{\bbb}^E_t=\bigg(\sqrt{t}\DD^{S\otimes E}+\nabla^{\pi_*E, u}-\frac{c(T)}{4\sqrt{t}}\bigg)\oplus\nabla^{L, \op}.
\end{equation}
By (\ref{eq 3.16}) and \cite[Theorem 10.23]{BGV},
\begin{equation}\label{eq 3.17}
\lim_{t\to 0}\ch(\wh{\bbb}^E_t)=\int_{X/B}\todd(\nabla^{T^VX})\wedge\ch(\nabla^E)-\ch(\nabla^L).
\end{equation}
Define the Bismut--Cheeger eta form associated to $\wh{\bbb}^E_t$ by
\begin{equation}\label{eq 3.18}
\wh{\eta}(\EE, \bpi, \LL)=\frac{1}{\sqrt{\pi}}\int^\infty_0\str\bigg(\frac{d\wh{\bbb}^E_t}{dt}e^{-\frac{1}{2\pi i}(\wh{\bbb}^E_t)^2}\bigg)dt.
\end{equation}
By (\ref{eq 3.15}) and (\ref{eq 3.17}), the local FIT for $\DD^{S\otimes E}$ in the MFFL approach states that
$$d\wh{\eta}(\EE, \bpi, \LL)=\int_{X/B}\todd(\nabla^{T^VX})\wedge\ch(\nabla^E)-\ch(\nabla^L).$$

We often employ the following special cases of (\ref{eq 3.13}), (\ref{eq 3.14}) and (\ref{eq 3.18}). Given a twisted spin$^c$ Dirac operator $\DD^{S\otimes E}$, let $(H, s)$ satisfy the AS property for $\DD^{S\otimes E}$ and $\HH$ a geometric bundle defining $\KER\big(\DD^{S\otimes E; \Delta}\big)$, where $\Delta$ is the odd self-adjoint map associated to $(H, s)$ as in (\ref{eq 3.3}). Analogous to (\ref{eq 3.12}), define a map
$$\DDD^{S\otimes E; \Delta}_+(z):\big(\pi_*E\oplus\wh{H}\oplus\ker\big(\DD^{S\otimes E; \Delta}\big)^{\op}\big)^+\to\big(\pi_*E\oplus\wh{H}\oplus\ker\big(\DD^{S\otimes E; \Delta}\big)^{\op}\big)^-$$
by
\begin{equation}\label{eq 3.19}
\DDD^{S\otimes E; \Delta}_+(z)=\begin{pmatrix} \DD^{S\otimes E; \Delta}_+ & zi_- \\ zP_+ & 0 \end{pmatrix},
\end{equation}
where $i_-:\ker\big(\DD^{S\otimes E; \Delta}\big)^-\to(\pi_*E\oplus\wh{H})^-$ and $P_+:(\pi_*E\oplus\wh{H})^+\to\ker\big(\DD^{S\otimes E; \Delta}\big)^+$ are the inclusion map and the even part of the projection map $P$ with respect to (\ref{eq 3.6}), respectively. Then the map
\begin{equation}\label{eq 3.20}
\DDD^{S\otimes E; \Delta}(z):\pi_*E\oplus\wh{H}\oplus\ker\big(\DD^{S\otimes E; \Delta}\big)^{\op}\to\pi_*E\oplus\wh{H}\oplus\ker\big(\DD^{S\otimes E; \Delta}\big)^{\op}
\end{equation}
defined by
$$\DDD^{S\otimes E; \Delta}(z)=\begin{pmatrix} 0 & \DDD^{S\otimes E; \Delta}_+(z)^* \\ \DDD^{S\otimes E; \Delta}_+(z) & 0 \end{pmatrix}$$
is odd self-adjoint. Define a rescaled Bismut superconnection on 
$$\pi_*E\oplus\wh{H}\oplus\ker\big(\DD^{S\otimes E; \Delta}\big)^{\op}\to B$$ 
by
\begin{equation}\label{eq 3.21}
\wh{\bbb}^{E; \Delta}_t=\sqrt{t}\DDD^{S\otimes E; \Delta}(\alpha(t))+\big(\nabla^{\pi_*E, u}\oplus\nabla^{\wh{H}}\oplus\nabla^{\ker(\DD^{S\otimes E; \Delta}), \op}\big)-\frac{c(T)}{4\sqrt{t}}.
\end{equation}
Similar to (\ref{eq 3.15}) and (\ref{eq 3.17}),
\begin{align}
\lim_{t\to\infty}\ch(\wh{\bbb}^{E; \Delta}_t)&=0,\label{eq 3.22}\\
\lim_{t\to 0}\ch(\wh{\bbb}^{E; \Delta}_t)&=\int_{X/B}\todd(\nabla^{T^VX})\wedge\ch(\nabla^E)-\ch(\nabla^{\ker(\DD^{S\otimes E; \Delta})})\label{eq 3.23}.
\end{align}
Define the Bismut--Cheeger eta form associated to $\wh{\bbb}^{E; \Delta}_t$ by
\begin{equation}\label{eq 3.24}
\wh{\eta}\big(\EE, \bpi, \KER\big(\DD^{S\otimes E; \Delta}\big)\big)=\frac{1}{\sqrt{\pi}}\int^\infty_0\str\bigg(\frac{d\wh{\bbb}^{E; \Delta}_t}{dt}e^{-\frac{1}{2\pi i}(\wh{\bbb}^{E; \Delta}_t)^2}\bigg)dt.
\end{equation}
By (\ref{eq 3.22}) and (\ref{eq 3.23}), the local FIT for $\DD^{S\otimes E; \Delta}$ in the MFFL approach states that
$$d\wh{\eta}\big(\EE, \bpi, \KER\big(\DD^{S\otimes E; \Delta}\big)\big)=\int_{X/B}\todd(\nabla^{T^VX})\wedge\ch(\nabla^E)-\ch(\nabla^{\ker(\DD^{S\otimes E; \Delta})}).$$

\section{ASGL meeting MFFL}\label{s 4}

In this section, we compare the ASGL and MFFL approaches to the local FIT without the kernel bundle assumption. We first recall the analytic indexes $\wt{\ind}^a_{\wh{K}}$ and $\ind^a_{\wh{K}}$ in differential $K$-theory defined by these approaches and prove a lemma relating the Bismut--Cheeger eta forms in these approaches in \S\ref{s 4.1}. We then prove that the analytic indexes $\wt{\ind}^a_{\wh{K}}$ and $\ind^a_{\wh{K}}$ coincide in \S\ref{s 4.2}.

\subsection{Analytic indexes in differential $K$-theory}\label{s 4.1}

The geometric model of the differential $K$-theory group $\wh{K}(X)$ given by Freed--Lott \cite[Definition 2.16]{FL10} is defined in terms of generators of the form $\E=(\EE, \omega)$, where $\EE$ is a geometric bundle over $X$ and $\omega\in\frac{\Omega^{\odd}(X)}{\im(d)}$. Two generators $\E$ and $\F$ are equal in $\wh{K}(X)$ if and only if there exist a geometric bundle $\VV$ and a morphism $\alpha:\EE\oplus\VV\to\FF\oplus\VV$ such that
$$\omega_E-\omega_F\equiv\CS\big(\nabla^E\oplus\nabla^V, \alpha^*(\nabla^F\oplus\nabla^V)\big).$$

Note that $\wh{K}(X)$ can also be described in terms of $\Z_2$-graded generators $\E=(\EE, \omega)$, where $\EE$ is $\Z_2$-graded. Two $\Z_2$-graded generators $\E$ and $\F$ are equal in $\wh{K}(X)$ if and if only there exist geometric bundles $\VV$ and $\WW$ and a $\Z_2$-graded morphism $\alpha:\EE\oplus\wh{\VV}\to\FF\oplus\wh{\WW}$ such that
\begin{equation}\label{eq 4.1}
\omega_E-\omega_F\equiv\CS\big(\nabla^E\oplus\nabla^{\wh{V}}, \alpha^*(\nabla^F\oplus\nabla^{\wh{W}})\big).
\end{equation}

Let $\pi:X\to B$ be a submersion with closed, oriented and spin$^c$ fibers of even dimension, equipped with a Riemannian and differential spin$^c$ structure $\bpi$. Let $\E=(\EE, \omega)$ be a generator of $\wh{K}(X)$. Denote by $\DD^{S\otimes E}$ the twisted spin$^c$ Dirac operator associated to $(\EE, \bpi)$. Let $(H, s)$ satisfy the AS property for $\DD^{S\otimes E}$ and $\HH$ a geometric bundle defining $\KER\big(\DD^{S\otimes E; \Delta}\big)$, where $\Delta$ is the odd self-adjoint map associated to $(H, s)$ as in (\ref{eq 3.3}). The element $\wt{\ind}^a_{\wh{K}}(\E)$ of $\wh{K}(B)$ defined by
\begin{equation}\label{eq 4.2}
\wt{\ind}^a_{\wh{K}}(\E)=\bigg(\KER\big(\DD^{S\otimes E; \Delta}\big), \int_{X/B}\todd(\nabla^{T^VX})\wedge\omega+\wt{\eta}(\EE, \bpi, \HH, s)\bigg)
\end{equation}
is a special case of \cite[Definition 7]{GL18}. \emph{A prior} $\wt{\ind}^a_{\wh{K}}(\E)$ depends on the choices of the pair $(H, s)$ satisfying the AS property for $\DD^{S\otimes E}$ and the geometric bundle $\HH$.

Recall that the analytic index $\ind^a_{\wh{K}}(\E)$ in differential $K$-theory \cite[(7.28)]{FL10} is defined to be
\begin{equation}\label{eq 4.3}
\ind^a_{\wh{K}}(\E)=\bigg(\LL, \int_{X/B}\todd(\nabla^{T^VX})\wedge\omega+\wh{\eta}(\EE, \bpi, \LL)\bigg),
\end{equation}
where $\LL=(L, g^L, \nabla^L)$ is a $\Z_2$-graded geometric bundle over $B$ so that $L\to B$ satisfies the MF property for $\DD^{S\otimes E}$. Note that $\ind^a_{\wh{K}}(\E)$ is independent of the choice of $\LL$ for which $L\to B$ satisfies the MF property for $\DD^{S\otimes E}$ \cite[(3) of Corollary 7.36]{FL10} (see also \cite[Corollary 1]{H20}). Moreover, $$\ind^a_{\wh{K}}:\wh{K}(X)\to\wh{K}(B)$$
is a well defined group homomorphism \cite[(2) of Corollary 7.36]{FL10} (see also \cite[Proposition 1.4]{H23}).

The following lemma roughly says the Bismut--Cheeger eta form $\wt{\eta}$ (\ref{eq 3.10}) in the ASGL approach is a special case of the Bismut--Cheeger eta form $\wh{\eta}$ (\ref{eq 3.24}) in the MFFL approach.
\begin{lemma}\label{lemma 4.1}
Let $\pi:X\to B$ be a submersion with closed, oriented and spin$^c$ fibers of even dimension, equipped with a Riemannian and differential spin$^c$ structure $\bpi$. Let $\EE=(E, g^E, \nabla^E)$ be a geometric bundle over $X$. Denote by $\DD^{S\otimes E}$ the twisted spin$^c$ Dirac operator associated to $(\EE, \bpi)$. Let $(H, s)$ satisfy the AS property for $\DD^{S\otimes E}$ and $\HH$ a geometric bundle defining $\KER\big(\DD^{S\otimes E; \Delta}\big)$, where $\Delta$ is the odd self-adjoint map associated to $(H, s)$ as in (\ref{eq 3.3}). Then
\begin{equation}\label{eq 4.4}
\wh{\eta}\big(\EE, \bpi, \KER\big(\DD^{S\otimes E; \Delta}\big)\big)\equiv\wt{\eta}(\EE, \bpi, \HH, s).
\end{equation}
\end{lemma}

Lemma \ref{lemma 4.1} is used several times in the remainder of this paper.
\begin{proof}
Let $t\in(0, a)$. By (\ref{eq 3.7}), (\ref{eq 3.21}) and the definition of $\alpha$,
\begin{equation}\label{eq 4.5}
\wh{\bbb}^{E; \Delta}_t=\wt{\bbb}^{E; \Delta}_t\oplus\nabla^{\ker(\DD^{S\otimes E; \Delta}), \op}.
\end{equation}
Let $T\in(1, \infty)$. By (\ref{eq 2.9}) to (\ref{eq 2.12}) and (\ref{eq 4.5}),
\begin{displaymath}
\begin{split}
&\CS\big(\wh{\bbb}^{E; \Delta}_t, \wh{\bbb}^{E; \Delta}_T\big)-\CS\big(\wt{\bbb}^{E; \Delta}_t, \wt{\bbb}^{E; \Delta}_T\big)\\
&\equiv\CS\big(\wh{\bbb}^{E; \Delta}_t, \wh{\bbb}^{E; \Delta}_T\big)-\CS\big(\wt{\bbb}^{E; \Delta}_t, \wt{\bbb}^{E; \Delta}_T\big)-\CS\big(\nabla^{\ker(\DD^{S\otimes E; \Delta}), \op}, \nabla^{\ker(\DD^{S\otimes E; \Delta}), \op}\big)\\
&\equiv-\CS\big(\wh{\bbb}^{E; \Delta}_T, \wh{\bbb}^{E; \Delta}_t\big)-\CS\big(\wt{\bbb}^{E; \Delta}_t\oplus\nabla^{\ker(\DD^{S\otimes E; \Delta}), \op}, \wt{\bbb}^{E; \Delta}_T\oplus\nabla^{\ker(\DD^{S\otimes E; \Delta}), \op}\big)\\
&\equiv-\CS\big(\wh{\bbb}^{E; \Delta}_T, \wt{\bbb}^{E; \Delta}_T\oplus\nabla^{\ker(\DD^{S\otimes E; \Delta}), \op}\big).
\end{split}
\end{displaymath}
By letting $t\to 0$ and $T\to\infty$ in above,
$$\wh{\eta}\big(\EE, \bpi, \KER\big(\DD^{S\otimes E; \Delta}\big)\big)-\wt{\eta}(\EE, \bpi, \HH, s)\equiv-\lim_{T\to\infty}\CS\big(\wh{\bbb}^{E; \Delta}_T, \wt{\bbb}^{E; \Delta}_T\oplus\nabla^{\ker(\DD^{S\otimes E; \Delta}), \op}\big).$$
By the estimates in \cite[\S9.3]{BGV},
$$\lim_{T\to\infty}\CS\big(\wh{\bbb}^{E; \Delta}_T, \wt{\bbb}^{E; \Delta}_T\oplus\nabla^{\ker(\DD^{S\otimes E; \Delta}), \op}\big)=0.$$
The (\ref{eq 4.4}) holds.
\end{proof}

\subsection{The equivalence of the analytic indexes}\label{s 4.2}

We prove that the analytic indexes $\wt{\ind}^a_{\wh{K}}$ and $\ind^a_{\wh{K}}$ defined by the ASGL and MFFL approaches, respectively, are equal.
\begin{prop}\label{prop 4.1}
Let $\pi:X\to B$ be a submersion with closed, oriented and spin$^c$ fibers of even dimension, equipped with a Riemannian and differential spin$^c$ structure $\bpi$. Then
$$\wt{\ind}^a_{\wh{K}}=\ind^a_{\wh{K}}:\wh{K}(X)\to\wh{K}(B).$$
\end{prop}
\begin{proof}
Let $\E=(\EE, \omega)$ be a generator of $\wh{K}(X)$. Denote by $\DD^{S\otimes E}$ the twisted spin$^c$ Dirac operator associated to $(\EE, \bpi)$. Let $(H, s)$ satisfy the AS property for $\DD^{S\otimes E}$ and $\HH$ a geometric bundle defining $\KER\big(\DD^{S\otimes E; \Delta}\big)$, where $\Delta$ is the odd self-adjoint map associated to $(H, s)$ as in (\ref{eq 3.3}). We show that
\begin{equation}\label{eq 4.6}
\wt{\ind}^a_{\wh{K}}(\E)=\ind^a_{\wh{K}}(\E).
\end{equation}
Since $\ind^a_{\wh{K}}(\E)$ does not depend on the choices of $(H, s)$ and $\HH$, and $\ind^a_{\wh{K}}:\wh{K}(X)\to\wh{K}(B)$ is a well defined group homomorphism (by \cite[Proposition 1.4]{H23}), it follows from (\ref{eq 4.6}) that the same are true for $\wt{\ind}^a_{\wh{K}}$.

Let $\LL=(L, g^L, \nabla^L)$ be a $\Z_2$-graded geometric bundle over $B$ for which $L\to B$ satisfies the MF property for $\DD^{S\otimes E}$. To prove (\ref{eq 4.6}), by (\ref{eq 4.1}) to (\ref{eq 4.3}) it suffices to show that there exist geometric bundles $\VV_0$ and $\VV_1$ over $B$ and a $\Z_2$-graded morphism
$$\wt{h}:\LL\oplus\wh{\VV}_0\to\KER\big(\DD^{S\otimes E; \Delta}\big)\oplus\wh{\VV}_1$$
such that
\begin{equation}\label{eq 4.7}
\wh{\eta}(\EE, \bpi, \LL)-\wt{\eta}(\EE, \bpi, \HH, s)\equiv\CS\big(\nabla^L\oplus\nabla^{\wh{V}_0}, \wt{h}^*(\nabla^{\ker(\DD^{S\otimes E; \Delta})}\oplus\nabla^{\wh{V}_1})\big).
\end{equation}

Let $\wt{\bpi}$ be the Riemannian and differential spin$^c$ structure on $\wt{\pi}:\wt{X}\to\wt{B}$ obtained by pulling back $\bpi$ via $p_X$. Define a $\Z_2$-graded geometric bundle $\bee$ over $\wt{X}$ by $\bee=p_X^*\EE$. Denote by  $\DD^{S\otimes\e}$ the twisted spin$^c$ Dirac operator associated to $(\bee, \wt{\bpi})$. Define a geometric bundle $\bhh$ over $\wt{B}$ by $\bhh=p_B^*\HH$.

Fix a smooth increasing function $\chi:I\to I$ satisfying $\chi(t)=0$ for $t\in[0, \frac{1}{4}]$ and $\chi(t)=1$ for $t\in[\frac{3}{4}, 1]$. The map $\wt{\Delta}:\wt{\pi}_*\e\oplus\wh{\hh}\to\wt{\pi}_*\e\oplus\wh{\hh}$ defined by
$$\wt{\Delta}(x, t)=(\chi(t)\Delta(x), t)$$
is odd self-adjoint, and the same is true for $\DD^{S\otimes\e; \wt{\Delta}}:=\DD^{S\otimes\e}+\wt{\Delta}$.

By applying \cite[Lemma 2.3]{MF79} to $\DD^{S\otimes\e; \wt{\Delta}}$, let $\bLLL=(\LLL, g^\LLL, \nabla^\LLL)$ be a $\Z_2$-graded geometric bundle over $\wt{B}$ for which $\LLL\to\wt{B}$ satisfies the MF property for $\DD^{S\otimes\e; \wt{\Delta}}$, i.e. there exists a $\Z_2$-graded closed complementary subbundle $\KKK\to\wt{B}$ of $\wt{\pi}_*\e\oplus\wh{\hh}\to\wt{B}$ to $\LLL\to\wt{B}$:
\begin{equation}\label{eq 4.8}
\begin{split}
(\wt{\pi}_*\e\oplus\wh{\hh})^+&=\KKK^+\oplus\LLL^+,\\
(\wt{\pi}_*\e\oplus\wh{\hh})^-&=\KKK^-\oplus\LLL^-,
\end{split}
\end{equation}
$\DD^{S\otimes\e; \wt{\Delta}}_+:(\wt{\pi}_*\e\oplus\wh{\hh})^+\to(\wt{\pi}_*\e\oplus\wh{\hh})^-$ is block diagonal as a map with respect to (\ref{eq 4.8}) and the restriction $\DD^{S\otimes\e; \wt{\Delta}}_+|_{\KKK^+}:\KKK^+\to\KKK^-$ is an isomorphism. Note that
$$i_{B, 0}^*(\wt{\pi}_*\e\oplus\wh{\hh})\cong i_{B, 1}^*(\wt{\pi}_*\e\oplus\wh{\hh})=\pi_*E\oplus\wh{H}.$$
Let $j\in\set{0, 1}$. Define a $\Z_2$-graded geometric bundle $\wt{\LL}_j$ over $B$ by $\wt{\LL}_j=i_{B, j}^*\bLLL$ and a complex vector bundle $K_j\to B$ by $K_j=i_{B, j}^*\KKK$. Note that
\begin{equation}\label{eq 4.9}
\wt{L}_0=i_{B, 0}^*\LLL\cong i_{B, 1}^*\LLL=\wt{L}_1
\end{equation}
as $\Z_2$-graded complex vector bundles. Since
\begin{equation}\label{eq 4.10}
\begin{split}
\DD_0:=\DD^{S\otimes\e; \wt{\Delta}}|_{i_{B, 0}^*(\wt{\pi}_*\e\oplus\wh{\hh})}&=\DD^{S\otimes E}\oplus 0,\\
\DD_1:=\DD^{S\otimes\e; \wt{\Delta}}|_{i_{B, 1}^*(\wt{\pi}_*\e\oplus\wh{\hh})}&=\DD^{S\otimes E; \Delta},
\end{split}
\end{equation}
it follows from (\ref{eq 4.8}) to (\ref{eq 4.10}) that
\begin{equation}\label{eq 4.11}
(\pi_*E\oplus\wh{H})^+=K_j^+\oplus\wt{L}_j^+,\qquad(\pi_*E\oplus\wh{H})^-=K_j^-\oplus\wt{L}_j^-,
\end{equation}
$\DD_{j, +}:(\pi_*E\oplus\wh{H})^+\to(\pi_*E\oplus\wh{H})^-$ is block diagonal as a map with respect to (\ref{eq 4.11}) and the restriction $\DD_{j, +}|_{K_j^+}:K_j^+\to K_j^-$ is an isomorphism. Thus
\begin{align}
\wt{L}_0\to B&\textrm{ satisfies the MF property for }\DD_0=\DD^{S\otimes E}\oplus 0,\label{eq 4.12}\\
\wt{L}_1\to B&\textrm{ satisfies the MF property for }\DD_1=\DD^{S\otimes E; \Delta}.\label{eq 4.13}
\end{align}

Since $L\to B$ satisfies the MF property for $\DD^{S\otimes E}$, it follows that
\begin{equation}\label{eq 4.14}
L\oplus\wh{H}\to B\textrm{ satisfies the MF property for }\DD^{S\otimes E}\oplus 0.
\end{equation}
By (\ref{eq 4.12}), (\ref{eq 4.14}) and \cite[Remark 4.6]{H23}, there exist complex vector bundles $V_0\to B$ and $V_1\to B$ such that
\begin{equation}\label{eq 4.15}
(L\oplus\wh{H})\oplus\wh{V}_0\cong\wt{L}_0\oplus\wh{V}_1
\end{equation}
as $\Z_2$-graded complex vector bundles. On the other hand, since
$$\ker\big(\DD^{S\otimes E; \Delta}\big)\to B\textrm{ satisfies the MF property for }\DD^{S\otimes E; \Delta},$$
it follows from (\ref{eq 4.13})  and \cite[Remark 4.6]{H23} that there exist complex vector bundles $W_0\to B$ and $W_1\to B$ such that
\begin{equation}\label{eq 4.16}
\wt{L}_1\oplus\wh{W}_0\cong\ker\big(\DD^{S\otimes E; \Delta}\big)\oplus\wh{W}_1
\end{equation}
as $\Z_2$-graded complex vector bundles. By (\ref{eq 4.15}), (\ref{eq 4.9}) and (\ref{eq 4.16}),
\begin{displaymath}
\begin{split}
L\oplus\wh{H}\oplus\wh{V}_0\oplus\wh{W}_0&\cong\wt{L}_0\oplus\wh{V}_1\oplus\wh{W}_0\\
&\cong\wt{L}_1\oplus\wh{V}_1\oplus\wh{W}_0\\
&\cong\ker\big(\DD^{S\otimes E; \Delta}\big)\oplus\wh{W}_1\oplus\wh{V}_1
\end{split}
\end{displaymath}
as $\Z_2$-graded complex vector bundles. By writing
\begin{displaymath}
\begin{split}
V_0&\textrm{ for }H\oplus V_0\oplus W_0,\\
V_1&\textrm{ for }W_1\oplus V_1,\\
W&\textrm{ for }V_1\oplus W_0,
\end{split}
\end{displaymath}
the above isomorphisms become
\begin{equation}\label{eq 4.17}
L\oplus\wh{V}_0\cong\wt{L}_0\oplus\wh{W}\cong\wt{L}_1\oplus\wh{W}\cong\ker\big(\DD^{S\otimes E; \Delta}\big)\oplus\wh{V}_1.
\end{equation}

By (\ref{eq 4.9}) and \cite[(2) of Remark 2.1]{H23}, let
$$h_1:\wt{\LL}_0\to\wt{\LL}_1$$
be a $\Z_2$-graded morphism. Put Hermitian metrics and unitary connections on $V_j\to B$ and $W\to B$, and denote by $\VV_j$ and $\WW$ the resulting geometric bundles, respectively. By (\ref{eq 4.17}) and \cite[(2) of Remark 2.1]{H23}, let
\begin{displaymath}
\begin{split}
h_0&:\wt{\LL}_0\oplus\wh{\WW}\to\LL\oplus\wh{\VV}_0,\\ 
h_2&:\wt{\LL}_1\oplus\wh{\WW}\to\KER\big(\DD^{S\otimes E; \Delta}\big)\oplus\wh{\VV}_1
\end{split}
\end{displaymath}
be $\Z_2$-graded morphisms. Define a $\Z_2$-graded morphism
$$\wt{h}:\LL\oplus\wh{\VV}_0\to\KER\big(\DD^{S\otimes E; \Delta}\big)\oplus\wh{\VV}_1$$
to be the composition
\begin{center}
\begin{tikzcd}
\LL\oplus\wh{\VV}_0 \arrow{r}{h_0^{-1}} & \wt{\LL}_0\oplus\wh{\WW} \arrow{rr}{h_1\oplus\id_{\wh{\WW}}} & &  \wt{\LL}_1\oplus\wh{\WW} \arrow{r}{h_2} & \KER(\DD^{S\otimes E; \Delta})\oplus\wh{\VV}_1
\end{tikzcd}
\end{center}
Note that
\begin{equation}\label{eq 4.18}
h_0^*\big(\nabla^L\oplus\nabla^{\wh{V}_0}\big)\quad\textrm{ and }\quad h_0^*\wt{h}^*\big(\nabla^{\ker(\DD^{S\otimes E; \Delta})}\oplus\nabla^{\wh{V}_1}\big)
\end{equation}
are $\Z_2$-graded unitary connections on $\wt{L}_0\oplus\wh{W}\to B$.

Define a $\Z_2$-graded geometric bundle $\bRRR=(\RRR, g^\RRR, \nabla^\RRR)$ over $\wt{B}$ by
$$\RRR=p_B^*\big(\wt{L}_0\oplus\wh{W}\big),\qquad g^\RRR=p_B^*\big(g^{\wt{L}_0}\oplus g^{\wh{W}}\big)$$
and $\nabla^\RRR$ by (\ref{eq 2.6}) and (\ref{eq 2.7}) with respect to (\ref{eq 4.18}). Note that
\begin{equation}\label{eq 4.19}
i_{B, 0}^*\nabla^\RRR=h_0^*\big(\nabla^L\oplus\nabla^{\wh{V}_0}\big)\quad\textrm{ and }\quad i_{B, 1}^*\nabla^\RRR=h_0^*\wt{h}^*\big(\nabla^{\ker(\DD^{S\otimes E; \Delta})}\oplus\nabla^{\wh{V}_1}\big).
\end{equation}

Define a geometric bundle $\bWWW$ over $\wt{B}$ by $\bWWW=p_B^*\WW$. Since $i_{B, 0}\circ p_B:\wt{B}\to\wt{B}$ is smoothly homotopic to $\id_{\wt{B}}$, it follows that $p_B^*\wt{L}_0\cong\LLL$. Thus
$$\LLL\oplus\wh{\WWW}\cong p_B^*\big(\wt{L}_0\oplus\wh{W}\big)\cong\RRR$$
as $\Z_2$-graded complex vector bundles. By \cite[(2) of Remark 2.1]{H23}, let $\varphi:\bLLL\to p_B^*\wt{\LL}_0$ be a $\Z_2$-graded morphism. Define a $\Z_2$-graded morphism $\wt{\varphi}:\bLLL\oplus\wh{\bWWW}\to\bRRR$ by $\wt{\varphi}=\varphi\oplus\id_{\wh{\bWWW}}$ and a $\Z_2$-graded unitary connection on $\LLL\oplus\wh{\WWW}\to\wt{B}$ by
\begin{equation}\label{eq 4.20}
\wt{\nabla}^{\LLL\oplus\wh{\WWW}}=\wt{\varphi}^*\nabla^\RRR.
\end{equation}
Henceforth, $\bLLL\oplus\wh{\bWWW}$ is regarded as the $\Z_2$-graded geometric bundle
$$\big(\LLL\oplus\wh{\WWW}, g^\LLL\oplus g^{\wh{\WWW}}, \wt{\nabla}^{\LLL\oplus\wh{\WWW}}\big).$$

Let $\PP_+:(\wt{\pi}_*\e\oplus\wh{\hh})^+\to\LLL^+$ and $\I_-:\LLL^-\to(\wt{\pi}_*\e\oplus\wh{\hh})^-$ be the projection and the inclusion maps with respect to (\ref{eq 4.8}), respectively. Let $z\in\C$. As in (\ref{eq 3.19}), define a map $\DDD^{S\otimes\e; \wt{\Delta}}_+(z):(\wt{\pi}_*\e\oplus\wh{\hh}\oplus\LLL^{\op})^+\to(\wt{\pi}_*\e\oplus\wh{\hh}\oplus\LLL^{\op})^-$ by
$$\DDD^{S\otimes\e; \wt{\Delta}}_+(z)=\begin{pmatrix} \DD^{S\otimes\e; \wt{\Delta}}_+ & z\I_- \\ z\PP_+ & 0 \end{pmatrix}.$$
Then the map $\DDD^{S\otimes\e; \wt{\Delta}}(z):\wt{\pi}_*\e\oplus\wh{\hh}\oplus\LLL^{\op}\to\wt{\pi}_*\e\oplus\wh{\hh}\oplus\LLL^{\op}$ defined as in (\ref{eq 3.20}) is odd self-adjoint. Define a map $s_{\wh{\WWW}}:\wh{\WWW}\to\wh{\WWW}$ by (\ref{eq 2.2}) and a rescaled Bismut superconnection on $\wt{\pi}_*\e\oplus\wh{\hh}\oplus(\LLL\oplus\wh{\WWW})^{\op}\to\wt{B}$ by
\begin{equation}\label{eq 4.21}
\wh{\bbb}_t=\sqrt{t}\big(\DDD^{S\otimes\e; \wt{\Delta}}(\alpha(t))\oplus\alpha(t)s_{\wh{\WWW}}\big)+\big(\nabla^{\wt{\pi}_*\e, u}\oplus\nabla^{\wh{\hh}}\oplus\wt{\nabla}^{\LLL\oplus\wh{\WWW}, \op}\big)-\frac{c(\wt{T})}{4\sqrt{t}}.
\end{equation}
Since $(\wh{\bbb}_t)_{[0]}$ is invertible for $t\geq 1$,
\begin{equation}\label{eq 4.22}
\lim_{t\to\infty}\ch(\wh{\bbb}_t)=0.
\end{equation}
On the other hand, for $t\leq a$, $\wh{\bbb}^{\e; \wt{\Delta}}_t$ decouples, i.e.
\begin{displaymath}
\begin{split}
\wh{\bbb}_t&=\bigg(\sqrt{t}\DD^{S\otimes\e; \wt{\Delta}}+\big(\nabla^{\wt{\pi}_*\e, u}\oplus\nabla^{\wh{\hh}}\big)-\frac{c(\wt{T})}{4\sqrt{t}}\bigg)\oplus\wt{\nabla}^{\LLL\oplus\wh{\WWW}, \op}\\
&=\wt{\bbb}^{\e; \wt{\Delta}}_t\oplus\wt{\nabla}^{\LLL\oplus\wh{\WWW}, \op},
\end{split}
\end{displaymath}
where $\wt{\bbb}^{\e; \wt{\Delta}}_t$ is defined by (\ref{eq 3.7}). By (\ref{eq 3.23}) and (\ref{eq 4.20}),
\begin{equation}\label{eq 4.23}
\begin{aligned}[b]
\lim_{t\to 0}\ch(\wh{\bbb}_t)&=\lim_{t\to 0}\ch(\wt{\bbb}^{\e; \wt{\Delta}}_t)-\ch(\wt{\nabla}^{\LLL\oplus\wh{\WWW}})\\
&=\int_{\wt{X}/\wt{B}}\todd(\nabla^{T^V\wt{X}})\wedge\ch(\nabla^\e)-\ch(\nabla^\RRR).
\end{aligned}
\end{equation}
Denote by $\wh{\eta}(\bee, \wt{\bpi}, \bLLL\oplus\wh{\bWWW})$ the Bismut--Cheeger eta form associated to $\wh{\bbb}_t$. By (\ref{eq 4.22}) and (\ref{eq 4.23}),
\begin{equation}\label{eq 4.24}
d\wh{\eta}(\bee, \wt{\bpi}, \bLLL\oplus\wh{\bWWW})=\int_{\wt{X}/\wt{B}}\todd(\nabla^{T^V\wt{X}})\wedge\ch(\nabla^\e)-\ch(\nabla^\RRR).
\end{equation}
By (\ref{eq 2.1}), (\ref{eq 4.24}), (\ref{eq 4.19}) and (\ref{eq 2.14}),
\begin{displaymath}
\begin{split}
&i_{B, 1}^*\wh{\eta}(\bee, \wt{\bpi}, \bLLL\oplus\wh{\bWWW})-i_{B, 0}^*\wh{\eta}(\bee, \wt{\bpi}, \bLLL\oplus\wh{\bWWW})\\
&\equiv-\int_{\wt{B}/B}d^{\wt{B}}\wh{\eta}(\bee, \wt{\bpi}, \bLLL\oplus\wh{\bWWW})\\
&\equiv-\int_{\wt{B}/B}\bigg(\int_{\wt{X}/\wt{B}}\todd(\nabla^{T^V\wt{X}})\wedge\ch(\nabla^\e)-\ch(\nabla^\RRR)\bigg)\\
&\equiv-\int_{\wt{B}/B}\int_{\wt{X}/\wt{B}}\todd(\nabla^{T^V\wt{X}})\wedge\ch(\nabla^\e)-\CS\big(h_0^*(\nabla^L\oplus
\nabla^{\wh{V}_0}), h_0^*\wt{h}^*(\nabla^{\ker(\DD^{S\otimes E; \Delta})}\oplus\nabla^{\wh{V}_1})\big)\\
&\equiv-\int_{\wt{B}/B}\int_{\wt{X}/\wt{B}}\todd(\nabla^{T^V\wt{X}})\wedge\ch(\nabla^\e)-\CS\big(\nabla^L\oplus
\nabla^{\wh{V}_0}, \wt{h}^*(\nabla^{\ker(\DD^{S\otimes E; \Delta})}\oplus\nabla^{\wh{V}_1})\big).
\end{split}
\end{displaymath}
Since $\todd(\nabla^{T^V\wt{X}})=p_X^*\todd(\nabla^{T^VX})$, it follows that
\begin{equation}\label{eq 4.25}
\begin{aligned}[b]
\int_{\wt{B}/B}\int_{\wt{X}/\wt{B}}\todd(\nabla^{T^V\wt{X}})\wedge\ch(\nabla^\e)&\equiv\int_{X/B}\int_{\wt{X}/X}p_X^*
\todd(\nabla^{T^VX})\wedge\ch(\nabla^\e)\\
&\equiv\int_{X/B}\todd(\nabla^{T^VX})\wedge\int_{\wt{X}/X}\ch(\nabla^\e)\\
&\equiv-\int_{X/B}\todd(\nabla^{T^VX})\wedge\CS(\nabla^E, \nabla^E)\\
&\equiv 0.
\end{aligned}
\end{equation}
Thus
\begin{equation}\label{eq 4.26}
i_{B, 0}^*\wh{\eta}(\bee, \wt{\bpi}, \bLLL\oplus\wh{\bWWW})-i_{B, 1}^*\wh{\eta}(\bee, \wt{\bpi}, \bLLL\oplus\wh{\bWWW})\equiv\CS\big(\nabla^L\oplus\nabla^{\wh{V}_0}, \wt{h}^*(\nabla^{\ker(\DD^{S\otimes E; \Delta})}\oplus\nabla^{\wh{V}_1})\big).
\end{equation}
To prove (\ref{eq 4.7}), by (\ref{eq 4.26}) it suffices to show that
\begin{align}
i_{B, 0}^*\wh{\eta}(\bee, \wt{\bpi}, \bLLL\oplus\wh{\bWWW})&\equiv\wh{\eta}(\EE, \bpi, \LL),\label{eq 4.27}\\
i_{B, 1}^*\wh{\eta}(\bee, \wt{\bpi}, \bLLL\oplus\wh{\bWWW})&\equiv\wt{\eta}(\EE, \bpi, \HH, s).\label{eq 4.28}
\end{align}

Define a $\Z_2$-graded morphism
$$f:\wt{\pi}_*\bee\oplus\wh{\bhh}\oplus(\bLLL\oplus\wh{\bWWW})^{\op}\to\wt{\pi}_*\bee\oplus\wh{\bhh}\oplus\bRRR^{\op}$$
defined by $f:=\id_{\wt{\pi}_*\bee}\oplus\id_{\wh{\bhh}}\oplus\wt{\varphi}^{\op}$. By (\ref{eq 4.20}) and (\ref{eq 4.21}),
\begin{equation}\label{eq 4.29}
(f^{-1})^*\wh{\bbb}_t=\sqrt{t}\big(\DDD^{S\otimes\e; \wt{\Delta}}(\alpha(t))\oplus\alpha(t)s_{\wh{\WWW}}\big)+\big(\nabla^{\wt{\pi}_*\e, u}\oplus\nabla^{\wh{\hh}}\oplus\nabla^{\RRR, \op}\big)-\frac{c(\wt{T})}{4\sqrt{t}}.
\end{equation}
Let $j\in\set{0, 1}$. Define a map $s_{\wh{V}_j}:\wh{V}_j\to\wh{V}_j$ by (\ref{eq 2.2}).

To prove (\ref{eq 4.27}), note that by (\ref{eq 4.29}), (\ref{eq 4.10}) and (\ref{eq 4.19}),
$$i_{B, 0}^*(f^{-1})^*\wh{\bbb}_t=\sqrt{t}\big(\DDD^{S\otimes E}(\alpha(t))\oplus\alpha(t)s_{\wh{W}}\big)+\big(\nabla^{\pi_*E, u}\oplus\nabla^{\wh{H}}\oplus(h_0^*(\nabla^L\oplus\nabla^{\wh{V}_0}))^{\op}\big)-\frac{c(T)}{4\sqrt{t}}.$$
Define a $\Z_2$-graded morphism
$$\wt{h}_0:\bpi_*\EE\oplus\wh{\HH}\oplus(\LL\oplus\wh{\VV}_0)^{\op}\to\bpi_*\EE\oplus\wh{\HH}\oplus(\wt{\LL}_0
\oplus\wh{\WW})^{\op}$$
by $\wt{h}_0:=\id_{\bpi_*\EE}\oplus\id_{\wh{\HH}}\oplus(h_0^{-1})^{\op}$. By (\ref{eq 3.14}) and (\ref{eq 2.3}),
\begin{equation}\label{eq 4.30}
\begin{aligned}[b]
\wt{h}_0^*i_{B, 0}^*(f^{-1})^*\wh{\bbb}_t&=\sqrt{t}\big(\DDD^{S\otimes E}(\alpha(t))\oplus\alpha(t)s_{\wh{V}_0}\big)+\big(\nabla^{\pi_*E, u}\oplus\nabla^{\wh{H}}\oplus\nabla^{L, \op}\oplus\nabla^{\wh{V}_0}\big)-\frac{c(T)}{4\sqrt{t}}\\
&=\wh{\bbb}^E_t\oplus\nabla^{\wh{H}}\oplus\aaa^{\wh{V}_0}_t.
\end{aligned}
\end{equation}
Let $t<T\in(0, \infty)$. By (\ref{eq 2.14}), (\ref{eq 4.30}), (\ref{eq 2.11}) and (\ref{eq 2.12}),
\begin{equation}\label{eq 4.31}
\begin{aligned}[b]
i_{B, 0}^*\CS\big(\wh{\bbb}_t, \wh{\bbb}_T\big)&\equiv i_{B, 0}^*\CS\big((f^{-1})^*\wh{\bbb}_t, (f^{-1})^*\wh{\bbb}_T\big)\\
&\equiv\CS\big(i_{B, 0}^*(f^{-1})^*\wh{\bbb}_t, i_{B, 0}^*(f^{-1})^*\wh{\bbb}_T\big)\\
&\equiv\CS\big(\wt{h}_0^*i_{B, 0}^*(f^{-1})^*\wh{\bbb}_t, \wt{h}_0^*i_{B, 0}^*(f^{-1})^*\wh{\bbb}_T\big)\\
&\equiv\CS\big(\wh{\bbb}^E_t\oplus\nabla^{\wh{H}}\oplus\aaa^{\wh{V}_0}_t, \wh{\bbb}^E_T\oplus\nabla^{\wh{H}}\oplus\aaa^{\wh{V}_0}_T\big)\\
&\equiv\CS\big(\wh{\bbb}^E_t, \wh{\bbb}^E_T\big)+\CS\big(\nabla^{\wh{H}}, \nabla^{\wh{H}}\big)+\CS\big(\aaa^{\wh{V}_0}_t, \aaa^{\wh{V}_0}_T\big)\\
&\equiv\CS\big(\wh{\bbb}^E_t, \wh{\bbb}^E_T\big)+\CS\big(\aaa^{\wh{V}_0}_t, \aaa^{\wh{V}_0}_T\big).
\end{aligned}
\end{equation}
By letting $t\to 0$ and $T\to\infty$ in (\ref{eq 4.31}),
$$i_{B, 0}^*\wh{\eta}(\bee, \wt{\bpi}, \bLLL\oplus\wh{\bWWW})\equiv\wh{\eta}(\EE, \bpi, \LL)+\wh{\eta}^{\wh{V}_0}.$$
Since $\wh{\eta}^{\wh{V}_0}\equiv 0$ by (\ref{eq 2.4}), (\ref{eq 4.27}) holds.

To prove (\ref{eq 4.28}), note that by (\ref{eq 4.10}), (\ref{eq 4.19}) and (\ref{eq 4.29}),
\begin{displaymath}
\begin{split}
i_{B, 1}^*(f^{-1})^*\wh{\bbb}_t&=\sqrt{t}\big(\DDD^{S\otimes E; \Delta}(\alpha(t))\oplus\alpha(t)s_{\wh{W}}\big)\\
&\qquad+\big(\nabla^{\pi_*E, u}\oplus\nabla^{\wh{H}}\oplus(h_0^*\wt{h}^*(\nabla^{\ker(\DD^{S\otimes E; \Delta})}\oplus\nabla^{\wh{V}_1}))^{\op}\big)-\frac{c(T)}{4\sqrt{t}}.
\end{split}
\end{displaymath}
Define a $\Z_2$-graded morphism
$$\wh{h}:\bpi_*\EE\oplus\wh{\HH}\oplus\big(\KER\big(\DD^{S\otimes E; \Delta}\big)\oplus\wh{\VV}_1\big)^{\op}\to\bpi_*\EE\oplus\wh{\HH}\oplus(\wt{\LL}_0\oplus\wh{\WW})^{\op}$$
by $\wh{h}:=\id_{\bpi_*\EE}\oplus\id_{\wh{\HH}}\oplus(h_0^{-1}\circ\wt{h}^{-1})^{\op}$. By (\ref{eq 3.21}) and (\ref{eq 2.3}),
\begin{equation}\label{eq 4.32}
\begin{aligned}[b]
&\wh{h}^*i_{B, 1}^*(f^{-1})^*\wh{\bbb}_t\\
&=\sqrt{t}\big(\DDD^{S\otimes E; \Delta}(\alpha(t))\oplus\alpha(t)s_{\wh{V}_1}\big)+\big(\nabla^{\pi_*E, u}\oplus\nabla^{\wh{H}}\oplus\nabla^{\ker(\DD^{S\otimes E; \Delta}), \op}\oplus\nabla^{\wh{V}_1}\big)-\frac{c(T)}{4\sqrt{t}}\\
&=\wh{\bbb}^{E; \Delta}_t\oplus\aaa^{\wh{V}_1}_t,
\end{aligned}
\end{equation}
where $\wh{\bbb}^{E; \Delta}_t$ is given by (\ref{eq 3.21}). Let $t<T\in(0, \infty)$. By applying the argument in (\ref{eq 4.31}) to (\ref{eq 4.32}),
\begin{equation}\label{eq 4.33}
i_{B, 1}^*\CS\big(\wh{\bbb}_t, \wh{\bbb}_T\big)\equiv\CS\big(\wh{\bbb}^{E; \Delta}_t, \wh{\bbb}^{E; \Delta}_T\big)+\CS\big(\aaa^{\wh{V}_1}_t, \aaa^{\wh{V}_1}_T\big).
\end{equation}
Let $t\to 0$ and $T\to\infty$ in (\ref{eq 4.33}). Since $\wh{\eta}^{\wh{V}_1}\equiv 0$ by (\ref{eq 2.4}), it follows from (\ref{eq 3.24}) and Lemma \ref{lemma 4.1} that
$$i_{B, 1}^*\wh{\eta}(\bee, \wt{\bpi}, \bLLL\oplus\wh{\bWWW})\equiv\wh{\eta}\big(\EE, \bpi, \KER\big(\DD^{S\otimes E; \Delta}\big)\big)\equiv\wt{\eta}(\EE, \bpi, \HH, s).$$
Thus (\ref{eq 4.28}) holds.
\end{proof}

At the expense of using \cite[Proposition 3]{GL18}, the proof of Proposition \ref{prop 4.1} can be simplified by taking $(H, s)$ to be $(L^-, i_-)$. We would like to thank a referee for pointing this out.

\section{Some properties of geometric kernel bundles and the Bismut--Cheeger eta form in the ASGL approach}\label{s 5}

In this section, we prove an extended variational formula of the Bismut--Cheeger eta form in the ASGL approach.

The following lemma, which roughly says geometric kernel bundles and the Bismut--Cheeger eta form in the ASGL approach are stable under perturbation of elementary $\Z_2$-graded geometric bundles, is one of the key ingredients of the construction of the analytic index $\ind^a_k$.
\begin{lemma}\label{lemma 5.1}
Let $\pi:X\to B$ be a submersion with closed, oriented and spin$^c$ fibers of even dimension, equipped with a Riemannian and differential spin$^c$ structure $\bpi$. Let $\EE=(E, g^E, \nabla^E)$ be a geometric bundle over $X$. Denote by $\DD^{S\otimes E}$ the twisted spin$^c$ Dirac operator associated to $(\EE, \bpi)$. Let $(H, s)$ satisfy the AS property for $\DD^{S\otimes E}$ and $\HH$ a geometric bundle defining $\KER\big(\DD^{S\otimes E; \Delta}\big)$, where $\Delta$ is the odd self-adjoint map associated to $(H, s)$ as in (\ref{eq 3.3}). For any geometric bundle $\KK$ over $B$, the pair $(H\oplus K, s\oplus 0)$ satisfies the AS property for $\DD^{S\otimes E}$ and $\HH\oplus\KK$ is the geometric bundle defining $\KER\big(\DD^{S\otimes E; \Phi}\big)$, where $\Phi$ is the odd self-adjoint map associated to $(H\oplus K, s\oplus 0)$ as in (\ref{eq 3.3}). Moreover,
\begin{equation}\label{eq 5.1}
\KER\big(\DD^{S\otimes E; \Phi}\big)=\KER\big(\DD^{S\otimes E; \Delta}\big)\oplus\wh{\KK},
\end{equation}
and
\begin{equation}\label{eq 5.2}
\wt{\eta}(\EE, \bpi, \HH\oplus\KK, s\oplus 0)=\wt{\eta}(\EE, \bpi, \HH, s).
\end{equation}
\end{lemma}
\begin{proof}
Since $\DD^{S\otimes E}_++s:(\pi_*E)^+\oplus H\to(\pi_*E)^-$ is surjective, it follows that
$$\DD^{S\otimes E}_++(s\oplus 0):(\pi_*E)^+\oplus(H\oplus K)\to(\pi_*E)^-\textrm{ is surjective.}$$
Thus $(H\oplus K, s\oplus 0)$ satisfies the AS property for $\DD^{S\otimes E}$. As shown in \S\ref{s 3.2}, the vector bundle $\ker\big(\DD^{S\otimes E; \Phi}\big)\to B$ exists. Since the map $\DD^{S\otimes E; \Phi}_+:(\pi_*E\oplus\wh{H}\oplus\wh{K})^+\to(\pi_*E\oplus\wh{H}\oplus\wh{K})^-$ is given by
\begin{equation}\label{eq 5.3}
\DD^{S\otimes E; \Phi}_+=\begin{pmatrix} \DD^{S\otimes E}_+ & s & 0 \\ 0 & 0 & 0 \\ 0 & 0 & 0 \end{pmatrix}=\DD^{S\otimes E; \Delta}_+\oplus 0,
\end{equation}
it follows that
\begin{equation}\label{eq 5.4}
\ker\big(\DD^{S\otimes E; \Phi}_+\big)=\ker(\DD^{S\otimes E}_++s)\oplus K=\ker\big(\DD^{S\otimes E; \Delta}_+\big)\oplus K.
\end{equation}
On the other hand, by (\ref{eq 3.5}),
\begin{equation}\label{eq 5.5}
\begin{aligned}[b]
\ker\big(\DD^{S\otimes E; \Phi}_-\big)&\cong\cok\big(\DD^{S\otimes E; \Phi}_+\big)\\
&=\frac{(\pi_*E)^-\oplus H\oplus K}{\im\big(\DD^{S\otimes E; \Phi}_+\big)}\\
&\cong H\oplus K\\
&\cong\ker\big(\DD^{S\otimes E; \Delta}_-\big)\oplus K.
\end{aligned}
\end{equation}
By regarding $\ker\big(\DD^{S\otimes E; \Delta}\big)\to B$ and $\wh{K}\to B$ as $\Z_2$-graded subbundles of $\ker\big(\DD^{S\otimes E; \Phi}\big)\to B$, it follows from (\ref{eq 5.4}) and (\ref{eq 5.5}) that
\begin{equation}\label{eq 5.6}
\ker\big(\DD^{S\otimes E; \Phi}\big)=\ker\big(\DD^{S\otimes E; \Delta}\big)\oplus\wh{K}.
\end{equation}
By (\ref{eq 5.6}),
$$g^{\ker(\DD^{S\otimes E; \Phi})}=g^{\ker(\DD^{S\otimes E; \Delta})}\oplus g^{\wh{K}}\quad\textrm{ and }\quad\nabla^{\ker(\DD^{S\otimes E; \Phi})}=\nabla^{\ker(\DD^{S\otimes E; \Delta})}\oplus\nabla^{\wh{K}}.$$
Thus (\ref{eq 5.1}) holds.

By (\ref{eq 5.3}),
\begin{equation}\label{eq 5.7}
\begin{aligned}[b]
\wt{\bbb}^{E; \Phi}_t&=\sqrt{t}\DD^{S\otimes E; \Phi}+\big(\nabla^{\pi_*E, u}\oplus\nabla^{\wh{H}}\oplus\nabla^{\wh{K}}\big)-\frac{c(T)}{4\sqrt{t}}\\
&=\bigg(\sqrt{t }\DD^{S\otimes E; \Delta}+\big(\nabla^{\pi_*E, u}\oplus\nabla^{\wh{H}}\big)-\frac{c(T)}{4\sqrt{t}}\bigg)\oplus\nabla^{\wh{K}}\\
&=\wt{\bbb}^{E; \Delta}_t\oplus\nabla^{\wh{K}}.
\end{aligned}
\end{equation}
Let $t<T\in(0, \infty)$. By (\ref{eq 5.7}), (\ref{eq 2.11}) and (\ref{eq 2.12}),
\begin{equation}\label{eq 5.8}
\begin{aligned}[b]
\CS\big(\wt{\bbb}^{E; \Phi}_t, \wt{\bbb}^{E; \Phi}_T)&\equiv\CS\big(\wt{\bbb}^{E; \Delta}_t\oplus\nabla^{\wh{K}}, \wt{\bbb}^{E; \Delta}_T\oplus\nabla^{\wh{K}}\big)\\
&\equiv\CS\big(\wt{\bbb}^{E; \Delta}_t, \wt{\bbb}^{E; \Delta}_T\big)+\CS(\nabla^{\wh{K}}, \nabla^{\wh{K}})\\
&\equiv\CS\big(\wt{\bbb}^{E; \Delta}_t, \wt{\bbb}^{E; \Delta}_T\big).
\end{aligned}
\end{equation}
By letting $t\to 0$ and $T\to\infty$ in (\ref{eq 5.8}), (\ref{eq 5.2}) holds.
\end{proof}

The following proposition is essentially a special case of \cite[Proposition 3]{GL18}. Since we need a more precise statement for the construction of the analytic index $\ind^a_k$, we include a slightly different proof for the sake of completeness.
\begin{prop}\label{prop 5.1}
Let $\pi:X\to B$ be a submersion with closed, oriented and spin$^c$ fibers of even dimension, equipped with a Riemannian and differential spin$^c$ structure $\bpi$. Let $\EE=(E, g^E, \nabla^E)$ be a geometric bundle over $X$. Denote by $\DD^{S\otimes E}$ the twisted spin$^c$ Dirac operator associated to $(\EE, \bpi)$. Let $j\in\set{0, 1}$. Let $(H_j, s_j)$ satisfy the AS property for $\DD^{S\otimes E}$ and $\HH_j$ a geometric bundle defining $\KER\big(\DD^{S\otimes E; \Delta_j}\big)$, where $\Delta_j$ is the odd self-adjoint map associated to $(H_j, s_j)$ as in (\ref{eq 3.3}). Then there exists a $\Z_2$-graded morphism
$$P_{01}:\KER\big(\DD^{S\otimes E; \Delta_0}\big)\oplus\wh{\HH}_1\to\KER\big(\DD^{S\otimes E; \Delta_1}\big)\oplus\wh{\HH}_0$$
such that
\begin{equation}\label{eq 5.9}
\wt{\eta}(\EE, \bpi, \HH_1, s_1)-\wt{\eta}(\EE, \bpi, \HH_0, s_0)\equiv-\CS\big(\nabla^{\ker(\DD^{S\otimes E; \Delta_0})}\oplus\nabla^{\wh{H}_1}, P_{01}^*(\nabla^{\ker(\DD^{S\otimes E; \Delta_1})}\oplus\nabla^{\wh{H}_0})\big).
\end{equation}
\end{prop}
\begin{proof}
By Lemma \ref{lemma 5.1}, $(H_0\oplus H_1, s_0\oplus 0)$ and $(H_0\oplus H_1, 0\oplus s_1)$ satisfy the AS property for $\DD^{S\otimes E}$, and
\begin{equation}\label{eq 5.10}
\begin{split}
\KER\big(\DD^{S\otimes E; \Phi_0}\big)&=\KER\big(\DD^{S\otimes E; \Delta_0}\big)\oplus\wh{\HH}_1,\\
\KER\big(\DD^{S\otimes E; \Phi_1}\big)&=\KER\big(\DD^{S\otimes E; \Delta_1}\big)\oplus\wh{\HH}_0,
\end{split}
\end{equation}
where $\Phi_0$ and $\Phi_1$ are the odd self-adjoint maps associated to $(H_0\oplus H_1, s_0\oplus 0)$ and $(H_0\oplus H_1, 0\oplus s_1)$ as in (\ref{eq 3.3}), respectively.

For each $t\in I$, define a map $S_t:H_0\oplus H_1\to(\pi_*E)^-$ by
$$S_t=(1-t)(s_0\oplus 0)+t(0\oplus s_1).$$
We claim that the map $\DD^{S\otimes E}_++S_t:(\pi_*E)^+\oplus(H_0\oplus H_1)\to(\pi_*E)^-$ is surjective for each $t\in I$. Let $\psi\in(\pi_*E)^-$. By assumption, there exist $\psi_0\oplus h_0\in(\pi_*E)^+\oplus H_0$ and $\psi_1\oplus h_1\in(\pi_*E)^-\oplus H_1$ such that
\begin{displaymath}
\begin{split}
\psi&=(\DD^{S\otimes E}_++s_0)(\psi_0\oplus h_0)=\DD^{S\otimes E}_+(\psi_0)+s_0(h_0),\\
\psi&=(\DD^{S\otimes E}_++s_1)(\psi_1\oplus h_1)=\DD^{S\otimes E}_+(\psi_1)+s_1(h_1),
\end{split}
\end{displaymath}
respectively. Since
\begin{displaymath}
\begin{split}
\psi&=(1-t)\psi+t\psi\\
&=(1-t)(\DD^{S\otimes E}_+(\psi_0)+s_0(h_0))+t(\DD^{S\otimes E}_+(\psi_1)+s_1(h_1))\\
&=\DD^{S\otimes E}_+((1-t)\psi_0+t\psi_1)+(1-t)s_0(h_0)+ts_1(h_1)\\
&=(\DD^{S\otimes E}_++S_t)(((1-t)\psi_0+t\psi_1)\oplus(h_0\oplus h_1)),
\end{split}
\end{displaymath}
the map $\DD^{S\otimes E}_++S_t$ is surjective for each $t\in I$. Thus $(H_0\oplus H_1, S_t)$ satisfies the AS property for $\DD^{S\otimes E}$.

Let $\wt{\bpi}$ be the Riemannian and differential spin$^c$ structure on $\wt{\pi}:\wt{X}\to\wt{B}$ obtained by pulling back $\bpi$ via $p_X$. Define a geometric bundle $\bee$ over $\wt{X}$ by $\bee=p_X^*\EE$. Denote by $\DD^{S\otimes\e}$ the twisted spin$^c$ Dirac operator associated to $(\bee, \wt{\bpi})$.

Write $\HH$ for $\HH_0\oplus\HH_1$. Define a geometric bundle $\bhh$ over $\wt{B}$ by $\bhh=p_B^*\HH$ and a linear map $\wt{s}:\hh\to(\wt{\pi}_*\e)^-$ by
$$\wt{s}(h, t)=(s_t(h), t).$$
Note that $\wt{s}(h, 0)=\big((s_0\oplus 0)(h), 0\big)$ and $\wt{s}(h, 1)=\big((0\oplus s_1)(h), 1\big)$. Since $\DD^{S\otimes E}_++S_t:(\pi_*E)^+\oplus H\to(\pi_*E)^-$ is surjective, it follows that
$$\DD^{S\otimes\e}_++\wt{s}:(\wt{\pi}_*\e)^+\oplus\hh\to(\wt{\pi}_*\e)^-\textrm{ is surjective}.$$
Thus $(\hh, \wt{s})$ satisfies the AS property for $\DD^{S\otimes\e}$. Denote by $\wt{\Phi}$ the odd self-adjoint map associated to $(\hh, \wt{s})$ as in (\ref{eq 3.3}). As shown in \S\ref{s 3.2}, the vector bundle $\ker\big(\DD^{S\otimes\e; \wt{\Phi}}\big)\to\wt{B}$ exists. Note that $\bhh$ is the geometric bundle defining $\KER\big(\DD^{S\otimes\e; \wt{\Phi}}\big)$.

Since
$$i_{B, 0}^*(\wt{\pi}_*\e\oplus\wh{\hh})\cong i_{B, 1}^*(\wt{\pi}_*\e\oplus\wh{\hh})\cong\pi_*E\oplus\wh{H}$$
and
$$\wt{s}|_{i_{B, 0}^*\hh}=s_0\oplus 0\quad\textrm{ and }\quad\wt{s}|_{i_{B, 1}^*\hh}=0\oplus s_1,$$
it follows that
\begin{equation}\label{eq 5.11}
\DD^{S\otimes\e; \wt{\Phi}}|_{i_{B, j}^*(\wt{\pi}_*\e\oplus\wh{\hh})}=\DD^{S\otimes E; \Phi_j}.
\end{equation}
By (\ref{eq 5.11}),
$$i_{B, j}^*\ker\big(\DD^{S\otimes\e; \wt{\Phi}}\big)=\ker\big(\DD^{S\otimes E; \Phi_j}\big)\quad\textrm{ and }\quad i_{B, j}^*g^{\ker(\DD^{S\otimes\e; \wt{\Phi}})}=g^{\ker(\DD^{S\otimes E; \Phi_j})}$$
for each $j\in\set{0, 1}$. Let $P_{01}:\KER\big(\DD^{S\otimes E; \Phi_0}\big)\to\KER\big(\DD^{S\otimes E; \Phi_1}\big)$ be the $\Z_2$-graded morphism (\ref{eq 2.5}). Note that
\begin{equation}\label{eq 5.12}
\nabla^{\ker(\DD^{S\otimes E; \Phi_0})}\quad\textrm{ and }\quad P_{01}^*\nabla^{\ker(\DD^{S\otimes E; \Phi_1})}
\end{equation}
are $\Z_2$-graded unitary connections on $\ker\big(\DD^{S\otimes E; \Phi_0}\big)\to B$.

Define a $\Z_2$-graded geometric bundle $\bkk=(\kk, g^\kk, \nabla^\kk)$ over $\wt{B}$ by
$$\kk:=p_B^*\ker\big(\DD^{S\otimes E; \Phi_0}\big),\qquad g^\kk=p_B^*g^{\ker(\DD^{S\otimes E; \Phi_0})}$$
and $\nabla^\kk$ by (\ref{eq 2.6}) and (\ref{eq 2.7}) with respect to (\ref{eq 5.12}). Then
\begin{equation}\label{eq 5.13}
i_{B, 0}^*\nabla^\kk=\nabla^{\ker(\DD^{S\otimes E; \Phi_0})}\quad\textrm{ and }\quad i_{B, 1}^*\nabla^\kk=P_{01}^*\nabla^{\ker(\DD^{S\otimes E; \Phi_1})}.
\end{equation}
Since $i_{B, 0}\circ p_B:\wt{B}\to\wt{B}$ is smoothly homotopic to $\id_{\wt{B}}$, it follows that
$$\kk\cong\ker\big(\DD^{S\otimes\e; \wt{\Phi}}\big)$$
as $\Z_2$-graded complex vector bundles. By \cite[(2) of Remark 2.1]{H23}, let
$$\varphi:\KER\big(\DD^{S\otimes\e; \wt{\Phi}}\big)\to\bkk$$
be a $\Z_2$-graded morphism. Then $\wt{\nabla}^{\ker(\DD^{S\otimes\e; \wt{\Phi}})}:=\varphi^*\nabla^\kk$ is a $\Z_2$-graded unitary connection on $\ker\big(\DD^{S\otimes\e; \wt{\Phi}}\big)\to\wt{B}$.

Define a rescaled Bismut superconnection on 
$$\wt{\pi}_*\e\oplus\wh{\hh}\oplus\ker\big(\DD^{S\otimes\e; \wt{\Phi}}\big)^{\op}\to\wt{B}$$ 
by
\begin{equation}\label{eq 5.14}
\wh{\bbb}_t=\sqrt{t}\DDD^{S\otimes\e; \wt{\Phi}}(\alpha(t))+\big(\nabla^{\wt{\pi}_*\e, u}\oplus\nabla^\wh{\hh}\oplus\wt{\nabla}^{\ker(\DD^{S\otimes\e; \wt{\Phi}}), \op}\big)-\frac{c(\wt{T})}{4\sqrt{t}},
\end{equation}
where $\DDD^{S\otimes\e; \wt{\Phi}}(\alpha(t))$ is given by (\ref{eq 3.20}). By (\ref{eq 3.22}),
\begin{equation}\label{eq 5.15}
\lim_{t\to\infty}\ch(\wh{\bbb}_t)=0.
\end{equation}
By (\ref{eq 3.23}),
\begin{equation}\label{eq 5.16}
\begin{aligned}[b]
\lim_{t\to 0}\ch(\wh{\bbb}_t)&=\int_{\wt{X}/\wt{B}}\todd(\nabla^{T^V\wt{X}})\wedge\ch(\nabla^\e)-\ch(\wt{\nabla}^{\ker(\DD^{S\otimes\e; \wt{\Phi}})})\\
&=\int_{\wt{X}/\wt{B}}\todd(\nabla^{T^V\wt{X}})\wedge\ch(\nabla^\e)-\ch(\nabla^\kk).
\end{aligned}
\end{equation}
Denote by $\wh{\eta}\big(\bee, \wt{\bpi}, \KER\big(\DD^{S\otimes\e; \wt{\Phi}}\big)\big)$ the Bismut--Cheeger eta form associated to $\wh{\bbb}_t$ (\ref{eq 5.14}). By (\ref{eq 5.15}) and (\ref{eq 5.16}),
\begin{equation}\label{eq 5.17}
d\wh{\eta}\big(\bee, \wt{\bpi}, \KER\big(\DD^{S\otimes\e; \wt{\Phi}}\big)\big)=\int_{\wt{X}/\wt{B}}\todd(\nabla^{T^V\wt{X}})\wedge\ch(\nabla^\e)-\ch(\nabla^\kk).
\end{equation}
By (\ref{eq 2.1}), (\ref{eq 5.17}), (\ref{eq 5.13}), (\ref{eq 4.25}) and (\ref{eq 5.10}),
\begin{equation}\label{eq 5.18}
\begin{aligned}[b]
&i_{B, 1}^*\wh{\eta}\big(\bee, \wt{\bpi}, \KER\big(\DD^{S\otimes\e; \wt{\Phi}}\big)\big)-i_{B, 0}^*\wh{\eta}\big(\bee, \wt{\bpi}, \KER(\DD^{S\otimes\e; \wt{\Phi}}\big)\big)\\
&\equiv-\int_{\wt{B}/B}d^{\wt{B}}\wh{\eta}\big(\bee, \wt{\bpi}, \KER\big(\DD^{S\otimes\e; \wt{\Phi}}\big)\big)\\
&\equiv-\int_{\wt{B}/B}\bigg(\int_{\wt{X}/\wt{B}}\todd(\nabla^{T^V\wt{X}})\wedge\ch(\nabla^\e)-\ch(\nabla^\kk)\bigg)\\
&\equiv-\int_{\wt{B}/B}\int_{\wt{X}/\wt{B}}\todd(\nabla^{T^V\wt{X}})\wedge\ch(\nabla^\e)-\CS\big(\nabla^{\ker(\DD^{S\otimes E; \Phi_0})}, P_{01}^*\nabla^{\ker(\DD^{S\otimes E; \Phi_1})}\big)\\
&\equiv-\CS\big(\nabla^{\ker(\DD^{S\otimes E; \Phi_0})}, P_{01}^*\nabla^{\ker(\DD^{S\otimes E; \Phi_1})}\big)\\
&\equiv-\CS\big(\nabla^{\ker(\DD^{S\otimes E; \Delta_0})}\oplus\nabla^{\wh{H}_1}, P_{01}^*(\nabla^{\ker(\DD^{S\otimes E; \Delta_1})}\oplus\nabla^{\wh{H}_0})\big).
\end{aligned}
\end{equation}
To prove (\ref{eq 5.9}), by (\ref{eq 5.18}) it suffices to prove that
\begin{equation}\label{eq 5.19}
i_{B, j}^*\wh{\eta}\big(\bee, \wt{\bpi}, \KER\big(\DD^{S\otimes\e; \wt{\Phi}}\big)\big)\equiv\wt{\eta}(\EE, \bpi, \HH_j, s_j)
\end{equation}
for $j\in\set{0, 1}$. Define a $\Z_2$-graded morphism
$$\wh{\varphi}:\wt{\pi}_*\bee\oplus\wh{\bhh}\oplus\KER\big(\DD^{S\otimes\e; \wt{\Phi}}\big)^{\op}\to
\wt{\pi}_*\bee\oplus\wh{\bhh}\oplus\bkk^{\op}$$
by $\wh{\varphi}:=\id_{\wt{\pi}_*\bee}\oplus\id_{\wh{\bhh}}\oplus\varphi^{\op}$. By (\ref{eq 5.14}),
\begin{equation}\label{eq 5.20}
(\wh{\varphi}^{-1})^*\wh{\bbb}_t=\sqrt{t}\DDD^{S\otimes\e; \wt{\Phi}}(\alpha(t))+\big(\nabla^{\wt{\pi}_*\e, u}\oplus\nabla^{\wh{\hh}}\oplus\nabla^{\kk, \op}\big)-\frac{c(\wt{T})}{4\sqrt{t}}.
\end{equation}

For $j=0$, by (\ref{eq 5.20}), (\ref{eq 5.13}) and (\ref{eq 5.10}),
\begin{equation}\label{eq 5.21}
\begin{aligned}[b]
&i_{B, 0}^*\big((\wh{\varphi}^{-1})^*\wh{\bbb}_t\big)\\
&=\sqrt{t}\DDD^{S\otimes E; \Phi_0}(\alpha(t))+\big(\nabla^{\pi_*E, u}\oplus\nabla^\wh{H}\oplus\nabla^{\ker(\DD^{S\otimes E; \Phi_0}), \op}\big)-\frac{c(T)}{4\sqrt{t}}\\
&=\sqrt{t}\big(\DDD^{S\otimes E; \Delta_0}(\alpha(t))\oplus\alpha(t)s_{\wh{H}_1}\big)+\big(\nabla^{\pi_*E, u}\oplus\nabla^{\wh{H}_0}\oplus\nabla^{\wh{H}_1}\oplus\nabla^{\ker(\DD^{S\otimes E; \Delta_0}), \op}\oplus\nabla^{\wh{H}_1}\big)-\frac{c(T)}{4\sqrt{t}}\\
&=\wh{\bbb}^{E; \Delta_0}_t\oplus\aaa^{\wh{H}_1}_t\oplus\nabla^{\wh{H}_1},
\end{aligned}
\end{equation}
where $\wh{\bbb}^{E; \Delta_0}_t$ and $\aaa^{\wh{H}_1}_t$ are given by (\ref{eq 3.21}) and (\ref{eq 2.3}), respectively. Let $t<T\in(0, \infty)$. By (\ref{eq 2.14}), (\ref{eq 5.21}) and (\ref{eq 2.12}),
\begin{equation}\label{eq 5.22}
\begin{aligned}[b]
i_{B, 0}^*\CS\big(\wh{\bbb}_t, \wh{\bbb}_T\big)& \equiv i_{B, 0}^*\CS\big((\wh{\varphi}^{-1})^*\wh{\bbb}_t, (\wh{\varphi}^{-1})^*\wh{\bbb}_T\big)\\
&\equiv\CS\big(i_{B, 0}^*(\wh{\varphi}^{-1})^*\wh{\bbb}_t, i_{B, 0}^*(\wh{\varphi}^{-1})^*\wh{\bbb}_T\big)\\
&\equiv\CS\big(\wh{\bbb}^{E; \Delta_0}_t\oplus\aaa^{\wh{H}_1}_t\oplus\nabla^{\wh{H}_1}, \wh{\bbb}^{E; \Delta_0}_{0, T}\oplus\aaa^{\wh{H}_1}_T\oplus\nabla^{\wh{H}_1}\big)\\
&\equiv\CS\big(\wh{\bbb}^{E; \Delta_0}_t, \wh{\bbb}^{E; \Delta_0}_{0, T}\big)+\CS\big(\aaa^{\wh{H}_1}_t, \aaa^{\wh{H}_1}_T\big).
\end{aligned}
\end{equation}
Let $t\to 0$ and $T\to\infty$ in (\ref{eq 5.22}). Since $\wh{\eta}^{\wh{H}_1}\equiv 0$ by (\ref{eq 2.4}), it follows from Lemma \ref{lemma 4.1} that
$$i_{B, 0}^*\wh{\eta}\big(\bee, \wt{\bpi}, \KER\big(\DD^{S\otimes\e; \wt{\Phi}}\big)\big)\equiv\wh{\eta}\big(\EE, \bpi, \KER\big(\DD^{S\otimes E; \Delta_0}\big)\big)\equiv\wt{\eta}(\EE, \bpi, \HH_0, s_0).$$
Thus (\ref{eq 5.19}) holds for $j=0$.

For $j=1$, define a $\Z_2$-graded morphism
$$\wh{h}:\pi_*\EE\oplus\wh{\HH}\oplus\KER\big(\DD^{S\otimes E; \Phi_0}\big)^{\op}\to\pi_*\EE\oplus\wh{\HH}\oplus\KER\big(\DD^{S\otimes E; \Phi_1}\big)^{\op}$$
by $\wh{h}:=\id_{\pi_*\EE}\oplus\id_{\wh{\HH}}\oplus P_{0, 1}^{\op}$. By (\ref{eq 5.20}), (\ref{eq 5.13}) and (\ref{eq 5.10}),
\begin{displaymath}
\begin{split}
i_{B, 1}^*\big((\wh{\varphi}^{-1})^*\wh{\bbb}_t\big)&=\sqrt{t}\DDD^{S\otimes E; \Phi_1}(\alpha(t))+\big(\nabla^{\pi_*E, u}\oplus\nabla^\wh{H}\oplus(P_{0, 1}^{\op})^*\nabla^{\ker(\DD^{S\otimes E; \Phi_1}), \op}\big)-\frac{c(T)}{4\sqrt{t}}\\
&=(\wh{h}^{-1})^*\bigg(\sqrt{t}\big(\DDD^{S\otimes E; \Delta_1}(\alpha(t))\oplus\alpha(t)s_{\wh{H}_0}\big)\\
&\qquad+\big(\nabla^{\pi_*E, u}\oplus\nabla^{\wh{H}_0}\oplus\nabla^{\wh{H}_1}\oplus\nabla^{\ker(\DD^{S\otimes E; \Delta_1}), \op}\oplus\nabla^{\wh{H}_0}\big)-\frac{c(T)}{4\sqrt{t}}\bigg)\\
&=(\wh{h}^{-1})^*\big(\wh{\bbb}^{E; \Delta_1}_t\oplus\aaa^{\wh{H}_0}_t\oplus\nabla^{\wh{H}_0}\big),
\end{split}
\end{displaymath}
where $\aaa^{\wh{H}_0}_t$ is given by (\ref{eq 2.3}). By the same argument in (\ref{eq 5.22}),
\begin{equation}\label{eq 5.23}
i_{B, 1}^*\CS\big(\wh{\bbb}_t, \wh{\bbb}_T\big)=\CS\big(\wh{\bbb}^{E; \Delta_1}_t, \wh{\bbb}^{E; \Delta_1}_T\big)+\CS\big(\aaa^{\wh{H}_0}_t, \aaa^{\wh{H}_0}_T\big).
\end{equation}
Let $t\to 0$ and $T\to\infty$ in (\ref{eq 5.23}). Since $\wh{\eta}^{\wh{H}_0}\equiv 0$ by (\ref{eq 2.4}), it follows from Lemma \ref{lemma 4.1} that
$$i_{B, 1}^*\wh{\eta}\big(\bee, \wt{\bpi}, \KER\big(\DD^{S\otimes\e; \wt{\Phi}}\big)\big)\equiv\wh{\eta}\big(\EE, \bpi, \KER\big(\DD^{S\otimes E; \Delta_1}\big)\big)\equiv\wt{\eta}(\EE, \bpi, \HH_1, s_1).$$
Thus (\ref{eq 5.19}) holds for $j=1$.
\end{proof}
\begin{prop}\label{prop 5.2}
Let $\pi:X\to B$ be a submersion with closed, oriented and spin$^c$ fibers of even dimension, equipped with Riemannian and differential spin$^c$ structures $\bpi_0$ and $\bpi_1$ with the same underlying topological spin$^c$ structures. Let $\EE_0=(E, g^E_0, \nabla^E_0)$ and $\EE_1=(E, g^E_1, \nabla^E_1)$ be geometric bundles over $X$. Denote by $\DD^{S\otimes E}_0$ and $\DD^{S\otimes E}_1$ the twisted spin$^c$ Dirac operators associated to $(\EE_0, \bpi_0)$ and $(\EE_1, \bpi_1)$, respectively. Let $(H_0, s_0)$ and $(H_1, s_1)$ satisfy the AS property for $\DD^{S\otimes E}_0$ and $\DD^{S\otimes E}_1$, and let $\HH_0$ and $\HH_1$ be geometric bundles defining $\KER\big(\DD^{S\otimes E; \Delta_0}_0\big)$ and $\KER\big(\DD^{S\otimes E; \Delta_1}_1\big)$, where $\Delta_0$ and $\Delta_1$ are the odd self-adjoint maps associated to $(H_0, s_0)$ and $(H_1, s_1)$ as in (\ref{eq 3.3}), respectively. Then there exist geometric bundles $\VV_0$ and $\VV_1$ over $B$ and a $\Z_2$-graded morphism
$$h_{01}:\KER\big(\DD^{S\otimes E; \Delta_0}_0\big)\oplus\wh{\VV}_0\to\KER\big(\DD^{S\otimes E; \Delta_1}_1\big)\oplus\wh{\VV}_1$$
such that
\begin{equation}\label{eq 5.24}
\begin{aligned}[b]
&\wt{\eta}(\EE_1, \bpi_1, \HH_1, s_1)-\wt{\eta}(\EE_0, \bpi_0, \HH_0, s_0)\\
&\equiv\int_{X/B}\big(T\wh{A}(\nabla^{T^VX}_0, \nabla^{T^VX}_1)\wedge e^{\frac{1}{2}c_1(\nabla^\lambda_0)}+\wh{A}(\nabla^{T^VX}_1)\wedge e^{\frac{1}{2}Tc_1(\nabla^\lambda_0, \nabla^\lambda_1)}\big)\wedge\ch(\nabla^E_0)\\
&\quad+\int_{X/B}\todd(\nabla^{T^VX}_1)\wedge\CS(\nabla^E_0, \nabla^E_1)-\CS\big(\nabla^{\ker(\DD^{S\otimes E; \Delta_0}_0)}\oplus\nabla^{\wh{V}_0}, h_{01}^*(\nabla^{\ker(\DD^{S\otimes E; \Delta_1}_1)}\oplus\nabla^{\wh{V}_1})\big).
\end{aligned}
\end{equation}
\end{prop}
\begin{proof}
Since the space of splitting maps is affine, there exists a smooth path of horizontal distributions $\set{T^H_tX\to X}_{t\in I}$ joining $T^H_0X\to X$ and $T^H_1X\to X$. Define smooth paths of Hermitian metrics and unitary connections 
$$g^E_t,\qquad g^{T^VX}_t,\qquad g^\lambda_t,\qquad\nabla^E_t,\qquad\nabla^\lambda_t$$ 
by (\ref{eq 2.6}) with respect to 
$$(g^E_0, g^E_1),\qquad(g^{T^VX}_0, g^{T^VX}_1),\qquad(g^\lambda_0, g^\lambda_1),\qquad(\nabla^E_0, \nabla^E_1),\qquad(\nabla^\lambda_0, \nabla^\lambda_1),$$ 
respectively. Then
$$c(t)=(g^E_t, \nabla^E_t, T^H_tX, g^{T^VX}_t, g^\lambda_t, \nabla^\lambda_t)$$
is a smooth path joining $c(0)$ and $c(1)$. The new path $\wt{c}(t)$ defined by
$$\wt{c}(t)=\left\{
\begin{array}{ll}
(g^E_0, \nabla^E_0, T^H_{3t}X, g^{T^VX}_{3t}, g^\lambda_0, \nabla^\lambda_0), & \displaystyle\textrm{ for } t\in\bigg[0, \frac{1}{3}\bigg]\\\\
(g^E_0, \nabla^E_0, T^H_1X, g^{T^VX}_1, g^\lambda_{3t-1}, \nabla^\lambda_{3t-1}), & \displaystyle\textrm{ for } t\in\bigg[\frac{1}{3}, \frac{2}{3}\bigg]\\\\
(g^E_{3t-2}, \nabla^E_{3t-2}, T^H_1X, g^{T^VX}_1, g^\lambda_1, \nabla^\lambda_1), & \displaystyle\textrm{ for } t\in\bigg[\frac{2}{3}, 1\bigg]
\end{array}
\right.$$
also joins $c(0)$ and $c(1)$.

Define complex vector bundles $\e\to\wt{X}$ and $\wt{\lambda}\to\wt{X}$ by $\e=p_X^*E$ and $\wt{\lambda}=p_X^*\lambda$, respectively. Define a geometric bundle $\bee=(\e, g^\e, \nabla^\e)$ over $\wt{X}$ and a Riemannian and differential spin$^c$ structure $\wt{\bpi}=(T^H\wt{X}, g^{T^V\wt{X}}, g^{\wt{\lambda}}, \nabla^{\wt{\lambda}})$ on $\wt{\pi}:\wt{X}\to\wt{B}$ so that the restriction of
$$(g^\e, \nabla^\e, T^H\wt{X}, g^{T^V\wt{X}}, g^{\wt{\lambda}}, \nabla^{\wt{\lambda}})$$
to $X\times\set{t}$ is given by $\wt{c}(t)$ for each $t\in I$.

Denote by $\DD^{S\otimes\e}$ the twisted spin$^c$ Dirac operator associated to $(\bee, \wt{\bpi})$. Let $(\kk, \wt{r})$ satisfy the AS property for $\DD^{S\otimes\e}$ and $\bkk$ a geometric bundle defining $\KER\big(\DD^{S\otimes\e; \wt{\Phi}}\big)$, where $\wt{\Phi}$ is the odd self-adjoint map associated to $(\kk, \wt{r})$ as in (\ref{eq 3.3}).

For $j\in\set{0, 1}$, define a geometric bundle $\KK_j$ over $B$ by $\KK_j:=i_{B, j}^*\bkk$. Note that
\begin{equation}\label{eq 5.25}
\begin{split}
i_{B, 0}^*(\wt{\pi}_*\e\oplus\wh{\kk})&=\pi_*E\oplus\wh{K}_0,\\
i_{B, 1}^*(\wt{\pi}_*\e\oplus\wh{\kk})&=\pi_*E\oplus\wh{K}_1.
\end{split}
\end{equation}
Define maps $r_j:K_j\to(\pi_*E)^-$ and $\Phi_j:\pi_*E\oplus\wh{K}_j\to\pi_*E\oplus\wh{K}_j$ by
\begin{equation}\label{eq 5.26}
r_j:=\wt{r}|_{i_{B, j}^*\kk}\quad\textrm{ and }\quad \Phi_j:=\wt{\Phi}|_{i_{B, j}^*\kk},
\end{equation}
respectively. By (\ref{eq 5.25}) and (\ref{eq 5.26}),
$$\big(\DD^{S\otimes\e}_++\wt{r}\big)|_{i_{B, j}^*((\wt{\pi}_*\e)^+\oplus\kk)}=\DD^{S\otimes E}_{j, +}+r_j.$$
Thus $(K_j, r_j)$ satisfies the AS property for $\DD^{S\otimes E}_j$ and $\Phi_j$ is the odd self-adjoint map associated to $(K_j, r_j)$ as in (\ref{eq 3.3}). As shown in \S\ref{s 3.2}, the vector bundle $\ker\big(\DD^{S\otimes E; \Phi_j}_j\big)\to B$ exists. Since
$$\DD^{S\otimes\e; \wt{\Phi}}|_{i_{B, j}^*(\wt{\pi}_*\e\oplus\wh{\kk})}=\DD^{S\otimes E; \Phi_j}_j,$$
it follows that
$$i_{B, j}^*\ker\big(\DD^{S\otimes\e; \wt{\Phi}}\big)=\ker\big(\DD^{S\otimes E; \Phi_j}_j\big)\quad\textrm{ and }\quad i_{B, j}^*g^{\ker(\DD^{S\otimes\e; \wt{\Phi}})}=g^{\ker(\DD^{S\otimes E; \Phi_j}_j)}.$$
Let
\begin{equation}\label{eq 5.27}
P_{01}:\KER\big(\DD^{S\otimes E; \Phi_0}_0\big)\to\KER\big(\DD^{S\otimes E; \Phi_1}_1\big)
\end{equation}
be the $\Z_2$-graded morphism (\ref{eq 2.5}). Note that $\KK_j$ is the geometric bundle defining $\KER\big(\DD^{S\otimes E; \Phi_j}_j\big)$, and
\begin{equation}\label{eq 5.28}
\nabla^{\ker(\DD^{S\otimes E; \Phi_0}_0)}\quad\textrm{ and }\quad P_{01}^*\nabla^{\ker(\DD^{S\otimes E; \Phi_1}_1)}
\end{equation}
are $\Z_2$-graded unitary connections on $\ker\big(\DD^{S\otimes E; \Phi_0}_0\big)\to B$.

Define a $\Z_2$-graded geometric bundle $\bjj=(\jj, g^\jj, \nabla^\jj)$ over $\wt{B}$ by
$$\jj=p_B^*\ker\big(\DD^{S\otimes E; \Phi_0}_0\big),\qquad g^\jj=p_B^*g^{\ker(\DD^{S\otimes E; \Phi_0}_0)}$$
and $\nabla^\jj$ by (\ref{eq 2.6}) and (\ref{eq 2.7}) with respect to (\ref{eq 5.28}). Then
\begin{equation}\label{eq 5.29}
i_{B, 0}^*\nabla^\jj=\nabla^{\ker(\DD^{S\otimes E; \Phi_0}_0)}\quad\textrm{ and }\quad i_{B, 1}^*\nabla^\jj=P_{01}^*\nabla^{\ker(\DD^{S\otimes E; \Phi_1}_1)}.
\end{equation}
Since $i_{B, 0}\circ p_B:\wt{B}\to\wt{B}$ is smoothly homotopic to $\id_{\wt{B}}$,
$$\jj\cong\ker\big(\DD^{S\otimes\e; \wt{\Phi}}\big)$$
as $\Z_2$-graded complex vector bundles. By \cite[(2) of Remark 2.1]{H23}, let
$$\varphi:\KER\big(\DD^{S\otimes\e; \wt{\Phi}}\big)\to\bjj$$
be a $\Z_2$-graded morphism. Then $\wt{\nabla}^{\ker(\DD^{S\otimes\e; \wt{\Phi}})}:=\varphi^*\nabla^\jj$ is a $\Z_2$-graded unitary connection on $\ker\big(\DD^{S\otimes\e; \wt{\Phi}}\big)\to\wt{B}$.

Define a rescaled Bismut superconnection on 
$$\wt{\pi}_*\e\oplus\wh{\kk}\oplus\ker\big(\DD^{S\otimes\e; \wt{\Phi}}\big)^{\op}\to\wt{B}$$ 
by
\begin{equation}\label{eq 5.30}
\wh{\bbb}_t=\sqrt{t}\DDD^{S\otimes\e; \wt{\Phi}}(\alpha(t))+\big(\nabla^{\wt{\pi}_*\e, u}\oplus\nabla^{\wh{\kk}}\oplus\wt{\nabla}^{\ker(\DD^{S\otimes\e; \wt{\Phi}}), \op}\big)-\frac{c(\wt{T})}{4\sqrt{t}}.
\end{equation}
By (\ref{eq 3.22}),
\begin{equation}\label{eq 5.31}
\lim_{t\to\infty}\ch(\wh{\bbb}_t)=0.
\end{equation}
By (\ref{eq 3.23}),
\begin{equation}\label{eq 5.32}
\begin{aligned}[b]
\lim_{t\to 0}\ch(\wh{\bbb}_t)&=\int_{\wt{X}/\wt{B}}\todd(\nabla^{T^V\wt{X}})\wedge\ch(\nabla^\e)-\ch(\wt{\nabla}^{\ker(\DD^{S\otimes\e; \wt{\Phi}})})\\
&=\int_{\wt{X}/\wt{B}}\todd(\nabla^{T^V\wt{X}})\wedge\ch(\nabla^\e)-\ch(\nabla^\jj).
\end{aligned}
\end{equation}
Denote by $\wh{\eta}\big(\bee, \wt{\bpi}, \KER\big(\DD^{S\otimes\e; \wt{\Phi}}\big)\big)$ the Bismut--Cheeger eta form associated to $\wh{\bbb}_t$ (\ref{eq 5.30}). By (\ref{eq 5.31}) and (\ref{eq 5.32}),
\begin{equation}\label{eq 5.33}
d\wh{\eta}\big(\bee, \wt{\bpi}, \KER\big(\DD^{S\otimes\e; \wt{\Phi}}\big)\big)=\int_{\wt{X}/\wt{B}}\todd(\nabla^{T^V\wt{X}})\wedge\ch(\nabla^\e)-\ch(\nabla^\jj).
\end{equation}
By (\ref{eq 2.1}), (\ref{eq 5.33}) and (\ref{eq 5.29}),
\begin{equation}\label{eq 5.34}
\begin{aligned}[b]
&i_{B, 1}^*\wh{\eta}\big(\bee, \wt{\bpi}, \KER\big(\DD^{S\otimes\e; \wt{\Phi}}\big)\big)-i_{B, 0}^*\wh{\eta}\big(\bee, \wt{\bpi}, \KER\big(\DD^{S\otimes\e; \wt{\Phi}}\big)\big)\\
&\equiv-\int_{\wt{B}/B}d^{\wt{B}}\wh{\eta}\big(\bee, \wt{\bpi}, \KER\big(\DD^{S\otimes\e; \wt{\Phi}}\big)\big)\\
&\equiv-\int_{\wt{B}/B}\bigg(\int_{\wt{X}/\wt{B}}\todd(\nabla^{T^V\wt{X}})\wedge\ch(\nabla^\e)-\ch(\nabla^\jj)\bigg)\\
&\equiv-\int_{\wt{B}/B}\int_{\wt{X}/\wt{B}}\todd(\nabla^{T^V\wt{X}})\wedge\ch(\nabla^\e)-\CS\big(\nabla^{\ker(\DD^{S\otimes E; \Phi_0}_0)}, P_{01}^*\nabla^{\ker(\DD^{S\otimes E; \Phi_1}_1)}\big).
\end{aligned}
\end{equation}
Since $\int_{X/B}\circ\int_{\wt{X}/X}=\int_{\wt{X}/B}=\int_{\wt{B}/B}\circ\int_{\wt{X}/\wt{B}}$, it follows from the definition of the smooth path $\wt{c}(t)$ that
\begin{displaymath}
\begin{split}
&-\int_{\wt{B}/B}\int_{\wt{X}/\wt{B}}\todd(\nabla^{T^V\wt{X}})\wedge\ch(\nabla^\e)\\
&=-\int_{X/B}\int_{\wt{X}/X}\todd(\nabla^{T^V\wt{X}})\wedge\ch(\nabla^\e)\\
&\equiv\int_{X/B}\big(T\wh{A}(\nabla^{T^VX}_0, \nabla^{T^VX}_1)\wedge e^{\frac{1}{2}c_1(\nabla^\lambda_0)}+\wh{A}(\nabla^{T^VX}_1)\wedge e^{\frac{1}{2}Tc_1(\nabla^\lambda_0, \nabla^\lambda_1)}\big)\wedge\ch(\nabla^E_0)\\
&\qquad+\int_{X/B}\todd(\nabla^{T^VX}_1)\wedge\CS(\nabla^E_0, \nabla^E_1).
\end{split}
\end{displaymath}
Thus (\ref{eq 5.34}) becomes
\begin{equation}\label{eq 5.35}
\begin{aligned}[b]
&i_{B, 1}^*\wh{\eta}\big(\bee, \wt{\bpi}, \KER\big(\DD^{S\otimes\e; \wt{\Phi}}\big)\big)-i_{B, 0}^*\wh{\eta}\big(\bee, \wt{\bpi}, \KER\big(\DD^{S\otimes\e; \wt{\Phi}}\big)\big)\\
&\equiv\int_{X/B}\big(T\wh{A}(\nabla^{T^VX}_0, \nabla^{T^VX}_1)\wedge e^{\frac{1}{2}c_1(\nabla^\lambda_0)}+\wh{A}(\nabla^{T^VX}_1)\wedge e^{\frac{1}{2}Tc_1(\nabla^\lambda_0, \nabla^\lambda_1)}\big)\wedge\ch(\nabla^E_0)\\
&\qquad+\int_{X/B}\todd(\nabla^{T^VX}_1)\wedge\CS(\nabla^E_0, \nabla^E_1)-\CS\big(\nabla^{\ker(\DD^{S\otimes E; \Phi_0}_0)}, P_{01}^*\nabla^{\ker(\DD^{S\otimes E; \Phi_1}_1)}\big).
\end{aligned}
\end{equation}

We claim that for $j\in\set{0, 1}$,
\begin{equation}\label{eq 5.36}
i_{B, j}^*\wh{\eta}\big(\bee, \wt{\bpi}, \KER\big(\DD^{S\otimes\e; \wt{\Phi}}\big)\big)\equiv\wt{\eta}(\EE_j, \bpi_j, \KK_j, r_j).
\end{equation}
We prove (\ref{eq 5.36}) for $j=0$. The argument for $j=1$ is similar. Define a $\Z_2$-graded morphism
$\wh{\varphi}:\wt{\pi}_*\bee\oplus\wh{\bkk}\oplus\KER\big(\DD^{S\otimes\e; \wt{\Phi}}\big)^{\op}\to\wt{\pi}_*\bee\oplus\wh{\bkk}\oplus\bjj^{\op}$ by $\wh{\varphi}:=\id_{\wt{\pi}_*\bee}\oplus\id_{\wh{\bkk}}\oplus\varphi^{\op}$. By (\ref{eq 5.30}),
\begin{equation}\label{eq 5.37}
(\wh{\varphi}^{-1})^*\wh{\bbb}_t=\sqrt{t}\DDD^{S\otimes\e; \wt{\Phi}}(\alpha(t))+\big(\nabla^{\wt{\pi}_*\e, u}\oplus\nabla^{\wh{\kk}}\oplus\nabla^{\jj, \op}\big)-\frac{c(\wt{T})}{4\sqrt{t}}.
\end{equation}
It follows from (\ref{eq 5.37}) and (\ref{eq 5.28}) that
\begin{equation}\label{eq 5.38}
\begin{aligned}[b]
i_{B, 0}^*\big((\wh{\varphi}^{-1})^*\wh{\bbb}_t\big)&=\sqrt{t}\DDD^{S\otimes E; \Phi_0}_0(\alpha(t))+\big(\nabla^{\pi_*E, u}\oplus\nabla^{\wh{K}_0}\oplus\nabla^{\ker(\DD^{S\otimes E; \Phi_0}_0), \op}\big)-\frac{c(T)}{4\sqrt{t}}\\
&=\wh{\bbb}^{E; \Phi_0}_{0, t}.
\end{aligned}
\end{equation}
Let $t<T\in(0, \infty)$. By (\ref{eq 5.38}) and (\ref{eq 2.14}),
\begin{equation}\label{eq 5.39}
\begin{aligned}[b]
i_{B, 0}^*\CS\big(\wh{\bbb}_t, \wh{\bbb}_T\big)&\equiv i_{B, 0}^*\CS\big((\wh{\varphi}^{-1})^*\wh{\bbb}_t, (\wh{\varphi}^{-1})^*\wh{\bbb}_T\big)\\
&\equiv\CS\big(i_{B, 0}^*(\wh{\varphi}^{-1})^*\wh{\bbb}_t, i_{B, 0}^*(\wh{\varphi}^{-1})^*\wh{\bbb}_T\big)\\
&\equiv\CS\big(\wh{\bbb}^{E, \Phi_0}_{0, t}, \wh{\bbb}^{E, \Phi_0}_{0, T}\big).
\end{aligned}
\end{equation}
By letting $t\to 0$ and $T\to\infty$ in (\ref{eq 5.39}), it follows from Lemma \ref{lemma 4.1} that it becomes
\begin{displaymath}
\begin{split}
i_{B, 0}^*\wh{\eta}\big(\bee, \wt{\bpi}, \KER\big(\DD^{S\otimes\e; \wt{\Phi}}\big)\big)&\equiv\wh{\eta}\big(\EE_0, \bpi, \KER\big(\DD^{S\otimes E; \Phi_0}_0\big)\big)\\
&\equiv\wt{\eta}(\EE_0, \bpi, \KK_0, r_0).
\end{split}
\end{displaymath}
Thus (\ref{eq 5.36}) holds, and therefore (\ref{eq 5.35}) becomes
\begin{equation}\label{eq 5.40}
\begin{aligned}
&\wt{\eta}(\EE_1, \bpi, \KK_1, r_1)-\wt{\eta}(\EE_0, \bpi, \KK_0, r_0)\\
&\equiv\int_{X/B}\big(T\wh{A}(\nabla^{T^VX}_0, \nabla^{T^VX}_1)\wedge e^{\frac{1}{2}c_1(\nabla^\lambda_0)}+\wh{A}(\nabla^{T^VX}_1)\wedge e^{\frac{1}{2}Tc_1(\nabla^\lambda_0, \nabla^\lambda_1)}\big)\wedge\ch(\nabla^E_0)\\
&\qquad+\int_{X/B}\todd(\nabla^{T^VX}_1)\wedge\CS(\nabla^E_0, \nabla^E_1)-\CS\big(\nabla^{\ker(\DD^{S\otimes E; \Phi_0}_0)}, P_{01}^*\nabla^{\ker(\DD^{S\otimes E; \Phi_1}_1)}\big).
\end{aligned}
\end{equation}

Since $(H_0, s_0)$ and $(K_0, r_0)$ satisfy the AS property for $\DD^{S\otimes E}_0$, it follows from Proposition \ref{prop 5.1} that there exists a $\Z_2$-graded morphism
\begin{equation}\label{eq 5.41}
P_{s_0r_0}:\KER\big(\DD^{S\otimes E; \Delta_0}_0\big)\oplus\wh{\KK}_0\to\KER\big(\DD^{S\otimes E; \Phi_0}_0\big)\oplus\wh{\HH}_0
\end{equation}
such that
\begin{equation}\label{eq 5.42}
\wt{\eta}(\EE_0, \bpi, \KK_0, r_0)-\wt{\eta}(\EE_0, \bpi, \HH_0, s_0)\equiv-\CS\big(\nabla^{\ker(\DD^{S\otimes E; \Delta_0}_0)}\oplus\nabla^{\wh{K}_0}, P_{s_0r_0}^*(\nabla^{\ker(\DD^{S\otimes E; \Phi_0})}\oplus\nabla^{\wh{H}_0})\big).
\end{equation}
Similarly, since $(H_1, s_1)$ and $(K_1, r_1)$ satisfy the AS property for $\DD^{S\otimes E}_1$, it follows from Proposition \ref{prop 5.1} that there exists a $\Z_2$-graded morphism
\begin{equation}\label{eq 5.43}
P_{s_1r_1}:\KER\big(\DD^{S\otimes E; \Delta_1}_1\big)\oplus\wh{\KK}_1\to\KER\big(\DD^{S\otimes E; \Phi_1}_1\big)\oplus\wh{\HH}_1
\end{equation}
such that
\begin{equation}\label{eq 5.44}
\wt{\eta}(\EE_1, \bpi, \KK_1, r_1)-\wt{\eta}(\EE_1, \bpi, \HH_1, s_1)\equiv-\CS\big(\nabla^{\ker(\DD^{S\otimes E; \Delta_1}_1)}\oplus\nabla^{\wh{K}_1}, P_{s_1r_1}^*(\nabla^{\ker(\DD^{S\otimes E; \Phi_1}_1)}\oplus\nabla^{\wh{H}_1})\big).
\end{equation}
By (\ref{eq 5.27}), (\ref{eq 5.41}) and (\ref{eq 5.43}), the composition
\begin{center}
\begin{tikzcd}
\KER\big(\DD^{S\otimes E; \Delta_0}_0\big)\oplus\wh{\KK}_0\oplus\wh{\HH}_1 \arrow{d}[swap]{P_{s_0r_0}\oplus\id_{\wh{\HH}_1}} \arrow[ddrr, dashed]{}{h_{01}} & & \\ \KER\big(\DD^{S\otimes E; \Phi_0}_0\big)\oplus\wh{\HH}_0\oplus\wh{\HH}_1 \arrow{d}[swap]{P_{01}\oplus\id_{\wh{\HH}_0}\oplus\id_{\wh{\HH}_1}} & & \\
\KER\big(\DD^{S\otimes E; \Phi_1}_1\big)\oplus\wh{\HH}_0\oplus\wh{\HH}_1 \arrow{rr}[swap]{P_{s_1r_1}^{-1}\oplus\id_{\wh{\HH}_0}} & & \KER\big(\DD^{S\otimes E; \Delta_1}_1\big)\oplus\wh{\KK}_1\oplus\wh{\HH}_0
\end{tikzcd}
\end{center}
is a $\Z_2$-graded morphism. Set $\VV_0=\KK_0\oplus\HH_1$, $\VV_1=\KK_1\oplus\HH_0$, and write
$$h_{01}:\KER\big(\DD^{S\otimes E; \Delta_0}_0\big)\oplus\wh{\VV}_0\to\KER\big(\DD^{S\otimes E; \Delta_1}_1\big)\oplus\wh{\VV}_1$$
for the above composition. By (\ref{eq 5.40}), (\ref{eq 5.42}) and (\ref{eq 5.44}),
\begin{equation}\label{eq 5.45}
\begin{aligned}[b]
&\wt{\eta}(\EE_1, \bpi, \HH_1, s_1)-\wt{\eta}(\EE_0, \bpi, \HH_0, s_0)\\
&\equiv\big(\wt{\eta}(\EE_1, \bpi, \HH_1, s_1)-\wt{\eta}(\EE_1, \bpi, \KK_1, r_1)\big)+\big(\wt{\eta}(\EE_1, \bpi, \KK_1, r_1)-\wt{\eta}(\EE_0, \bpi, \KK_0, r_0)\big)\\
&\quad+\big(\wt{\eta}(\EE_0, \bpi, \KK_0, r_0)-\wt{\eta}(\EE_0, \bpi, \HH_0, s_0)\big)\\
&\equiv\int_{X/B}\big(T\wh{A}(\nabla^{T^VX}_0, \nabla^{T^VX}_1)\wedge e^{\frac{1}{2}c_1(\nabla^\lambda_0)}+\wh{A}(\nabla^{T^VX}_1)\wedge e^{\frac{1}{2}Tc_1(\nabla^\lambda_0, \nabla^\lambda_1)}\big)\wedge\ch(\nabla^E_0)\\
&\quad+\int_{X/B}\todd(\nabla^{T^VX}_1)\wedge\CS(\nabla^E_0, \nabla^E_1)-\CS\big(\nabla^{\ker(\DD^{S\otimes E; \Phi_0}_0)}, P_{01}^*\nabla^{\ker(\DD^{S\otimes E; \Phi_1}_1)}\big)\\
&\qquad+\CS\big(\nabla^{\ker(\DD^{S\otimes E; \Delta_1}_1)}\oplus\nabla^{\wh{K}_1}, P_{s_1r_1}^*(\nabla^{\ker(\DD^{S\otimes E; \Phi_1})}\oplus\nabla^{\wh{H}_1})\big)\\
&\qquad-\CS\big(\nabla^{\ker(\DD^{S\otimes E; \Delta_0}_0)}\oplus\nabla^{\wh{K}_0}, P_{s_0r_0}^*(\nabla^{\ker(\DD^{S\otimes E; \Phi_0})}\oplus\nabla^{\wh{H}_0})\big).
\end{aligned}
\end{equation}
Write $\HH=\HH_0\oplus\HH_1$. By (\ref{eq 2.12}), (\ref{eq 2.14}), (\ref{eq 2.11}) and (\ref{eq 2.9}), the sum of the last three terms of the right-hand side of (\ref{eq 5.45}) becomes
\begin{equation}\label{eq 5.46}
\begin{aligned}[b]
&-\CS\big(\nabla^{\ker(\DD^{S\otimes E; \Phi_0}_0)}, P_{01}^*\nabla^{\ker(\DD^{S\otimes E; \Phi_1}_1)}\big)\\
&\quad+\CS\big(\nabla^{\ker(\DD^{S\otimes E; \Delta_1}_1)}\oplus\nabla^{\wh{K}_1}, P_{s_1r_1}^*(\nabla^{\ker(\DD^{S\otimes E; \Phi_1})}\oplus\nabla^{\wh{H}_1})\big)\\
&\qquad-\CS\big(\nabla^{\ker(\DD^{S\otimes E; \Delta_0}_0)}\oplus\nabla^{\wh{K}_0}, P_{s_0r_0}^*(\nabla^{\ker(\DD^{S\otimes E; \Phi_0})}\oplus\nabla^{\wh{H}_0})\big)\\
&\equiv-\CS\big(\nabla^{\ker(\DD^{S\otimes E; \Delta_0}_0)}\oplus\nabla^{\wh{V}_0}, (P_{s_0r_0}\oplus\id_{\wh{\HH}_1})^*(\nabla^{\ker(\DD^{S\otimes E; \Phi_0}_0)}\oplus\nabla^{\wh{H}})\big)\\
&\quad-\CS\big(\nabla^{\ker(\DD^{S\otimes E; \Phi_0}_0)}\oplus\nabla^{\wh{H}}, (P_{01}\oplus\id_{\wh{\HH}})^*(\nabla^{\ker(\DD^{S\otimes E; \Phi_1}_1)}\oplus\nabla^{\wh{H}})\big)\\
&\qquad-\CS\big(\nabla^{\ker(\DD^{S\otimes E; \Phi_1}_1)}\oplus\nabla^{\wh{H}}, (P_{s_1r_1}^{-1}\oplus\id_{\wh{\HH}_0})^*(\nabla^{\ker(\DD^{S\otimes E; \Delta_1})}\oplus\nabla^{\wh{V}_1})\big)\\
&\equiv-\CS\big(\nabla^{\ker(\DD^{S\otimes E; \Delta_0}_0)}\oplus\nabla^{\wh{V}_0}, h_{01}^*(\nabla^{\ker(\DD^{S\otimes E; \Delta_1})}\oplus\nabla^{\wh{V}_1})\big).
\end{aligned}
\end{equation}
It follows from (\ref{eq 5.45}) and (\ref{eq 5.46}) that (\ref{eq 5.24}) holds.
\end{proof}
\begin{prop}\label{prop 5.3}
Let $\pi:X\to B$ be a submersion with closed, oriented and spin$^c$ fibers of even dimension, equipped with a Riemannian and differential spin$^c$ structure $\bpi$, and let $\EE$ and $\FF$ be geometric bundles over $X$. Denote by $\DD^{S\otimes E}$ and $\DD^{S\otimes F}$ the twisted spin$^c$ Dirac operators associated to $(\EE, \bpi)$ and $(\FF, \bpi)$, respectively. Let $(H_F, s_F)$ satisfy the AS property for $\DD^{S\otimes F}$ and $\HH_F$ be a geometric bundle defining $\KER\big(\DD^{S\otimes F; \Delta_F}\big)$, where $\Delta_F$ is the odd self-adjoint map associated to $(H_F, s_F)$ as in (\ref{eq 3.3}). If there exists a morphism $\alpha:\EE\to\FF$ such that $\nabla^E=\alpha^*\nabla^F$, then the pair $(H_F, s_\alpha)$, where the map $s_\alpha:H_F\to(\pi_*E)^-$ is defined by
\begin{equation}\label{eq 5.47}
s_\alpha:=\wt{\alpha}^{-1}_-\circ s_F,
\end{equation}
satisfies the AS property for $\DD^{S\otimes E}$ and $\HH_F$ is the geometric bundle defining $\KER\big(\DD^{S\otimes E; \Delta_\alpha}\big)$, where $\Delta_\alpha$ is the odd self-adjoint map associated to $(H_F, s_\alpha)$ as in (\ref{eq 3.3}), and there exists a uniquely determined $\Z_2$-graded morphism
\begin{equation}\label{eq 5.48}
\wh{\alpha}:\KER\big(\DD^{S\otimes E; \Delta_\alpha}\big)\to\KER\big(\DD^{S\otimes F; \Delta_F}\big)
\end{equation}
such that $\wt{\eta}(\EE, \bpi, \HH_F, s_\alpha)=\wt{\eta}(\FF, \bpi, \HH_F, s_F)$.
\end{prop}

Since the proof of Proposition \ref{prop 5.3} is similar to the corresponding result in the MFFL approach \cite[Proposition 3.3]{H23}, it is omitted.

By applying Propositions \ref{prop 5.2} and \ref{prop 5.3}, we obtain the following extended variational formula for the Bismut--Cheeger eta form in the ASGL approach.
\begin{prop}\label{prop 5.4}
Let $\pi:X\to B$ be a submersion with closed, oriented and spin$^c$ fibers of even dimension, equipped with Riemannian and differential spin$^c$ structures $\bpi_0$ and $\bpi_1$ with the same underlying topological spin$^c$ structures, and let $\EE$ and $\FF$ be geometric bundles over $X$. Denote by $\DD^{S\otimes E}$ and $\DD^{S\otimes F}$ the twisted spin$^c$ Dirac operators associated to $(\EE, \bpi_0)$ and $(\FF, \bpi_1)$, respectively. Let $(H_E, s_E)$ and $(H_F, s_F)$ satisfy the AS property for $\DD^{S\otimes E}$ and $\DD^{S\otimes F}$, and let $\HH_E$ and $\HH_F$ be geometric bundles defining $\KER\big(\DD^{S\otimes E; \Delta_E}\big)$ and $\KER\big(\DD^{S\otimes F; \Delta_F}\big)$, where $\Delta_E$ and $\Delta_F$ are the odd self-adjoint maps associated to $(H_E, s_E)$ and $(H_F, s_F)$ as in (\ref{eq 3.3}), respectively. If there exists a morphism $\alpha:\EE\to\FF$, then there exist geometric bundles $\VV_0$ and $\VV_1$ over $B$ and a $\Z_2$-graded morphism
$$h_{01}:\KER\big(\DD^{S\otimes E; \Delta_E}\big)\oplus\wh{\VV}_0\to\KER\big(\DD^{S\otimes F; \Delta_F}\big)\oplus\wh{\VV}_1$$
such that
\begin{displaymath}
\begin{split}
&\wt{\eta}(\FF, \bpi_1, \HH_F, s_F)-\wt{\eta}(\EE, \bpi_0, \HH_E, s_E)\\
&\equiv\int_{X/B}\big(T\wh{A}(\nabla^{T^VX}_0, \nabla^{T^VX}_1)\wedge e^{\frac{1}{2}c_1(\nabla^\lambda_0)}+\wh{A}(\nabla^{T^VX}_1)\wedge e^{\frac{1}{2}Tc_1(\nabla^\lambda_0, \nabla^\lambda_1)}\big)\wedge\ch(\nabla^E)\\
&\quad+\int_{X/B}\todd(\nabla^{T^VX}_1)\wedge\CS(\nabla^E, \alpha^*\nabla^F)-\CS\big(\nabla^{\ker(\DD^{S\otimes E; \Delta_E})}\oplus\nabla^{\wh{V}_0}, h_{01}^*(\nabla^{\ker(\DD^{S\otimes F; \Delta_F})}\oplus\nabla^{\wh{V}_1})\big).
\end{split}
\end{displaymath}
\end{prop}

The following proposition is the additivity of geometric kernel bundles and the Bismut--Cheeger eta form in the ASGL approach. Since the proof is similar to the corresponding result in the MFFL approach \cite[Proposition 4.1]{H23}, it is omitted.
\begin{prop}\label{prop 5.5}
Let $\pi:X\to B$ be a submersion with closed, oriented and spin$^c$ fibers of even dimension, equipped with a Riemannian and differential spin$^c$ structure $\bpi$, and let $\EE$ and $\FF$ be geometric bundles over $X$. Denote by $\DD^{S\otimes E}$ and $\DD^{S\otimes F}$ the twisted spin$^c$ Dirac operators associated to $(\EE, \bpi)$ and $(\FF, \bpi)$, respectively. Let $(H_E, s_E)$ and $(H_F, s_F)$ satisfy the AS property for $\DD^{S\otimes E}$ and $\DD^{S\otimes F}$, and let $\HH_E$ and $\HH_F$ be geometric bundles defining $\KER\big(\DD^{S\otimes E; \Delta_E}\big)$ and $\KER\big(\DD^{S\otimes F; \Delta_F}\big)$, where $\Delta_E$ and $\Delta_F$ are the odd self-adjoint maps associated to $(H_E, s_E)$ and $(H_F, s_F)$ as in (\ref{eq 3.3}), respectively. Then $(H_E\oplus H_F, s_E\oplus s_F)$ satisfies the AS property for $\DD^{S\otimes(E\oplus F)}$ and $\HH_E\oplus\HH_F$ is the geometric bundle defining $\KER\big(\DD^{S\otimes(E\oplus F); \Delta_{E\oplus F}}\big)$, where $\Delta_{E\oplus F}$ is the odd self-adjoint map associated to $(H_E\oplus H_F, s_E\oplus s_F)$ as in (\ref{eq 3.3}). Moreover, there exists a uniquely determined $\Z_2$-graded morphism
\begin{equation}\label{eq 5.49}
\wh{h}:\KER\big(\DD^{S\otimes(E\oplus F); \Delta_{E\oplus F}}\big)\to\KER\big(\DD^{S\otimes E; \Delta_E}\big)\oplus\KER\big(\DD^{S\otimes F; \Delta_F}\big)
\end{equation}
such that
$$\nabla^{\ker(\DD^{S\otimes(E\oplus F); \Delta_{E\oplus F}})}=\wh{h}^*\big(\nabla^{\ker(\DD^{S\otimes E; \Delta_E})}\oplus\nabla^{\ker(\DD^{S\otimes F; \Delta_F})}\big)$$
and
$$\wt{\eta}(\EE\oplus\FF, \bpi, \HH_E\oplus\HH_F, s_E\oplus s_F)\equiv\wt{\eta}(\EE, \bpi, \HH_E, s_E)+\wt{\eta}(\FF, \bpi, \HH_F, s_F).$$
\end{prop}

\begin{remark}\label{remark 5.1}
Here are two remarks concerning the results in this section.
\begin{enumerate}
  \item Propositions \ref{prop 5.4} and \ref{prop 5.5} can be used to give a proof that the analytic index $\wt{\ind}^a_{\wh{K}}$ (\ref{eq 4.2}) in differential $K$-theory defined via the ASGL approach is a well defined group homomorphism.
  \item As mentioned in \S\ref{s 1.2}, by the result of Atiyah--Singer \cite[Proposition 2.2]{AS71}, the geometric bundles $\KK$ in Lemma \ref{lemma 5.1}, $\HH_0$ and $\HH_1$ in both Propositions \ref{prop 5.1} and \ref{prop 5.2}, and $\HH_F$ in Proposition \ref{prop 5.3} can always be taken as trivial geometric bundles. Thus the geometric bundles $\VV_0$ and $\VV_1$ in Proposition \ref{prop 5.2}, and hence in Proposition \ref{prop 5.4}, can also be taken as trivial geometric bundles. This fact, together with Lemma \ref{lemma 5.1}, plays an essential role when defining the analytic index $\ind^a_k$ at the cocycle level.
\end{enumerate}
\end{remark}

\section{Analytic index and RRG-type formula in odd $\Z/k\Z$ $K$-theory}\label{s 6}

This section contains the main results of this paper. We first review Karoubi's construction of the odd $\Z/k\Z$ $K$-theory group and construct a group homomorphism from the odd $\Z/k\Z$ $K$-theory group to the $\R/\Z$ $K$-theory group in \S\ref{s 6.1}. We define the analytic index $\ind^a_k$ at the cocycle level and prove the RRG-type formula in odd $\Z/k\Z$ $K$-theory in \S\ref{s 6.2}. Finally, we show that the analytic index $\ind^a_k$ in odd $\Z/k\Z$ $K$-theory refines the underlying geometric bundle of the analytic index $\ind^a_{\R/\Z}$ in $\R/\Z$ $K$-theory, and the RRG-type formula in odd $\Z/k\Z$ $K$-theory is a differential form representative of the RRG theorem in $\R/\Z$ $K$-theory in \S\ref{s 6.3}.

\subsection{Odd $\Z/k\Z$ $K$-theory}\label{s 6.1}

Since $\bun_\nabla(X)$ is a Banach category and $\varphi_{k, \nabla}:\bun_\nabla(X)\to\bun_\nabla(X)$, given by 
$$\varphi_{k, \nabla}(E)=kE,\qquad\varphi_{k, \nabla}(T)=kT,$$ 
is a quasi-surjective Banach functor, by applying \cite[Example 2.11 and Exercise 6.18 of Chapter II]{K08} to $\bun_\nabla(X)$ and $\varphi_{k, \nabla}$, we obtain the odd $\Z/k\Z$ $K$-theory group $K^{-1}(X; \Z/k\Z)$. In more detail, define a set by
$$\Gamma(\varphi_{k, \nabla})=\set{(\EE, \FF, \alpha)|\EE, \FF\in\bun_\nabla(X), \alpha\in\ho(k\EE, k\FF)}.$$
Elements in $\Gamma(\varphi_{k, \nabla})$ are called $k$-cocycles over $X$. Two $k$-cocycles $(\EE_0, \FF_0, \alpha_0)$ and $(\EE_1, \FF_1, \alpha_1)$ are isomorphic, written as $(\EE_0, \FF_0, \alpha_0)\cong(\EE_1, \FF_1, \alpha_1)$, if there exist morphisms $f_+:\EE_0\to\EE_1$ and $f_-:\FF_0\to\FF_1$ such that the diagram
\begin{center}
\begin{tikzcd}
k\EE_0 \arrow{r}{\alpha_0} \arrow{d}[swap]{kf_+} & k\FF_0 \arrow{d}{kf_-} \\ k\EE_1 \arrow{r}[swap]{\alpha_1} & k\FF_1
\end{tikzcd}
\end{center}
commutes. A $k$-cocycle $(\EE, \FF, \alpha)$ is said to be elementary if $\EE=\FF$ and $\alpha:k\EE\to k\EE$ is smoothly homotopic to $\id_{k\EE}$ through morphisms. Define a sum on $\Gamma(\varphi_{k, \nabla})$ by
$$(\EE_0, \FF_0, \alpha_0)+(\EE_1, \FF_1, \alpha_1):=(\EE_0\oplus\EE_1, \FF_0\oplus\FF_1, \alpha_0\oplus\alpha_1).$$
Note that $\Gamma(\varphi_{k, \nabla})$ is an abelian semigroup.

Define an equivalence relation $\sim$ on $\Gamma(\varphi_{k, \nabla})$ by $(\EE_0, \FF_0, \alpha_0)\sim(\EE_1, \FF_1, \alpha_1)$ if and only if there exist elementary $k$-cocycles $(\WW_0, \WW_0, \gamma_0)$ and $(\WW_1, \WW_1, \gamma_1)$ such that
$$(\EE_0, \FF_0, \alpha_0)+(\WW_0, \WW_0, \gamma_0)\cong(\EE_1, \FF_1, \alpha_1)+(\WW_1, \WW_1, \gamma_1).$$
The odd $\Z/k\Z$ $K$-theory group $K^{-1}(X; \Z/k\Z)$ is defined to be the quotient $\Gamma(\varphi_{k, \nabla})/\sim$. Denote by $[\EE, \FF, \alpha]$ the class of $(\EE, \FF, \alpha)$ in $K^{-1}(X; \Z/k\Z)$. Note that 
$$[\EE_0, \FF_0, \alpha_0]=[\EE_1, \FF_1, \alpha_1]\qquad\textrm{ in } K^{-1}(X; \Z/k\Z)$$ 
if and only if there exist elementary $k$-cocycles $(\WW_0, \WW_0, \gamma_0)$ and $(\WW_1, \WW_1, \gamma_1)$ and morphisms $f_+:\EE_0\oplus\WW_0\to\EE_1\oplus\WW_1$ and $f_-:\FF_0\oplus\WW_0\to\FF_1\oplus\WW_1$ such that the diagram
\begin{equation}\label{eq 6.1}
\begin{tikzcd}
k\EE_0\oplus k\WW_0 \arrow{r}{\alpha_0\oplus\gamma_0} \arrow{d}[swap]{kf_+} & k\FF_0\oplus k\WW_0 \arrow{d}{kf_-} \\ k\EE_1\oplus k\WW_1 \arrow{r}[swap]{\alpha_1\oplus\gamma_1} & k\FF_1\oplus k\WW_1
\end{tikzcd}
\end{equation}
commutes.

\begin{remark}\label{remark 6.1}
Temporarily denote by $K^{-1}_\nabla(X; \Z/k\Z)$ and $K^{-1}(X; \Z/k\Z)$ the odd $\Z/k\Z$ $K$-theory groups defined by $(\bun_\nabla(X), \varphi_{k, \nabla})$ and $(\bun(X), \varphi_k)$ using Karoubi's construction, respectively. Since there is an obvious forgetful functor $\Phi:\bun_\nabla(X)\to\bun(X)$ given by 
$$\Phi(\EE)=E,\qquad\Phi(T)=\wt{T},$$ 
where $T:\EE\to\FF$ is a morphism in $\bun_\nabla(X)$ and $\wt{T}:E\to F$ is the isomorphism between the underlying bundles, there is an obvious surjective group homomorphism $\wh{\Phi}:K^{-1}_\nabla(X; \Z/k\Z)\to K^{-1}(X; \Z/k\Z)$. Since metrics and connections are considered part of the data in $K^{-1}_\nabla(X; \Z/k\Z)$, for the purpose in this paper it is better to work with $K^{-1}_\nabla(X; \Z/k\Z)$ than $K^{-1}(X; \Z/k\Z)$.
\end{remark}

We associate to every $k$-cocycle $(\EE, \FF, \alpha)$ over $X$ the Cheeger--Chern--Simons form $\CCS_k(\EE, \FF, \alpha)$, defined by
$$\CCS_k(\EE, \FF, \alpha):=\frac{1}{k}\CS\big(k\nabla^E, \alpha^*(k\nabla^F)\big)\in\frac{\Omega^{\odd}(X)}{\im(d)}.$$
Since morphisms in $\bun_\nabla(X)$ are not assumed to preserve unitary connections, i.e. $k\nabla^E\neq\alpha^*(k\nabla^F)$ in general, it follows that
$$d\CCS_k(\EE, \FF, \alpha)=\frac{1}{k}\big(\ch(k\nabla^F)-\ch(k\nabla^E)\big)\neq 0.$$
Thus $\CCS_k(\EE, \FF, \alpha)$ is not a closed form.

We associate to every $k$-cocycle $(\EE, \FF, \alpha)$ over $X$ a pair
\begin{equation}\label{eq 6.2}
T(\EE, \FF, \alpha):=\big(\VV, \CCS_k(\EE, \FF, \alpha)\big),
\end{equation}
where $\VV$ is the $\Z_2$-graded geometric bundle defined by $\VV^+=\EE$ and $\VV^-=\FF$. Since $-d(\CCS_k(\EE, \FF, \alpha))=\ch(\nabla^E)-\ch(\nabla^F)=\ch(\nabla^V)$, it follows from (\ref{eq 1.2}) that the right-hand side of (\ref{eq 6.2}) is a $\Z_2$-graded generator of $K^{-1}(X; \R/\Z)$. If the $k$-cocycle $(\EE, \FF, \alpha)$ is $\Z_2$-graded, define
\begin{equation}\label{eq 6.3}
T(\EE, \FF, \alpha):=\big(\EE\oplus\FF^{\op}, \CCS_k(\EE, \FF, \alpha)\big).
\end{equation}
Similarly, the right-hand side of (\ref{eq 6.3}) is a $\Z_2$-graded generator of $K^{-1}(X; \R/\Z)$. Note that if $(\EE, \FF, \alpha)$ is a $\Z_2$-graded $k$-cocycle so that $\EE^-$ and $\FF^-$ are the zero geometric bundles and $\alpha_-$ is the zero morphism, then the right-hand side of (\ref{eq 6.3}) reduces to that of (\ref{eq 6.2}).

\begin{prop}\label{prop 6.1}
The map $T:K^{-1}(X; \Z/k\Z)\to K^{-1}(X; \R/\Z)$ defined by
$$T([\EE, \FF, \alpha]):=T(\EE, \FF, \alpha),$$
where $(\EE, \FF, \alpha)\in[\EE, \FF, \alpha]$, is a well defined group homomorphism.
\end{prop}
\begin{proof}
We prove the statement for $T$ defined for $k$-cocycles (\ref{eq 6.2}). The proof for $T$ defined for $\Z_2$-graded $k$-cocycles (\ref{eq 6.3}) is similar. We first show that $T$ is a well defined map, i.e. if $[\EE_0, \FF_0, \alpha_0]=[\EE_1, \FF_1, \alpha_1]$ in $K^{-1}(X; \Z/k\Z)$, then
\begin{equation}\label{eq 6.4}
T(\EE_0, \FF_0, \alpha_0)=T(\EE_1, \FF_1, \alpha_1)
\end{equation}
in $K^{-1}(X; \R/\Z)$.

Let $j\in\set{0, 1}$. By (\ref{eq 6.2}), write $T(\EE_j, \FF_j, \alpha_j)=\big(\VV_j, \CCS_k(\EE_j, \FF_j, \alpha_j)\big)$. 
By assumption, there exist elementary $k$-cocycles $(\WW_0, \WW_0, \gamma_0)$ and $(\WW_1, \WW_1, \gamma_1)$ and morphisms $f_+:\EE_0\oplus\WW_0\to\EE_1\oplus\WW_1$ and $f_-:\FF_0\oplus\WW_0\to\FF_1\oplus\WW_1$ such that (\ref{eq 6.1}) commutes. To prove (\ref{eq 6.4}), it suffices to show that
\begin{equation}\label{eq 6.5}
\CCS_k(\EE_0, \FF_0, \alpha_0)-\CCS_k(\EE_1, \FF_1, \alpha_1)\equiv\CS\big(\nabla^{V_0}\oplus\nabla^{\wh{W}_0}, f^*(\nabla^{V_1}\oplus\nabla^{\wh{W}_1})\big),
\end{equation}
where $f:\VV_0\oplus\wh{\WW}_0\to\VV_1\oplus\wh{\WW}_1$ is the $\Z_2$-graded morphism defined by $f=f_+\oplus f_-$.

Note that $(kW_j, \gamma_j)$ defines an element in $K^{-1}(X)$. Recall that the odd Chern character form $\ch^{\odd}(\gamma_j, k\nabla^{W_j})$ of $(kW_j, \gamma_j)$ is defined to be 
$$\ch^{\odd}(\gamma_j, k\nabla^{W_j}):=\CS\big(k\nabla^{W_j}, \gamma_j^*(k\nabla^{W_j})\big).$$ 
Since $\gamma_j$ is smoothly homotopic to $\id_{kW_j}$ through $\Aut(kW_j)$, it follows from \cite[p.491]{G93} that
\begin{equation}\label{eq 6.6}
\begin{aligned}[b]
\CS\big(k\nabla^{W_j}, \gamma_j^*(k\nabla^{W_j})\big)&=\ch^{\odd}(\gamma_j, k\nabla^{W_j})\\
&\equiv\ch^{\odd}(\id_{kW_j}, k\nabla^{W_j})\\
&\equiv 0.
\end{aligned}
\end{equation}
Since (\ref{eq 6.1}) commutes, i.e. $(kf_-)\circ(\alpha_0\oplus\gamma_0)=(\alpha_1\oplus\gamma_1)\circ(kf_+)$, it follows that $$(\alpha_0\oplus\gamma_0)^*\circ(kf_-)^*=(kf_+)^*\circ(\alpha_1\oplus\gamma_1)^*.$$ 
By (\ref{eq 6.6}), (\ref{eq 2.11}) and (\ref{eq 2.14}),
\begin{displaymath}
\begin{split}
&\CS\big(k\nabla^{E_0}, \alpha_0^*(k\nabla^{F_0})\big)-\CS\big(k\nabla^{E_1}, \alpha_1^*(k\nabla^{F_1})\big)\\
&\equiv\CS\big(k\nabla^{E_0}, \alpha_0^*(k\nabla^{F_0})\big)+\CS\big(k\nabla^{W_0}, \gamma_0^*(k\nabla^{W_0})\big)-\CS\big(k\nabla^{E_1}, \alpha_1^*(k\nabla^{F_1})\big)\\
&\qquad-\CS\big(k\nabla^{W_1}, \gamma_1^*(k\nabla^{W_1})\big)\\
&\equiv\CS\big(k\nabla^{E_0}\oplus k\nabla^{W_0}, (\alpha_0\oplus\gamma_0)^*(k\nabla^{F_0}\oplus k\nabla^{W_0})\big)\\
&\qquad-\CS\big(k\nabla^{E_1}\oplus k\nabla^{W_1}, (\alpha_1\oplus\gamma_1)^*(k\nabla^{F_1}\oplus k\nabla^{W_1})\big)\\
&\equiv\CS\big(k\nabla^{E_0}\oplus k\nabla^{W_0}, (\alpha_0\oplus\gamma_0)^*(k\nabla^{F_0}\oplus k\nabla^{W_0})\big)\\
&\qquad-\CS\big((kf_+)^*(k\nabla^{E_1}\oplus k\nabla^{W_1}), (kf_+)^*(\alpha_1\oplus\gamma_1)^*(k\nabla^{F_1}\oplus k\nabla^{W_1})\big)\\
&\equiv\CS\big(k\nabla^{E_0}\oplus k\nabla^{W_0}, (\alpha_0\oplus\gamma_0)^*(k\nabla^{F_0}\oplus k\nabla^{W_0})\big)\\
&\qquad-\CS\big((kf_+)^*(k\nabla^{E_1}\oplus k\nabla^{W_1}), (\alpha_0\oplus\gamma_0)^*(kf_-)^*(k\nabla^{F_1}\oplus k\nabla^{W_1})\big).
\end{split}
\end{displaymath}
By (\ref{eq 2.14}), (\ref{eq 2.10}), (\ref{eq 2.13}) and (\ref{eq 2.11}),
\begin{displaymath}
\begin{split}
&\CS\big(k\nabla^{E_0}\oplus k\nabla^{W_0}, (\alpha_0\oplus\gamma_0)^*(k\nabla^{F_0}\oplus k\nabla^{W_0})\big)\\
&\qquad-\CS\big((kf_+)^*(k\nabla^{E_1}\oplus k\nabla^{W_1}), (\alpha_0\oplus\gamma_0)^*(kf_-)^*(k\nabla^{F_1}\oplus k\nabla^{W_1})\big)\\
&\equiv\CS\big(k\nabla^{E_0}\oplus k\nabla^{W_0}, (kf_+)^*(k\nabla^{E_1}\oplus k\nabla^{W_1})\big)\\
&\qquad+\CS\big((kf_+)^*(k\nabla^{E_1}\oplus k\nabla^{W_1}), (\alpha_0\oplus\gamma_0)^*(k\nabla^{F_0}\oplus k\nabla^{W_0})\big)\\
&\qquad-\CS\big((kf_+)^*(k\nabla^{E_1}\oplus k\nabla^{W_1}), (\alpha_0\oplus\gamma_0)^*(k\nabla^{F_0}\oplus k\nabla^{W_0})\big)\\
&\qquad-\CS\big((\alpha_0\oplus\gamma_0)^*(k\nabla^{F_0}\oplus k\nabla^{W_0}), (\alpha_0\oplus\gamma_0)^*(kf_-)^*(k\nabla^{F_1}\oplus k\nabla^{W_1})\big)\\
&\equiv\CS\big(k\nabla^{E_0}\oplus k\nabla^{W_0}, (kf_+)^*(k\nabla^{E_1}\oplus k\nabla^{W_1})\big)\\
&\qquad-\CS\big(k\nabla^{F_0}\oplus k\nabla^{W_0}, (kf_-)^*(k\nabla^{F_1}\oplus k\nabla^{W_1})\big)\\
&\equiv\CS\big(k\nabla^{V_0}\oplus k\nabla^{\wh{W}_0}, (kf)^*(k\nabla^{V_1}\oplus k\nabla^{\wh{W}_1})\big)\\
&\equiv k\CS\big(\nabla^{V_0}\oplus\nabla^{\wh{W}_0}, f^*(\nabla^{V_1}\oplus\nabla^{\wh{W}_1})\big).
\end{split}
\end{displaymath}
By dividing both sides of the above by $k$, (\ref{eq 6.5}) holds.

Let $[\EE_0, \FF_0, \alpha_0], [\EE_1, \FF_1, \alpha_1]\in K^{-1}(X; \Z/k\Z)$. By (\ref{eq 2.11}),
\begin{displaymath}
\begin{split}
T([\EE_0, \FF_0, \alpha_0])+T([\EE_1, \FF_1, \alpha_1])&=\big(\VV_0, \CCS_k(\EE_0, \FF_0, \alpha_0)\big)+\big(\VV_1, \CCS_k(\EE_1, \FF_1, \alpha_1)\big)\\
&=\big(\VV_0\oplus\VV_1, \CCS_k(\EE_0, \FF_0, \alpha_0)+\CCS_k(\EE_1, \FF_1, \alpha_1)\big)\\
&=\big(\VV_0\oplus\VV_1, \CCS_k(\EE_0\oplus\EE_1, \FF_0\oplus\FF_1, \alpha_0\oplus\alpha_1)\big)\\
&=T([\EE_0\oplus\EE_1, \FF_0\oplus\FF_1, \alpha_0\oplus\alpha_1]).
\end{split}
\end{displaymath}
Thus $T$ is a group homomorphism.
\end{proof}

Note that if $(\EE_0, \FF_0, \alpha_0)$ and $(\EE_1, \FF_1, \alpha_1)$ are two $k$-cocycles over $X$ that are equal in $K^{-1}(X; \Z/k\Z)$, then (\ref{eq 6.5}) shows that 
$$\CCS_k(\EE_0, \FF_0, \alpha_0)-\CCS_k(\EE_1, \FF_1, \alpha_1)\not\equiv 0$$ 
in general. Thus the Cheeger--Chern--Simons form $\CCS_k$ does not descend to a well defined map $K^{-1}(X; \Z/k\Z)\to\frac{\Omega^{\odd}(X)}{\im(d)}$.

\subsection{Analytic index and RRG-type formula}\label{s 6.2}

Let $\pi:X\to B$ be a submersion with closed, oriented and spin$^c$ fibers of even dimension, equipped with a Riemannian and differential spin$^c$ structure $\bpi$. Let $(\EE, \FF, \alpha)$ be a $k$-cocycle over $X$. Denote by $\DD^{S\otimes E}$ and $\DD^{S\otimes F}$ the twisted spin$^c$ Dirac operators associated to $(\EE, \bpi)$ and $(\FF, \bpi)$, respectively. Let $(H_E, s_E)$ and $(H_F, s_F)$ satisfy the AS property for $\DD^{S\otimes E}$ and $\DD^{S\otimes F}$, and let $\HH_E$ and $\HH_F$ be geometric bundles defining $\KER\big(\DD^{S\otimes E; \Delta_E}\big)$ and $\KER\big(\DD^{S\otimes F; \Delta_F}\big)$, where $\Delta_E$ and $\Delta_F$ are odd self-adjoint maps associated to $(H_E, s_E)$ and $(H_F, s_F)$ as in (\ref{eq 3.3}), respectively.

By Proposition \ref{prop 5.5}, $(kH_E, ks_E)$ and $(kH_F, ks_F)$ satisfy the AS property for $\DD^{S\otimes kE}$ and $\DD^{S\otimes kF}$, respectively. Let
\begin{equation}\label{eq 6.7}
\begin{split}
\wh{h}_E&:\KER\big(\DD^{S\otimes kE; k\Delta_E}\big)\to k\KER\big(\DD^{S\otimes E; \Delta_E}\big),\\
\wh{h}_F&:\KER\big(\DD^{S\otimes kF; k\Delta_F}\big)\to k\KER\big(\DD^{S\otimes F; \Delta_F}\big)
\end{split}
\end{equation}
be the uniquely determined $\Z_2$-graded morphisms (\ref{eq 5.49}).

Define a geometric bundle $\EE_{k, \alpha}$ over $X$ by 
$$\EE_{k, \alpha}=(kE, kg^E, \alpha^*(k\nabla^F)).$$ 
Denote by $\DD^{S\otimes kE}_\alpha$ the twisted spin$^c$ Dirac operator associated to $(\EE_{k, \alpha}, \bpi)$. By applying Proposition \ref{prop 5.3} to $\alpha:\EE_{k\alpha}\to k\FF$, the pair $(kH_F, s_\alpha)$, where the map $s_\alpha:kH_F\to\pi_*(kE)^-$ is defined by $s_\alpha:=\wt{\alpha}^{-1}_-\circ ks_F$ (\ref{eq 5.47}), satisfies the AS property for $\DD^{S\otimes kE}_\alpha$. Let
$$\wh{\alpha}:\KER\big(\DD^{S\otimes kE; \Delta_\alpha}_\alpha\big)\to\KER\big(\DD^{S\otimes kF; k\Delta_F}\big)$$
be the uniquely determined $\Z_2$-graded morphism (\ref{eq 5.48}), where $\Delta_\alpha$ is the odd self-adjoint map associated to $(kH_F, s_\alpha)$ as in (\ref{eq 3.3}).

By applying Proposition \ref{prop 5.2} to $k\EE$ and $\EE_{k, \alpha}$, denote by
\begin{equation}\label{eq 6.8}
h_{0\alpha}:\KER\big(\DD^{S\otimes kE; k\Delta_E}\big)\oplus\wh{\VV}_0\to\KER\big(\DD^{S\otimes kE; \Delta_\alpha}_\alpha\big)\oplus\wh{\VV}_1
\end{equation}
a $\Z_2$-graded morphism, where $\VV_0$ and $\VV_1$ are geometric bundles over $B$.

By (\ref{eq 6.7}) and (\ref{eq 6.8}), define a $\Z_2$-graded morphism
\begin{equation}\label{eq 6.9}
h:k\KER\big(\DD^{S\otimes E; \Delta_E}\big)\oplus\wh{\VV}_0\to k\KER\big(\DD^{S\otimes F; \Delta_F}\big)\oplus\wh{\VV}_1
\end{equation}
by
$$h=(\wh{h}_F\oplus\id_{\wh{\VV}_1})\circ(\wh{\alpha}\oplus\id_{\wh{\VV}_1})\circ h_{0\alpha}\circ(\wh{h}_E^{-1}\oplus\id_{\wh{\VV}_0}).$$
That is, $h$ is defined to be the composition
\begin{equation}\label{eq 6.10}
\begin{tikzcd}
k\KER\big(\DD^{S\otimes E; \Delta_E}\big)\oplus\wh{\VV}_0 \arrow{d}[swap]{\wh{h}_E^{-1}\oplus\id_{\wh{\VV}_0}} \arrow{rr}{h} & & k\KER\big(\DD^{S\otimes F; \Delta_F}\big)\oplus\wh{\VV}_1 \\ \KER\big(\DD^{S\otimes kE; k\Delta_E}\big)\oplus\wh{\VV}_0 \arrow{r}[swap]{h_{0\alpha}} & \KER\big(\DD^{S\otimes kE; \Delta_\alpha}_\alpha\big)\oplus\wh{\VV}_1 \arrow{r}[swap]{\alpha\oplus\id_{\wh{\VV}_1}} & \KER\big(\DD^{S\otimes kF; k\Delta_F}\big)\oplus\wh{\VV}_1 \arrow{u}[swap]{\wh{h}_F\oplus\id_{\wh{\VV}_1}} & 
\end{tikzcd}
\end{equation}
Note that the $\Z_2$-graded morphism
\begin{equation}\label{eq 6.11}
h_{01}:\KER\big(\DD^{S\otimes kE; k\Delta_E}\big)\oplus\wh{\VV}_0\to\KER\big(\DD^{S\otimes kF; k\Delta_F}\big)\oplus\wh{\VV}_1,
\end{equation}
given by applying Proposition \ref{prop 5.4} to $\alpha:k\EE\to k\FF$, is equal to $(\wh{\alpha}\oplus\id_{\wh{\VV}_1})\circ h_{0\alpha}$.

Take $\HH_E$ and $\HH_F$ to be trivial geometric bundles. Also, when applying Proposition \ref{prop 5.2} to $k\EE$ and $\EE_{k, \alpha}$ to obtain the $\Z_2$-graded morphism $h_{0\alpha}$ (\ref{eq 6.8}), take $\bkk$ in the proof of Proposition \ref{prop 5.2} to be a trivial geometric bundle. By (2) of Remark \ref{remark 5.1}, $\VV_0=\CC^j$ and $\VV_1=\CC^\ell$ for some $j, \ell\in\N$. Thus the even and odd part of the $\Z_2$-graded morphism $h$ (\ref{eq 6.9}) are given by
\begin{equation}\label{eq 6.12}
h_\pm:k\KER\big(\DD^{S\otimes E; \Delta_E}\big)^\pm\oplus\CC^j\to k\KER\big(\DD^{S\otimes F; \Delta_F}\big)^\pm\oplus\CC^\ell.
\end{equation}
Let $a_\pm=\rk\big(\ker\big(\DD^{S\otimes E; \Delta}\big)^\pm\big)$ and $b_\pm=\rk\big(\ker\big(\DD^{S\otimes F; \Delta}\big)^\pm\big)$. By (\ref{eq 6.9}),
$$ka_++j=kb_++\ell\quad\textrm{ and }\quad ka_-+j=kb_-+\ell.$$
Thus
$$k(b_+-a_+)=j-\ell=k(b_--a_-),$$
which implies $b_+-a_+=b_--a_-$. Write $m$ for $b_+-a_+$. Without loss of generality, we assume $m\geq 0$. Then $j=km+\ell$, and therefore
\begin{equation}\label{eq 6.13}
\CC^j=k\CC^m\oplus\CC^\ell.
\end{equation}
By taking direct sum of both sides of (\ref{eq 6.12}) with the trivial geometric bundles $(k-1)\CC^\ell$, (\ref{eq 6.12}) becomes
\begin{center}
\begin{tikzcd}
k\KER\big(\DD^{S\otimes E; \Delta_E}\big)^\pm\oplus\CC^j\oplus(k-1)\CC^\ell \arrow{rr}{h_\pm\oplus\id_{(k-1)\CC^\ell}} & & k\KER\big(\DD^{S\otimes E; \Delta_E}\big)^\pm\oplus\CC^\ell\oplus(k-1)\CC^\ell.
\end{tikzcd}
\end{center}
By (\ref{eq 6.13}), $\CC^j\oplus(k-1)\CC^\ell=k\CC^m\oplus\CC^\ell\oplus(k-1)\CC^\ell=k\CC^{m+\ell}$. Thus the above morphisms become
\begin{equation}\label{eq 6.14}
\begin{tikzcd}
k\big(\KER\big(\DD^{S\otimes E; \Delta_E}\big)^\pm\oplus\CC^{m+\ell}\big) \arrow{rr}{h_\pm\oplus\id_{(k-1)\CC^\ell}} & & k\big(\KER\big(\DD^{S\otimes E; \Delta_E}\big)^\pm\oplus\CC^\ell\big).
\end{tikzcd}
\end{equation}
The case $m<0$ is similar.

By Lemma \ref{lemma 5.1},
$$(H_E\oplus\C^{m+\ell}, s_E\oplus 0)\quad\textrm{ and }\quad(H_F\oplus\C^\ell, s_F\oplus 0)$$
satisfy the AS property for $\DD^{S\otimes E}$ and $\DD^{S\otimes F}$, and $\HH_E\oplus\CC^{m+\ell}$ and $\HH_F\oplus\CC^\ell$ are trivial geometric bundles defining
\begin{equation}\label{eq 6.15}
\begin{split}
\KER\big(\DD^{S\otimes E; \wt{\Delta}_E}\big)&=\KER\big(\DD^{S\otimes E; \Delta_E}\big)\oplus\wh{\CC}^{m+\ell},\\
\KER\big(\DD^{S\otimes F; \wt{\Delta}_F}\big)&=\KER\big(\DD^{S\otimes F; \Delta_F}\big)\oplus\wh{\CC}^\ell,
\end{split}
\end{equation}
where $\wt{\Delta}_E$ and $\wt{\Delta}_F$ are odd self-adjoint maps associated to $(H_E\oplus\C^{m+\ell}, s_E\oplus 0)$ and $(H_F\oplus\C^\ell, s_F\oplus 0)$ as in (\ref{eq 3.3}), respectively. Write
\begin{equation}\label{eq 6.16}
\wt{h}:k\KER\big(\DD^{S\otimes E; \wt{\Delta}_E}\big)\to k\KER\big(\DD^{S\otimes F; \wt{\Delta}_F}\big)
\end{equation}
for the $\Z_2$-graded morphism $h\oplus\id_{(k-1)\wh{\CC}^\ell}$ given by (\ref{eq 6.14}). Thus
\begin{equation}\label{eq 6.17}
\ind^a_k(\EE, \FF, \alpha):=\big(\KER\big(\DD^{S\otimes E; \wt{\Delta}_E}\big), \KER\big(\DD^{S\otimes F; \wt{\Delta}_F}\big), \wt{h}\big)
\end{equation}
is a $\Z_2$-graded $k$-cocycle over $B$. We call $\ind^a_k(\EE, \FF, \alpha)$ defined by (\ref{eq 6.17}) the analytic index of $(\EE, \FF, \alpha)$ in odd $\Z/k\Z$ $K$-theory.

For the remainder of this paper, we adopt the following conventions, except in Remark \ref{remark 6.2}.
\begin{itemize}
  \item Adopt the notations used in defining $\ind^a_k(\EE, \FF, \alpha)$ (\ref{eq 6.17}).
  \item Whenever a pair $(H, s)$ satisfies the AS property for a twisted spin$^c$ Dirac operator, $\HH$ is always taken to be a trivial geometric bundle.
  \item Let $\EE$ be a $\Z_2$-graded geometric bundle over $X$. If there exist $k\in\N$ and a morphism $\alpha:k\EE^+\to k\EE^-$, then we take the geometric bundles $\EE$ and $\FF$ in the construction of $\ind^a_k$ (\ref{eq 6.17}) to be $\EE^+$ and $\EE^-$, respectively.
  \item Assume $m\geq 0$. The proofs in the case $m<0$ are similar.
  \item To shorten notations, for any geometric bundle $\EE$ over $X$, we write 
  		$$\bkk(\EE)=\big(\kk(\EE), g^{\kk(\EE)}, \nabla^{\kk(\EE)}\big)$$ 
		for the geometric kernel bundle $\KER\big(\DD^{S\otimes E, \Delta}\big)$ over $B$. Similarly, we write $\wt{\bkk}(\EE)$ for the geometric kernel bundle $\KER\big(\DD^{S\otimes E, \wt{\Delta}}\big)$ given by (\ref{eq 6.15}).
\end{itemize}

The following theorem is the RRG-type formula in odd $\Z/k\Z$ $K$-theory.
\begin{thm}\label{thm 6.1}
Let $\pi:X\to B$ be a submersion with closed, oriented and spin$^c$ fibers of even dimension, equipped with a Riemannian and differential spin$^c$ structure $\bpi$. Let $(\EE, \FF, \alpha)$ be a $k$-cocycle over $X$. Denote by $\DD^{S\otimes E}$ and $\DD^{S\otimes F}$ the twisted spin$^c$ Dirac operators associated to $(\EE, \bpi)$ and $(\FF, \bpi)$, respectively. Let $(H_E, s_E)$ and $(H_F, s_F)$ satisfy the AS property for $\DD^{S\otimes E}$ and $\DD^{S\otimes F}$, and let $\HH_E$ and $\HH_F$ be geometric bundles defining $\bkk(\EE)$ and $\bkk(\FF)$, respectively. Then
\begin{equation}\label{eq 6.18}
\begin{aligned}[b]
\CCS_k\big(\ind^a_k(\EE, \FF, \alpha)\big)&\equiv\int_{X/B}\todd(\nabla^{T^VX})\wedge\CCS_k(\EE, \FF, \alpha)\\
&\qquad+\wt{\eta}(\EE, \bpi, \HH_E, s_E)-\wt{\eta}(\FF, \bpi, \HH_F, s_F).
\end{aligned}
\end{equation}
\end{thm}
\begin{proof}
By Proposition \ref{prop 5.5}, the $\Z_2$-graded morphisms $\wh{h}_E$ and $\wh{h}_F$ (\ref{eq 6.7}) satisfy
\begin{equation}\label{eq 6.19}
\nabla^{\bkk(k\EE)}=\wh{h}_E^*(k\nabla^{\bkk(\EE)}),\quad\nabla^{\bkk(k\FF)}=\wh{h}_F^*(k\nabla^{\bkk(\FF)})
\end{equation}
and
\begin{equation}\label{eq 6.20}
\begin{split}
\wt{\eta}(k\EE, \bpi, k\HH_E, ks_E)&\equiv k\wt{\eta}(\EE, \bpi, \HH_E, s_E),\\
\wt{\eta}(k\FF, \bpi, k\HH_F, ks_F)&\equiv k\wt{\eta}(\FF, \bpi, \HH_F, s_F).
\end{split}
\end{equation}
By Proposition \ref{prop 5.4}, the $\Z_2$-graded morphism $h_{01}$ (\ref{eq 6.11}) satisfies
\begin{equation}\label{eq 6.21}
\begin{aligned}[b]
\wt{\eta}(k\FF, \bpi, k\HH_F, ks_F)-\wt{\eta}(k\EE, \bpi, k\HH_E, ks_E)&\equiv\int_{X/B}\todd(\nabla^{T^VX})\wedge\CS\big(k\nabla^E, \alpha^*(k\nabla^F)\big)\\
&\qquad-\CS\big(\nabla^{\bkk(k\EE)}\oplus\wh{d}^j, h_{01}^*(\nabla^{\bkk(k\FF)}\oplus\wh{d}^\ell)\big).
\end{aligned}
\end{equation}
By (\ref{eq 6.19}), (\ref{eq 6.9}), (\ref{eq 2.14}), (\ref{eq 2.12}), (\ref{eq 2.11}), (\ref{eq 6.15}) and (\ref{eq 6.16}),
\begin{equation}\label{eq 6.22}
\begin{aligned}[b]
&-\CS\big(\nabla^{\bkk(k\EE)}\oplus\wh{d}^j, h_{01}^*(\nabla^{\bkk(k\FF)}\oplus\wh{d}^\ell)\big)\\
&\equiv-\CS\big((\wh{h}_E\oplus\id_{\wh{\C}^j})^*(k\nabla^{\kk(\EE)}\oplus\wh{d}^j), h_{01}^*(\wh{h}_F\oplus\id_{\wh{\C}^\ell})^*(k\nabla^{\kk(\FF)}\oplus\wh{d}^\ell)\big)\\
&\equiv-\CS\big(k\nabla^{\kk(\EE)}\oplus\wh{d}^j, h^*(k\nabla^{\kk(\FF)}\oplus\wh{d}^\ell)\big)\\
&\equiv-\CS\big(k\nabla^{\kk(\EE)}\oplus\wh{d}^j, h^*(k\nabla^{\kk(\FF)}\oplus\wh{d}^\ell)\big)-\CS((k-1)\wh{d}^\ell, (k-1)\wh{d}^\ell)\\
&\equiv-\CS\big(k\nabla^{\kk(\EE)}\oplus\wh{d}^j\oplus(k-1)\wh{d}^\ell, h^*(k\nabla^{\kk(\FF)}\oplus\wh{d}^\ell)\oplus(k-1)\wh{d}^\ell)\big)\\
&\equiv-\CS\big(k\nabla^{\wt{\kk}(\EE)}, \wt{h}^*(k\nabla^{\wt{\kk}(\FF)})\big).
\end{aligned}
\end{equation}
By (\ref{eq 6.20}) and (\ref{eq 6.22}), (\ref{eq 6.21}) becomes
\begin{equation}\label{eq 6.23}
\begin{aligned}[b]
k\wt{\eta}(\FF, \bpi, \HH_F, s_F)-k\wt{\eta}(\EE, \bpi, \HH_E, s_E)&\equiv\int_{X/B}\todd(\nabla^{T^VX})\wedge\CS\big(k\nabla^E, \alpha^*(k\nabla^F)\big)\\
&\qquad-\CS\big(k\nabla^{\wt{\kk}(\EE)}, \wt{h}^*(k\nabla^{\wt{\kk}(\FF)})\big).
\end{aligned}
\end{equation}
Then (\ref{eq 6.18}) follows from dividing both sides of (\ref{eq 6.23}) by $k$.
\end{proof}

\subsection{Connections to $\R/\Z$ $K$-theory}\label{s 6.3}

\begin{coro}\label{coro 6.1}
Let $\pi:X\to B$ be a submersion with closed, oriented and spin$^c$ fibers of even dimension, equipped with a Riemannian and differential spin$^c$ structure $\bpi$, and let $\E=(\EE, \omega)$ be a $\Z_2$-graded generator of $K^{-1}(X; \R/\Z)$. Denote by $\DD^{S\otimes E^\pm}$ the twisted spin$^c$ Dirac operator associated to $(\EE^\pm, \bpi)$. Let $(H_\pm, s_\pm)$ satisfy the AS property for $\DD^{S\otimes E^\pm}$, and let $\HH_\pm$ be a geometric bundle defining $\bkk(\EE^\pm)$. Then the analytic index $\ind^a_{\R/\Z}(\E)$ in $\R/\Z$ $K$-theory is equal to
\begin{equation}\label{eq 6.24}
\begin{aligned}[b]
\ind^a_{\R/\Z}(\E)&=\bigg(\bkk(\EE^+)\oplus\bkk(\EE^-)^{\op}, \int_{X/B}\todd(\nabla^{T^VX})\wedge\omega\\
&\qquad+\wt{\eta}(\EE^+, \bpi, \HH_+, s_+)-\wt{\eta}(\EE^-, \bpi, \HH_-, s_-)\bigg).
\end{aligned}
\end{equation}
\end{coro}
\begin{proof}
First note that the right-hand side of (\ref{eq 6.24}) is a $\Z_2$-graded generator of $K^{-1}(B; \R/\Z)$.

Let $\LL_\pm$ be a $\Z_2$-graded geometric bundle over $B$ so that $L_\pm\to B$ satisfies the MF property for $\DD^{S\otimes E^\pm}$. By \cite[(4.34) and Theorem 4.5]{H23},
\begin{equation}\label{eq 6.25}
\begin{aligned}[b]
\ind^a_{\R/\Z}(\E)&=\bigg(\LL_+\oplus\LL_-^{\op}, \int_{X/B}\todd(\nabla^{T^VX})\wedge\omega+\wh{\eta}(\EE, \bpi, \LL_+\oplus\LL_-^{\op})\bigg)\\
&=\bigg(\LL_+\oplus\LL_-^{\op}, \int_{X/B}\todd(\nabla^{T^VX})\wedge\omega+\wh{\eta}(\EE^+, \bpi, \LL_+)-\wh{\eta}(\EE^-, \bpi, \LL_-)\bigg)\\
&=\bigg(\LL_+, \int_{X/B}\todd(\nabla^{T^VX})\wedge\omega+\wh{\eta}(\EE^+, \bpi, \LL_+)\bigg)-\bigg(\LL_-, \wh{\eta}(\EE^-, \bpi, \LL_-)\bigg).
\end{aligned}
\end{equation}
Define generators $\E^+$ and $\E^-_0$ of $\wh{K}(X)$ by $\E^+=(\EE^+, \omega)$ and $\E^-_0=(\EE^-, 0)$, respectively. By lifting (\ref{eq 6.25}) to $\wh{K}(B)$ and making use of Proposition \ref{prop 4.1}, (\ref{eq 6.25}) becomes
\begin{displaymath}
\begin{split}
&\ind^a_{\R/\Z}(\E)\\
&=\bigg(\LL_+, \int_{X/B}\todd(\nabla^{T^VX})\wedge\omega+\wh{\eta}(\EE^+, \bpi, \LL_+)\bigg)-\bigg(\LL_-, \wh{\eta}(\EE^-, \bpi, \LL_-)\bigg)\\
&=\ind^a_{\wh{K}}(\E^+)-\ind^a_{\wh{K}}(\E^-_0)\\
&=\wt{\ind}^a_{\wh{K}}(\E^+)-\wt{\ind}^a_{\wh{K}}(\E^-_0)\\
&=\bigg(\bkk(\EE^+), \int_{X/B}\todd(\nabla^{T^VX})\wedge\omega+\wt{\eta}(\EE^+, \bpi, \HH_+, s_+)\bigg)-\bigg(\bkk(\EE^-), \wt{\eta}(\EE^-, \bpi, \HH_-, s_-)\bigg)\\
&=\bigg(\bkk(\EE^+)\oplus\bkk(\EE^-)^{\op}, \int_{X/B}\todd(\nabla^{T^VX})\wedge\omega+\wt{\eta}(\EE^+, \bpi, \HH_+, s_+)-\wt{\eta}(\EE^-, \bpi, \HH_-, s_-)\bigg).
\end{split}
\end{displaymath}
Thus (\ref{eq 6.24}) holds.
\end{proof}

Recall that the $\R/\Q$ Chern character $\ch_{\R/\Q}:K^{-1}(X; \R/\Z)\to H^{\odd}(X; \R/\Q)$ \cite[Proposition 1]{L94} is defined as follows. Given a $\Z_2$-graded generator $\E=(\EE, \omega)$ of $K^{-1}(X; \R/\Z)$, let $k\in\N$ satisfy $kE^+\cong kE^-$ and $\alpha:k\EE^+\to k\EE^-$ a morphism. Define $\ch_{\R/\Q}(\E)$ to be the image of the differential form
\begin{equation}\label{eq 6.26}
\frac{1}{k}\CS\big(\alpha^*(k\nabla^{E, -}), k\nabla^{E, +}\big)+\omega
\end{equation}
under the group homomorphism $H^{\odd}(X; \R)\to H^{\odd}(X; \R/\Q)$ that is associated to the short exact sequence $0\to\Q\to\R\to\R/\Q\to 0$. Note that $\ch_{\R/\Q}(\E)$ is independent of the choices made \cite[Lemma 1]{L94} (see also \cite[Remark 3]{H20}).

The following proposition says the RRG-type formula in odd $\Z/k\Z$ $K$-theory (Theorem \ref{thm 6.1}) is a refinement of the RRG theorem in $\R/\Z$ $K$-theory (\ref{eq 1.3}) at the differential form level.
\begin{prop}\label{prop 6.2}
Let $\pi:X\to B$ be a submersion with closed, oriented and spin$^c$ fibers of even dimension, equipped with a Riemannian and differential spin$^c$ structure $\bpi$, and let $\E=(\EE, \omega)$ be a $\Z_2$-graded generator of $K^{-1}(X; \R/\Z)$. Fix a $k\in\N$ satisfying $k\EE^+\cong k\EE^-$ and a morphism $\alpha:k\EE^+\to k\EE^-$. The RRG-type formula in odd $\Z/k\Z$ $K$-theory (Theorem \ref{thm 6.1}) for the $k$-cocycle $(\EE^+, \EE^-, \alpha)$, given by
\begin{equation}\label{eq 6.27}
\begin{split}
&\CCS_k\big(\ind^a_k(\EE^+, \EE^-, \alpha)\big)+\wt{\eta}(\EE^+, \bpi, \HH_+, s_+)-\wt{\eta}(\EE^-, \bpi, \HH_-, s_-)\\
&\qquad-\int_{X/B}\todd(\nabla^{T^VX})\wedge\CCS_k(\EE^+, \EE^-, \alpha)\equiv 0,
\end{split}
\end{equation}
is a refinement of the RRG theorem in $\R/\Z$ $K$-theory for $\E$, given by
\begin{equation}\label{eq 6.28}
\ch_{\R/\Q}(\ind^a_{\R/\Z}(\E))-\int_{X/B}\todd(T^VX)\cup\ch_{\R/\Q}(\E)=0
\end{equation}
in $H^{\odd}(B; \R/\Q)$, at the differential form level, i.e. the left-hand side of (\ref{eq 6.27}) is a differential form representative of that of (\ref{eq 6.28}).
\end{prop}
\begin{proof}
Denote by $\DD^{S\otimes E^\pm}$ the twisted spin$^c$ Dirac operator associated to $(\EE^\pm, \bpi)$. Let $(H_\pm, s_\pm)$ satisfy the AS property for $\DD^{S\otimes E^\pm}$ and $\HH_\pm$ a geometric bundle defining $\bkk(\EE^\pm)$. Consider the $k$-cocycle $(\EE^+, \EE^-, \alpha)$ over $X$. By (\ref{eq 6.15}),
\begin{equation}\label{eq 6.29}
\begin{split}
\wt{\bkk}(\EE^+)&=\bkk(\EE^+)\oplus\wh{\CC}^{m+\ell},\\
\wt{\bkk}(\EE^+)&=\bkk(\EE^-)\oplus\wh{\CC}^\ell
\end{split}
\end{equation}
for some $\ell\in\N$. Let $\wt{h}:k\wt{\bkk}(\EE^+)\to k\wt{\bkk}(\EE^-)$ be the $\Z_2$-graded morphism given by (\ref{eq 6.16}). Since $(\wh{\CC}^{m+\ell}, 0)$ and $(\wh{\CC}^\ell, 0)$ represent the zero element in $K^{-1}(B; \R/\Z)$, it follows from (\ref{eq 6.24}) and (\ref{eq 6.29}) that
\begin{equation}\label{eq 6.30}
\begin{aligned}[b]
\ind^a_{\R/\Z}(\E)&=\bigg(\bkk(\EE^+)\oplus\bkk(\EE^-)^{\op}, \int_{X/B}\todd(\nabla^{T^VX})\wedge\omega+\wt{\eta}(\EE^+, \bpi, \HH_+, s_+)\\
&\qquad-\wt{\eta}(\EE^-, \bpi, \HH_-, s_-)\bigg)+(\wh{\CC}^{m+\ell}, 0)+(\wh{\CC}^\ell, 0)\\
&=\bigg(\wt{\bkk}(\EE^+)\oplus\wt{\bkk}(\EE^-)^{\op}, \int_{X/B}\todd(\nabla^{T^VX})\wedge\omega+\wt{\eta}(\EE^+, \bpi, \HH_+, s_+)\\
&\qquad-\wt{\eta}(\EE^-, \bpi, \HH_-, s_-)\bigg).
\end{aligned}
\end{equation}
The RRG-type formula in odd $\Z/k\Z$ $K$-theory (Theorem \ref{thm 6.1}) for $(\EE^+, \EE^-, \alpha)$ is given by
\begin{equation}\label{eq 6.31}
\begin{aligned}[b]
&-\CCS_k\big(\ind^a_k(\EE^+, \EE^-, \alpha)\big)+\wt{\eta}(\EE^+, \bpi, \HH_+, s_+)-\wt{\eta}(\EE^-, \bpi, \HH_-, s_-)\\
&\qquad+\int_{X/B}\todd(\nabla^{T^VX})\wedge\CCS_k(\EE^+, \EE^-, \alpha)\equiv 0.
\end{aligned}
\end{equation}
Define a morphism $\wh{h}:k\wt{\bkk}(\EE^+)^+\oplus k\wt{\bkk}(\EE^-)^-\to k\wt{\bkk}(\EE^+)^-\oplus k\wt{\bkk}(\EE^-)^+$ by $\wh{h}=\wt{h}_+\oplus\wt{h}_-^{-1}$. By (\ref{eq 2.13}), (\ref{eq 2.9}) and (\ref{eq 2.11}),
\begin{equation}\label{eq 6.32}
\begin{aligned}[b]
&k\CCS_k\big(\ind^a_k(\EE^+, \EE^-, \alpha)\big)\\
&\equiv\CS\big(k\nabla^{\wt{\kk}(\EE^+)}, \wt{h}^*(k\nabla^{\wt{\kk}(\EE^-)})\big)\\
&\equiv\CS\big(k\nabla^{\wt{\kk}(\EE^+), +}, \wt{h}_+^*(k\nabla^{\wt{\kk}(\EE^-), +})\big)-\CS\big(k\nabla^{\wt{\kk}(\EE^+), -}, \wt{h}_-^*(k\nabla^{\wt{\kk}(\EE^-), -})\big)\\
&\equiv-\CS\big(\wt{h}_+^*(k\nabla^{\wt{\kk}(\EE^-), +}), k\nabla^{\wt{\kk}(\EE^+), +}\big)-\CS\big(k\nabla^{\wt{\kk}(\EE^+), -}, \wt{h}_-^*(k\nabla^{\wt{\kk}(\EE^-), -})\big)\\
&\equiv-\CS\big(\wt{h}_+^*(k\nabla^{\wt{\kk}(\EE^-), +}), k\nabla^{\wt{\kk}(\EE^+), +}\big)-\CS\big((\wt{h}_-^{-1})^*(k\nabla^{\wt{\kk}(\EE^+), -}), k\nabla^{\wt{\kk}(\EE^-), -}\big)\\
&\equiv-\CS\big(\wh{h}^*(k\nabla^{\wt{\kk}(\EE^-), +}\oplus k\nabla^{\wt{\kk}(\EE^+), -}), k\nabla^{\wt{\kk}(\EE^+), +}\oplus k\nabla^{\wt{\kk}(\EE^-), -}\big).
\end{aligned}
\end{equation}
By (\ref{eq 2.9}) and (\ref{eq 6.32}), (\ref{eq 6.31}) is equivalent to
\begin{equation}\label{eq 6.33}
\begin{aligned}[b]
&\frac{1}{k}\CS\big(\wh{h}^*(k\nabla^{\wt{\kk}(\EE^-), +}\oplus k\nabla^{\wt{\kk}(\EE^+), -}), k\nabla^{\wt{\kk}(\EE^+), +}\oplus k\nabla^{\wt{\kk}(\EE^-), -}\big)+\wt{\eta}(\EE^+, \bpi, \HH_+, s_+)\\
&\qquad-\wt{\eta}(\EE^-, \bpi, \HH_-, s_-)-\int_{X/B}\todd(\nabla^{T^VX})\wedge\bigg(\frac{1}{k}\CS\big(\alpha^*(k\nabla^{E, -}), k\nabla^{E, +}\big)\bigg)\equiv 0.
\end{aligned}
\end{equation}
By adding and subtracting $\int_{X/B}\todd(\nabla^{T^VX})\wedge\omega$ to the left-hand side of (\ref{eq 6.33}), it is equivalent to
\begin{displaymath}
\begin{split}
&\frac{1}{k}\CS\big(\wh{h}^*(k\nabla^{\wt{\kk}(\EE^-), +}\oplus k\nabla^{\wt{\kk}(\EE^+), -}), k\nabla^{\wt{\kk}(\EE^+), +}\oplus k\nabla^{\wt{\kk}(\EE^-), -}\big)\\
&\quad+\int_{X/B}\todd(\nabla^{T^VX})\wedge\omega+\wt{\eta}(\EE^+, \bpi, \HH_+, s_+)-\wt{\eta}(\EE^-, \bpi, \HH_-, s_-)\\
&\qquad-\int_{X/B}\todd(\nabla^{T^VX})\wedge\bigg(\frac{1}{k}\CS\big(\alpha^*(k\nabla^{E, -}), k\nabla^{E, +}\big)+\omega\bigg)\equiv 0.
\end{split}
\end{displaymath}
By (\ref{eq 6.26}) and (\ref{eq 6.30}),
\begin{displaymath}
\begin{split}
&\frac{1}{k}\CS\big(\wh{h}^*(k\nabla^{\wt{\kk}(\EE^-), +}\oplus k\nabla^{\wt{\kk}(\EE^+), -}), k\nabla^{\wt{\kk}(\EE^+), +}\oplus k\nabla^{\wt{\kk}(\EE^-), -}\big)\\
&\quad+\int_{X/B}\todd(\nabla^{T^VX})\wedge\omega+\wt{\eta}(\EE^+, \bpi, \HH_+, s_+)-\wt{\eta}(\EE^-, \bpi, \HH_-, s_-)
\end{split}
\end{displaymath}
is a differential form representative of $\ch_{\R/\Q}(\ind^a_{\R/\Z}(\E))$. Similarly,
$$\int_{X/B}\todd(\nabla^{T^VX})\wedge\bigg(\frac{1}{k}\CS\big(\alpha^*(k\nabla^{E, -}), k\nabla^{E, +}\big)+\omega\bigg)$$
is a differential form representative of $\int_{X/B}\todd(T^VX)\cup\ch_{\R/\Q}(\E)$. Thus the left-hand side of (\ref{eq 6.33}), and hence the left-hand side of (\ref{eq 6.31}), is a differential form representative of that of (\ref{eq 6.28}).
\end{proof}

It is instructive to compare (\ref{eq 6.33}) with \cite[(4.40)]{H23}.

The following proposition says the analytic index $\ind^a_k$ in odd $\Z/k\Z$ $K$-theory (\ref{eq 6.17}) refines the underlying geometric bundle of the analytic index $\ind^a_{\R/\Z}$ in $\R/\Z$ $K$-theory.
\begin{prop}\label{prop 6.3}
Let $\pi:X\to B$ be a submersion with closed, oriented and spin$^c$ fibers of even dimension, equipped with a Riemannian and differential spin$^c$ structure $\bpi$. For any $k$-cocycle $(\EE, \FF, \alpha)$ over $X$,
\begin{equation}\label{eq 6.34}
T\big(\ind^a_k(\EE, \FF, \alpha)\big)=\ind^a_{\R/\Z}\big(T(\EE, \FF, \alpha)\big),
\end{equation}
where $T:K^{-1}(B; \Z/k\Z)\to K^{-1}(B; \R/\Z)$ is the group homomorphism given by Proposition \ref{prop 6.1}.
\end{prop}
\begin{proof}
Let $(\EE, \FF, \alpha)$ be a $k$-cocycle over $X$. Denote by $\DD^{S\otimes E}$ and $\DD^{S\otimes F}$ the twisted spin$^c$ Dirac operators associated to $(\EE, \bpi)$ and $(\FF, \bpi)$, respectively. Let $(H_E, s_E)$ and $(H_F, s_F)$ satisfy the AS property for $\DD^{S\otimes E}$ and $\DD^{S\otimes F}$, and let $\HH_E$ and $\HH_F$ be geometric bundles defining $\bkk(\EE)$ and $\bkk(\FF)$, respectively. By (\ref{eq 6.17}) and (\ref{eq 6.3}),
\begin{equation}\label{eq 6.35}
\begin{aligned}[b]
T\big(\ind^a_k(\EE, \FF, \alpha)\big)&=T\big(\wt{\bkk}(\EE), \wt{\bkk}(\FF), \wt{h}\big)\\
&=\bigg(\wt{\bkk}(\EE)\oplus\wt{\bkk}(\FF)^{\op}, \CCS_k\big(\wt{\bkk}(\EE), \wt{\bkk}(\FF), \wt{h}\big)\bigg)\\
&=\bigg(\wt{\bkk}(\EE)\oplus\wt{\bkk}(\FF)^{\op}, \CCS_k\big(\ind^a_k(\EE, \FF, \alpha)\big)\bigg).
\end{aligned}
\end{equation}
On the other hand, by (\ref{eq 6.2}) and (\ref{eq 6.30}),
\begin{equation}\label{eq 6.36}
\begin{aligned}[b]
\ind^a_{\R/\Z}\big(T(\EE, \FF, \alpha)\big)&=\bigg(\wt{\bkk}(\EE)\oplus\wt{\bkk}(\FF)^{\op}, \int_{X/B}\todd(\nabla^{T^VX})\wedge\CCS_k(\EE, \FF, \alpha)\\
&\qquad+\wt{\eta}(\EE, \bpi, \HH_E, s_E)-\wt{\eta}(\FF, \bpi, \HH_F, s_F)\bigg).
\end{aligned}
\end{equation}
By the RRG-type formula in odd $\Z/k\Z$ $K$-theory (Theorem \ref{thm 6.1}), the right-hand side of (\ref{eq 6.35}) is equal to that of (\ref{eq 6.36}). Thus (\ref{eq 6.34}) holds.
\end{proof}
\begin{remark}\label{remark 6.2}
In this remark, we explain a difficulty in proving the analytic index $\ind^a_k$ in odd $\Z/k\Z$ $K$-theory induces a well defined group homomorphism.

Let $\pi:X\to B$ be a submersion with closed, oriented and spin$^c$ fibers of even dimension, equipped with a Riemannian and differential spin$^c$ structure $\bpi$, and let $(\EE, \FF, \alpha)$ be a $k$-cocycle over $X$. The $\Z_2$-graded morphism $\wt{h}$ (\ref{eq 6.16}) in the definition of the analytic index $\ind^a_k(\EE, \FF, \alpha)$ in odd $\Z/k\Z$ $K$-theory (\ref{eq 6.17}) is defined in terms of, among other things, the $\Z_2$-graded morphism
$$h_{0\alpha}:\KER\big(\DD^{S\otimes kE; k\Delta_E}\big)\oplus\wh{\VV}_0\to\KER\big(\DD^{S\otimes kE; \Delta_\alpha}_\alpha\big)\oplus\wh{\VV}_1$$
(see (\ref{eq 6.10})), which is obtained by applying a special case (i.e. only one Riemannian and differential spin$^c$ structure $\bpi$ on $\pi:X\to B$) of Proposition \ref{prop 5.2} to the geometric bundles $k\EE$ and $\EE_{k, \alpha}:=\big(kE, kg^E, \alpha^*(k\nabla^F)\big)$. More precisely, $h_{0\alpha}$ is defined to be the composition 
$$h_{0\alpha}:=(P_{s_\alpha r_\alpha}^{-1}\oplus\id_{\wh{\HH}_0})\circ(P_{0\alpha}\oplus\id_{\wh{\HH}_0\oplus\wh{\HH}_1})\circ(P_{s_0r_0}\oplus\id_{\wh{\HH}_1}),$$
where $P_{s_\alpha r_\alpha}$, $P_{0\alpha}$ and $P_{s_0r_0}$ are given by (\ref{eq 5.43}), (\ref{eq 5.27}) and (\ref{eq 5.41}), respectively. Note that $P_{s_0r_0}$ and $P_{s_\alpha r_\alpha}$ depend on $P_{0\alpha}$.

While all the other $\Z_2$-graded morphisms in (\ref{eq 6.10}) are uniquely determined by the objects involved, the $\Z_2$-graded morphism $P_{0\alpha}$ involves auxiliary choices as shown in the proof of (the special case of) Proposition \ref{prop 5.2}. The first choice is the smooth curve $c_b:I\to\wt{B}$ satisfying $c_b(0)=(b, 0)$ and $c_b(1)=(b, 1)$ when defining the $\Z_2$-graded morphism (\ref{eq 2.5}). The second choice, which is more problematic, is the pair $(\kk, \wt{r})$ satisfying the AS property of $\DD^{S\otimes\e}$ associated to $(\bee, \wt{\bpi})$ in the proof of Proposition \ref{prop 5.2}. The reason for making this choice is as follows: given twisted spin$^c$ Dirac operators $\DD^{S\otimes kE}$ and $\DD^{S\otimes kE}_\alpha$ associated to $(k\EE, \bpi)$ and $(\EE_{k, \alpha}, \bpi)$, respectively, consider a smooth curve of twisted spin$^c$ Dirac operators
$$\DD^{S\otimes kE}_t:=\sum_{j=1}^nc(e_j)\nabla^{S(T^VX)\otimes kE}_t,$$
where $\nabla^{S(T^VX)\otimes kE}_t$ is the tensor product of $\nabla^{S(T^VX)}$ and the smooth curve of unitary connections $\nabla^{kE}_t$ joining $k\nabla^E$ and $\alpha^*(k\nabla^F)$ defined by (\ref{eq 2.6}). The pairs $(kH_E, ks_E)$ and $(kH_F, s_\alpha)$ satisfy the AS property for $\DD^{S\otimes kE}$ and $\DD^{S\otimes kE}_\alpha$ by Propositions \ref{prop 5.5} and \ref{prop 5.3}, respectively. However, in contrast to the proof of Proposition \ref{prop 5.1}, there does not appear to be a \emph{canonical} way to construct a pair $(H, s_t)$ satisfying the AS property for $\DD^{S\otimes kE}_t$ \emph{and} joining $(kH_E, ks_E)$ and $(kH_F, s_\alpha)$. This is the reason that, in the proof of Proposition \ref{prop 5.2}, we consider the twisted spin$^c$ Dirac operator $\DD^{S\otimes\e}$ associated to $(\bee, \wt{\bpi})$ whose restrictions to $\e|_{X\times\set{0}}$ and $\e|_{X\times\set{1}}$ are $\DD^{S\otimes kE}$ and $\DD^{S\otimes kE}_\alpha$, respectively, and choose a pair $(\kk, \wt{r})$ satisfying the AS property for $\DD^{S\otimes\e}$.

Now suppose $(\kk_0, \wt{r}_0)$ and $(\kk_\alpha, \wt{r}_1)$ satisfy the AS property for $\DD^{S\otimes\e}$, and let $\wt{\Phi}^0$ and $\wt{\Phi}^1$ be the odd self-adjoint maps associated to $(\kk_0, \wt{r}_0)$ and $(\kk_1, \wt{r}_1)$ as in (\ref{eq 3.3}), respectively. Consider the geometric kernel bundles $\KER\big(\DD^{S\otimes\e; \wt{\Phi}^0}\big)$ and $\KER\big(\DD^{S\otimes\e; \wt{\Phi}^1}\big)$ over $\wt{B}$. As in the proof of Proposition \ref{prop 5.2}, let
\begin{displaymath}
\begin{split}
P^0_{0\alpha}&:\KER\big(\DD^{S\otimes kE; \Phi^0_0}\big)\to\KER\big(\DD^{S\otimes kE; \Phi^0_1}_\alpha\big),\\
P^1_{0\alpha}&:\KER\big(\DD^{S\otimes kE; \Phi^1_0}\big)\to\KER\big(\DD^{S\otimes kE; \Phi^1_1}_\alpha\big)
\end{split}
\end{displaymath}
be the $\Z_2$-graded morphisms (\ref{eq 5.27}). Denote by $\wt{h}^0$ and $\wt{h}^1$ the resulting $\Z_2$-graded morphisms (\ref{eq 6.16}) in the definition of the analytic index $\ind^a_k(\EE, \FF, \alpha)$ in odd $\Z/k\Z$ $K$-theory. If one can show that the $\Z_2$-graded $k$-cocycles 
\begin{displaymath}
\begin{split}
&\big(\KER\big(\DD^{S\otimes E; \wt{\Delta}_E}\big), \KER\big(\DD^{S\otimes F; \wt{\Delta}_F}\big), \wt{h}^0\big),\\
&\big(\KER\big(\DD^{S\otimes E; \wt{\Delta}_E}\big), \KER\big(\DD^{S\otimes F; \wt{\Delta}_F}\big), \wt{h}^1\big)
\end{split}
\end{displaymath}
are equal in $K^{-1}(B; \Z/k\Z)$, i.e. they satisfy (\ref{eq 6.1}), then $\ind^a_k(\EE, \FF, \alpha)$ does not depend on the choices made. We hope to address this difficulty in future work.
\end{remark}

\bibliographystyle{amsplain}
\bibliography{MBib}

\providecommand{\bysame}{\leavevmode\hbox to3em{\hrulefill}\thinspace}
\providecommand{\MR}{\relax\ifhmode\unskip\space\fi MR }
\providecommand{\MRhref}[2]{%
  \href{http://www.ams.org/mathscinet-getitem?mr=#1}{#2}
}
\providecommand{\href}[2]{#2}
\begin{thebibliography}{10}

\bibitem{AB07}
Torsten Asselmeyer-Maluga and Carl~H. Brans, \emph{Exotic smoothness and
  physics}, World Scientific Publishing Co. Pte. Ltd., Hackensack, NJ, 2007.

\bibitem{APS75}
Michael~F. Atiyah, Vijay~K. Patodi, and Isadore~M. Singer, \emph{Spectral
  asymmetry and {R}iemannian geometry. {II}}, Math. Proc. Cambridge Philos.
  Soc. \textbf{78} (1975), no.~3, 405--432.

\bibitem{APS76}
\bysame, \emph{Spectral asymmetry and {R}iemannian geometry. {III}}, Math.
  Proc. Cambridge Philos. Soc. \textbf{79} (1976), no.~1, 71--99.

\bibitem{AS71}
Michael~F. Atiyah and Isadore~M. Singer, \emph{The index of elliptic operators.
  {IV}}, Ann. of Math. (2) \textbf{93} (1971), 119--138.

\bibitem{BGV}
Nicole Berline, Ezra Getzler, and Mich{\`e}le Vergne, \emph{Heat kernels and
  {D}irac operators}, Grundlehren Text Editions, Springer-Verlag, Berlin, 2004,
  Corrected reprint of the 1992 original.

\bibitem{B86}
Jean-Michel Bismut, \emph{The {A}tiyah-{S}inger index theorem for families of
  {D}irac operators: two heat equation proofs}, Invent. Math. \textbf{83}
  (1986), no.~1, 91--151.

\bibitem{B96}
\bysame, \emph{Local index theory, eta invariants and holomorphic torsion: a
  survey}, Surveys in differential geometry, {V}ol. {III} ({C}ambridge, {MA},
  1996), Int. Press, Boston, MA, 1998, pp.~1--76.

\bibitem{D69}
Eldon Dyer, \emph{Cohomology theories}, Mathematics Lecture Note Series, W. A.
  Benjamin, Inc., New York-Amsterdam, 1969.

\bibitem{FL10}
Daniel~S. Freed and John Lott, \emph{An index theorem in differential
  {$K$}-theory}, Geom. Topol. \textbf{14} (2010), no.~2, 903--966.

\bibitem{G93}
Ezra Getzler, \emph{The odd {C}hern character in cyclic homology and spectral
  flow}, Topology \textbf{32} (1993), no.~3, 489--507.

\bibitem{GL18}
Alexander Gorokhovsky and John Lott, \emph{A {H}ilbert bundle description of
  differential {$K$}-theory}, Adv. Math. \textbf{328} (2018), 661--712.

\bibitem{GHV}
Werner Greub, Stephen Halperin, and Ray Vanstone, \emph{Connections, curvature,
  and cohomology. {V}ol. {I}: {D}e {R}ham cohomology of manifolds and vector
  bundles}, Academic Press, New York-London, 1972, Pure and Applied
  Mathematics, Vol. 47.

\bibitem{H20}
Man-Ho Ho, \emph{Local index theory and the {R}iemann-{R}och-{G}rothendieck
  theorem for complex flat vector bundles}, J. Topol. Anal. \textbf{12} (2020),
  no.~4, 941--987.

\bibitem{H23}
\bysame, \emph{An extended variational formula for the {B}ismut-{C}heeger eta
  form and its applications}, New York J. Math. \textbf{29} (2023), 1496--1530.

\bibitem{K08}
Max Karoubi, \emph{{$K$}-theory}, Classics in Mathematics, Springer-Verlag,
  Berlin, 2008, An introduction, Reprint of the 1978 edition, With a new
  postface by the author and a list of errata.

\bibitem{LM89}
H.~Blaine Lawson, Jr. and Marie-Louise Michelsohn, \emph{Spin geometry},
  Princeton Mathematical Series, vol.~38, Princeton University Press,
  Princeton, NJ, 1989.

\bibitem{L94}
John Lott, \emph{{${\bf R}/{\bf Z}$} index theory}, Comm. Anal. Geom.
  \textbf{2} (1994), no.~2, 279--311.

\bibitem{MM07}
Xiaonan Ma and George Marinescu, \emph{Holomorphic {M}orse inequalities and
  {B}ergman kernels}, Progress in Mathematics, vol. 254, Birkh\"auser Verlag,
  Basel, 2007.

\bibitem{MF79}
Alexandr~S. Mi{\v s}{\v c}enko and Anatoli{\u\i}~T. Fomenko, \emph{The index of
  elliptic operators over {$C^{\ast} $}-algebras}, Izv. Akad. Nauk SSSR Ser.
  Mat. \textbf{43} (1979), no.~4, 831--859, 967.

\bibitem{D85}
Daniel Quillen, \emph{Superconnections and the {C}hern character}, Topology
  \textbf{24} (1985), no.~1, 89--95.

\end{thebibliography}
\end{document}